\documentclass[noinfoline]{imsart}
 
 \usepackage{color,hyperref}
    \catcode`\_=11\relax
    \newcommand\email[1]{\_email #1\q_nil}
    \def\_email#1@#2\q_nil{%
      \href{mailto:#1@#2}{{\emailfont #1\emailampersat #2}}
    }
    \newcommand\emailfont{\sffamily}
    \newcommand\emailampersat{{\color{blue}\small@}}
    \catcode`\_=8\relax  
\RequirePackage[OT1]{fontenc}
\RequirePackage{amsthm,amsmath}
\usepackage[square,numbers]{natbib}

\RequirePackage[colorlinks,citecolor=blue,urlcolor=blue]{hyperref}
\usepackage[dvipsnames]{xcolor}
\usepackage{verbatim}
\usepackage{xcolor}

\allowdisplaybreaks

\usepackage{fullpage}
\makeatletter
\newcommand{\crossout}[1]{%
 \begingroup
 \sbox\z@{#1}%
 \dimen\z@=\wd\z@
 \dimen\tw@=\ht\z@
 \dimen\z@=.99626\dimen\z@  % get big points
 \dimen\tw@=.99626\dimen\tw@ % get big points
 \edef\co@wd{\strip@pt\dimen\z@}% just the number
 \edef\co@ht{\strip@pt\dimen\tw@}% just the number
 \leavevmode
 \rlap{\pdfliteral{q 1 J 0.4 w 0 0 m \co@wd\space \co@ht\space l S Q}}%
 \rlap{\pdfliteral{q 1 J 0.4 w 0 \co@ht\space m \co@wd\space 0 l S Q}}%
 #1%
 \endgroup
}

\makeatother

\usepackage{mathtools}

\usepackage{enumitem, hyperref}
\makeatletter
\def\namedlabel#1#2{\begingroup
  #2%
  \def\@currentlabel{#2}%
  \phantomsection\label{#1}\endgroup
}
\makeatother
\DeclareMathAlphabet{\mymathbb}{U}{BOONDOX-ds}{m}{n}
\usepackage[normalem]{ulem}
\newcommand{\E}{\mathbb{E}}

\newcommand {\R}{\mathbb{R}}

\usepackage{subfig}
\usepackage{layout}
\usepackage{dsfont}
\usepackage[mathscr]{eucal}
\usepackage[toc,page]{appendix}
\usepackage{mathrsfs}
\usepackage{color}
\usepackage{pifont}
\usepackage{bm}
\usepackage{latexsym}
\usepackage{amsfonts}
\usepackage{amssymb}
\usepackage{epsfig}
\usepackage{graphicx}
\usepackage{multirow}

\newtheorem{theorem}{Theorem}
\newtheorem{proposition}{Proposition}
\newtheorem{corollary}{Corollary}
\newtheorem{lemma}{Lemma}

\newtheorem{remark}{Remark}

 \usepackage[stable]{footmisc}

\startlocaldefs
\numberwithin{equation}{section}
\theoremstyle{plain}
\endlocaldefs

\begin{document}

\begin{frontmatter}

\title{{\large Nonparametric Estimation for  SDE with Sparsely Sampled Paths:} \\ {\large an FDA Perspective\footnote{Research supported by a Swiss National Science Foundation (SNF) grant.}
}}

\runtitle{SDE with Sparsely Sampled Paths}

\begin{aug}
\author{\fnms{Neda} \snm{Mohammadi}
}, \author{\fnms{Leonardo V.} \snm{Santoro}
} \and
\author{\fnms{Victor M.} \snm{Panaretos}
}

\runauthor{Mohammadi, Santoro \& Panaretos}

\affiliation{Ecole Polytechnique F\'ed\'erale de Lausanne}

\address{Institut de Math\'ematiques\\
\'Ecole Polytechnique F\'ed\'erale de Lausanne}
\end{aug}
 \email{neda.mohammadijouzdani@epfl.ch},
  \email{leonardo.santoro@epfl.ch},
 \email{victor.panaretos@epfl.ch}

\begin{abstract} 
We consider the problem of nonparametric estimation of the drift and diffusion coefficients of a Stochastic Differential Equation (SDE), based on $n$ independent replicates $\left\{X_i(t)\::\: t\in [0,1]\right\}_{1 \leq i \leq n}$, observed sparsely and irregularly on the unit interval, and subject to additive noise corruption. By \emph{sparse} we intend to mean that the number of measurements per path can be arbitrary (as small as two), and can remain constant with respect to $n$.  We focus on time-inhomogeneous SDE of the form $dX(t) = \mu(t)X(t)^{\alpha}dt + \sigma(t)X(t)^{\beta}dB(t)$, where $\alpha \in \{0,1\}$ and $\beta \in \{0,1/2,1\}$, which includes prominent examples such as Brownian motion,  Ornstein-Uhlenbeck process, geometric Brownian motion and Brownian bridge. Our estimators are constructed by relating the local (drift/diffusion) parameters of the SDE to its global parameters (mean/covariance, and their derivatives) by means of an apparently novel Partial Differential Equation (PDE). This allows us to use methods inspired by functional data analysis, and pool information across the sparsely measured paths. The methodology we develop is fully non-parametric and avoids any functional form specification on the time-dependency of either the drift function or the diffusion function. We establish almost sure uniform asymptotic convergence rates of the proposed estimators as the number of observed paths $n$ grows to infinity. Our rates are non-asymptotic in the number of measurements per path, explicitly reflecting how different sampling frequency might affect the speed of convergence. Our framework suggests possible further fruitful interactions between FDA and SDE methods in problems with replication.

\end{abstract}



\end{frontmatter}

\tableofcontents

\newpage

\section{Introduction}
Stochastic Differential Equations (SDEs) and Diffusion Processes are fundamental in the study and modelling of a multitude of phenomena ranging in economics, finance, biology, medicine, engineering and physics, to mention but a few \citep{mao2007stochastic}. Most generally, any diffusion process can be defined  via an SDE of the form:
\begin{eqnarray}\label{eq : SDE most general form}
dX(t) = \mu\left(t,X(t) \right)dt + \sigma\left(t,X(t) \right)dB(t), \qquad {t \in [0,T]},
\end{eqnarray}
where $B(t)$ denotes a standard Brownian Motion and \eqref{eq : SDE most general form} is interpreted in the sense of It\^{o} integrals. Any diffusion thus involves two components: the (infinitesimal) conditional mean, i.e.\ the \textit{drift} $\mu$, and the (infinitesimal) conditional variance, i.e.\ the \textit{diffusion coefficient} $\sigma$. These functional parameters fully determine the probabilistic behaviour of the SDE's solutions (\citet{Oskendal03}).

Consequently, drift and diffusion estimation have long been a core statistical inference problem, widely investigated by many authors in both parametric and nonparametric frameworks. A particularly interesting aspect of this inference problem is the fact that diffusions are typically observed discretely, and different estimation regimes are possible, for instance depending on whether one has infill asymptotics, or time asymptotics. Our focus in this article will be on the case where the discrete measurements are potentially few, irregular and noise-corrupted, but where sample path replicates are available. While replication is increasingly being considered in the SDE literature, the sparse/noisy measurement setting appears to have so far remained effectively wide open. This is a setting that is routinely investigated in the context of functional data analysis ({\citet{Yao2005}, \citet{zhang2016sparse}, \citet{hall2006} and \citet{li2010}}), but under the restrictive smoothness assumptions, which exclude SDE. Our approach to making valid inferences for sparsely/noisely sampled SDE will blend tools from both SDE and FDA, and for this reason we begin with a short review before presenting our contributions.

\subsubsection*{Inference for SDE}

Much of the early theory and methodology for drift and diffusion estimation concerns the time-homogeneous case.
Early authors considered estimation from a \textit{single} sample solution, and developed an estimation framework building on kernel methods (both parametric and nonparametric). Paths were assumed to either be observed continuously 
(\citet{banon1978nonparametric},  \citet{geman1979common},  \citet{PhamDinh1981}, \citet{banon1981recursive}, among others)
% , { \citet{pokern_posterior_2013}}
or discretely
(\citet{florens1993estimating},  \citet{stanton1997nonparametric},  \citet{ait1998nonparametric},  \citet{jiang1997nonparametric},  \citet{bandi2003fully}), in the latter case investigating the asymptotic theory of the estimators as the sample frequency increases or/and the observation time span grows to infinity. 
Together with this standard asymptotic framework, authors typically also assumed \textit{stationarity} and \textit{ergodicity} of the underlying diffusion process, which allowed for considerable simplification and mathematical convenience.

SDEs with time-varying drift and diffusion have also been studied, although less extensively.
Still, many applications feature diffusion processes whose intrinsic nature
involves time evolving diffusion and drift. Time homogeneous approaches cannot simply be extended to this setting, and require a different methodological framework.
Clearly, the most flexible model one ideally would consider is free of any restriction on the structure of the bivariate drift and diffusion, as in \eqref{eq : SDE most general form}. However, such an unconstrained model is known to be {severely unidentifiable} (in the statistical sense), i.e. not allowing estimation of the drift/diffusion  (\citet{fan2003time}). Hence, when considering time-varying stochastic differential equations, some {semi-parametric} form for the drift and/or diffusion functions should be imposed. For instance, \citet{nguyen1982identification}, \citet{ho1986term}, \citet{hull1990pricing}, \citet{black1990one}, and later \citet{fan2003time} and \citet{koo2010semiparametric} made various semi-parametric assumptions to explicitly express the time-dependence of the diffusion. The semi-parametric forms they considered are encompassed by the following time-dependent model:
\begin{align}\label{eq : time dep fan}
    dX(t)  = (\alpha_0(t) + \alpha_1(t) X(t)^{\alpha_2}) dt + \beta_0(t)X(t)^{\beta_1} dB(t).
\end{align}
which  ``\textit{arises naturally from various considerations and encompasses most of the commonly used models as special cases}" (\citet{fan2003time}).
The cited papers developed nonparametric techniques based on kernel or sieve methods applied to a single sample path, establishing convergence of the proposed estimators as the distance between successive discrete observations goes to zero.

\subsubsection*{The FDA Perspective}
Functional Data Analysis (FDA, see \citet{hsing_theoretical_2015} or \citet{ramsay2008functional}) considers inferences on the law of a continuous time stochastic process $\{X(t )\;: \;t \in [0, T]\}$ on the basis of a collection of $n$ realisations of this stochastic process, $\{X_i(t)\}_{i=1}^n$. These are viewed as a sample of size $n$ of random elements in a separable Hilbert space of functions (usually $L^2([0, T],\R)$). This global view allows for nonparametric inferences (estimation, testing, regression, classification) to be made making use of the mean function and the (spectral decomposition of) covariance kernel of $X$. That the inferences can be nonparametric is due to the  availability of $n$ replications, sometimes referred to as a \emph{panel} or \textit{longitudinal} setting,
i.e.\ repeated measurements on a collection of individuals taken across time.
This setting encompasses a very expansive collection of applications (see \citet{rao2019conceptual}). When the processes $\{X_i(t)\}_{i=1}^n$ are only observed discretely, smoothing techniques are applied, and these usually come with $C^2$ assumptions on the process itself and/or its covariance, that a priori rule out processes of low regularity such as diffusion paths.

Some recent work on drift and diffusion estimation make important contact with the FDA framework, even if not always presented as such explicitly: in particular, \citet{ditlevsen2005mixed}, \citet{overgaard2005non}, \citet{picchini2010stochastic}, \citet{picchini2011practical}, \citet{comte2013nonparametric}, \citet{delattre2018parametric} and  \citet{dion2016bidimensional}  modeled diffusions as functional data with parametric approaches, more precisely as stochastic differential equations with non-linear mixed effects. These approaches are not fully functional, as they use a parametric specification of the drift and diffusion. By contrast, the main strength of FDA methods is that they can be implemented nonparametrically. In this sense, \citet{comte2020nonparametric} and \citet{marie2021nadaraya} seem to be the first to provide a fully nonparametric FDA analysis of replicated diffusions: the first by extending the least-squares projection method and the second by implementing a nonparametric Nadaraya-Watson estimator. However, these techniques come at the cost of assuming continuous (perfect) observation of the underlying sample paths. {And, there is no clear way to adapt them to discretely/noisely observed paths, since path recovery by means of smoothing fails due to the roughness of the paths.}

\subsection{Our Contributions}
In the present work we provide a way forward to the problem of carrying out nonparametric inference for replicated diffusions, whose paths are observed discretely, possibly sparsely, and with measurement error. We focus on the following class of {linear}, time dependent stochastic differential equations:
\begin{equation}
\label{eq:generalSDE}
dX(t) = \mu(t)X(t)^{\alpha}dt + \sigma(t)X(t)^{\beta}dB(t), \qquad
\alpha \in \{0,1\} \text{ and } \beta \in \{0,1/2,1\}.
\end{equation}
We introduce nonparametric estimators of the drift function $\mu(t)$ and diffusion function $\sigma(t)$ from $n$ i.i.d. copies $\{X_i(t)\}_{i=1}^n$ solving \eqref{eq:generalSDE}, each of which is observed at $r\geq 2$ random locations with additive noise contamination (see Equation \eqref{observations} for more details).  
Our approach is inspired by FDA methods, but suitably adapted to the SDE setting where sample paths are rough. Specifically,
\begin{enumerate}
\item { We first establish a family of systems of PDE relating the drift and diffusion coefficients of $X$ to the mean and covariance functions of $X$ (see Proposition \ref{prop:system2}, which may be of interest in its own right, as it links the infinitesimal local variation of the process at time $t$, to its global first and second order behaviour). By an averaging argument, the family is then combined into a single equation \eqref{system2averaged}, to be used for estimation. }

\item We subsequently nonparametrically estimate mean and covariance (and their derivatives) based on a modification of standard FDA methods: pooling empirical correlations from measured locations on different paths to construct a global nonparametric estimate of the covariance and its derivative e.g. \citet{Yao2005}). The modification accounts for the fact that, contrary to the usual FDA setting, the sample paths are nowhere differentiable, and the covariance possesses a singularity on the diagonal (\citet{jouzdani2021functional}).

\item We finally plug our mean/covariance estimators (and their derivatives) into  PDE system \eqref{system2averaged}, allowing us to translate the global/pooled information into local information on the drift and diffusion. In short, we obtain plug-in estimators for the drift and diffusion via the PDE system \eqref{system2averaged}.
\end{enumerate}

We establish the consistency and almost sure uniform convergence rates of our estimators as the number $n$ of replications diverges while the number of measurements per path $r$ is unrestricted, and can be constant or varying in $n$ (our rates reflect the effect of $r$ explicitly). See Theorem \ref{thm : mu, sigma}. 
To our knowledge, this is the first example in the SDE literature to consider methods and theory applicable to \textit{sparse} and \textit{noise corrupted} data. Our methods do not rely on stationarity, local stationarity or assuming a Gaussian law, and do not impose range restrictions on the drift/diffusion. Another appealing feature is  that they estimate drift and diffusion simultaneously, rather than serially. Our methods are computationally easy to implement, and their finite sample performance is illustrated by two small simulation studies (see Section \ref{sec:smulation}).

\section{Model and Measurement Scheme}

In the following we mainly focus on a particular class of diffusion processes, corresponding to the solutions of the following class of linear stochastic differential equations:
$$
dX(t) = \mu(t)X(t)dt + \sigma(t)dB(t),\quad t\in[0,1],
$$
or equivalently, in the integral form:
\begin{equation}\label{eq : our sde, integral form}
X(t) = X(0) + \int_0^t \mu(u)X(u) du + \int_0^t \sigma(u) dB(u),\quad t\in [0,1].
\end{equation} 
Nevertheless, the methodology we develop extends to a wider class of diffusions, satisfying any of the SDE encompassed by \eqref{eq:generalSDE}. We show in full detail how to extend our methods and results to the different cases in  \eqref{eq:generalSDE} in Section \ref{sec:Extension}.  
We choose to focus on the case of \eqref{eq : our sde, integral form} for the purpose of making the presentation of our general approach more readable. In the remainder of the paper, we will always assume that $\mu(\cdot)$ and $\sigma(\cdot)$ are continuously differentiable on the unit interval $[0,1]$.
\medskip

\noindent Let $\left\{X_i(t),\: t \in [0, 1 ]\right\}_{1 \leq i \leq n}$ be $n$ iid diffusions satisfying the stochastic differential equations:
\begin{equation}
\label{eq : sde}
dX_i(t) = \mu(t)X_i(t) dt + \sigma(t)dB_i(t), \qquad i = 1, . . . , n,\: t \in [0,1],
\end{equation}
 where $B_1, \dots , B_n $ and $X_1(0), \dots,X_n(0)$ are totally independent sequences of Brownian motions and random initial points\footnote{Note that the probabilistic framework introduced in this model above is well defined, as \eqref{eq : sde} fully determines the probabilistic behaviour of its solutions -- under minimal assumptions --  all finite-dimensional distributions of the SDE in Equation \eqref{eq : sde} are \textit{uniquely determined} by the drift and diffusion coefficients.}. 
We assume to have access to the $n$ random paths through the discrete/noisy observations 
\begin{eqnarray}\label{observations}
Y_i\left( T_{ij}\right)= X_i\left( T_{ij}\right) + U_{ij} \qquad i=1,\ldots n, \; j=1,\ldots r(n),
\end{eqnarray}
where:
\begin{itemize}
    \item[i.] $\left\{U_{ij}\right\}$ forms an i.i.d. array of centered measurement errors with unknown finite variance $\nu^2$.
    
    \item[ii.] $\left\{T_{ij}\right\}$ is the triangular array of the ordered random design points ($T_{ik} < T_{ij}$, for $1 \leq k < j \leq r(n)$, and $i=1,2,\ldots,n$), each drawn independently from a strictly positive density on $[0,1]$.
    
    \item[iii.] $\left\{r(n)\right\}$ is the sequence of grid sizes, determining the denseness of the sampling scheme. {The sequence of grid sizes satisfies $r(n)\geq 2$ ($r(n)=1$ cannot provide information about covariation) but is {otherwise}  unconstrained and \emph{need not} to diverge with $n$. }
    
     \item[iv.] $\left\{X_{i}\right\}$, $\{\tilde T_{ij}\}$ and $\left\{U_{ij}\right\}$ are totally independent across all indices $i$ and $j$, where $\tilde T_{ij}$ denote the \textit{unordered} design points.
\end{itemize}
Notice that requirements (ii) and (iii) are very mild, allowing a great variety of sampling schemes: from dense to sparse. See Figure \ref{fig:obs_regime} for a schematic illustration of the measurement scheme. For the sake of simplicity we will occasionally reduce the notations $Y_i\left( T_{ij}\right)$ and $X_i\left( T_{ij}\right)$ to $Y_{ij}$ and $X_{ij}$, respectively.

\begin{figure}[h!]
  \centering
 \includegraphics[width=1\columnwidth]{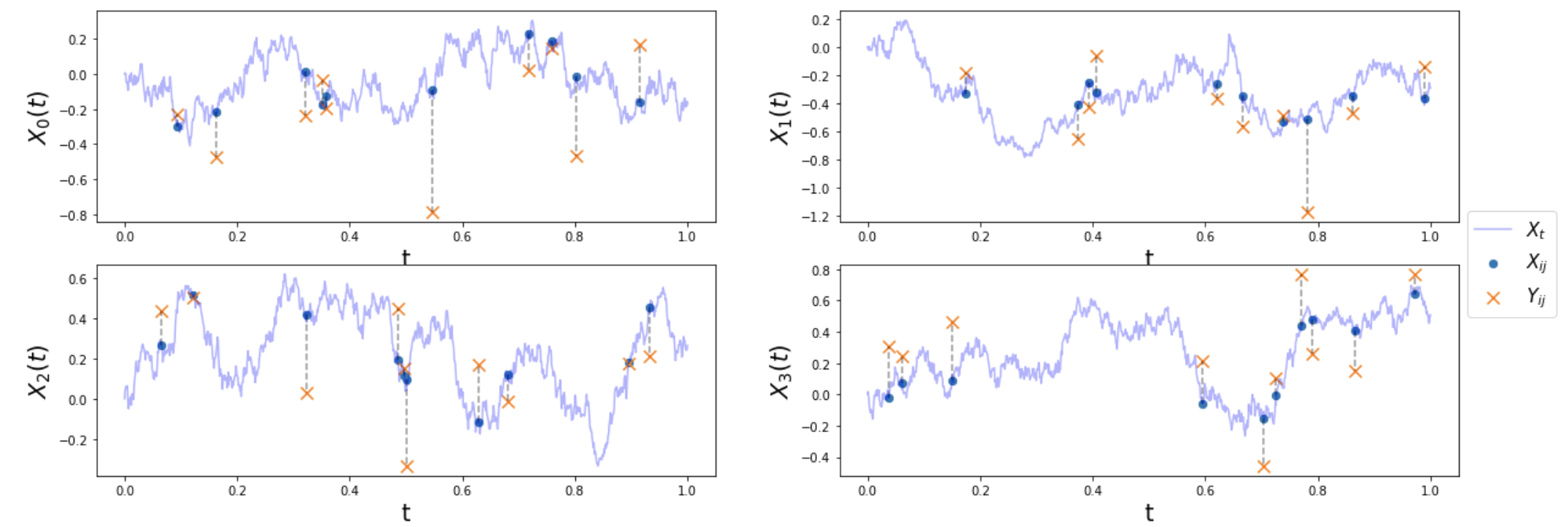}
  \caption{Illustration of the measurement scheme on the i.i.d diffusions $X_1, \dots, X_n$, with $n=4$ and $r=10$. The red crosses illustrate what is observable, whereas elements in blue are latent.}
  \label{fig:obs_regime}
\end{figure}

\section{Methodology and Theory} \label{sec:methodology and theory}
\subsection{Motivation and Strategy}

It is well known that the drift and diffusion functional coefficients in \eqref{eq : SDE most general form} determine the \textit{local} behaviour of the process. Indeed, they can be naturally interpreted as \textit{conditional mean} and \textit{variance} of the process:
$$
\mu(x,t) = \lim_{h\rightarrow 0}\frac{1}{h}\E[X(t+h)-x|X(t)=x],
\qquad\quad
\sigma^2(x,t) = \lim_{h \rightarrow 0}\frac{1}{h}\E[(X(t+h)-x)^2|X(t)=x].
$$
See \cite{fan2003time}, Equation (2).
For this reason, the standard approach for conducting inference on (any class of) SDE relies on estimating \textit{local} properties of the observed process. Clearly, this can only be done directly if one has access to continuous or dense observations.
The local perspective shared by existing methods does not adapt to sparsely observed paths, especially when there is noise.

Our approach is \emph{global} rather than \emph{local}, and unfolds over three steps. In Step (1), it connects the global features we can estimate from pooling to the local features we are targeting, by means of suitable PDE connecting the two.  In Step (2), it accumulates information by \textit{sharing information across replications by suitably pooling observations}. And, in Step (3) it plugs in suitable estimates of global features into the PDE of Step (1) to obtain plug-in estimators of the drift and diffusion coefficients. In more detail:  %\Leo{Victor: you mentioned some old reference! not YMW, older}. It can be divided in three steps.

\subsubsection*{Step 1 : Linking the global (FDA) and local (SDE) parametrisations}
{ We establish a bijective correspondence between the drift/diffusion coefficients and the mean/covariance functions (and their derivatives): see the \eqref{system2} and its averaged version \eqref{system2averaged}.
The fact that such a correspondence exists is not a priori obvious: though the Fokker-Plank equation (\citet{Oskendal03}) shows that drift and diffusion fully characterize the probability law of the SDE's solution, the same is not necessarily true for the mean and covariance, except in Gaussian diffusions -- but we have made no such restriction on the law. Our system effectively shows that \emph{global} information on the process (mean and covariance and their derivatives) can be translated into information on its \emph{local} features (drift and diffusion).  Consequently, we shift our focus on the estimation of the mean and covariance of the process (and their derivatives).}

\subsubsection*{Step 2 : Pooling data and estimating the covariance}

The idea of \textit{sharing information} is a well-developed and widely used approach in FDA, mostly for estimating the covariance structure from sparse and noise corrupted observations of iid data: see \citet{Yao2005}. One first notices that, given observations $Y_{ij} = X_i(T_{ij}) +U_{ij}$ as in \eqref{observations}, one can view products of \textit{off-diagonal} observations as proxies for the second-moment function of the latent process. Indeed:
$$
\E [ Y_{ij}Y_{ik} ] = \E [X(T_{ij})X(T_{ik})] + \delta_{j,k}\cdot\nu^2.
$$
Intuitively, one can thus recover the covariance by performing 2D scatterplot smoothing on the pooled product observations: see Figure \ref{fig_covest}.
\begin{figure}[h!]
  \centering
 \includegraphics[width=.8\columnwidth]{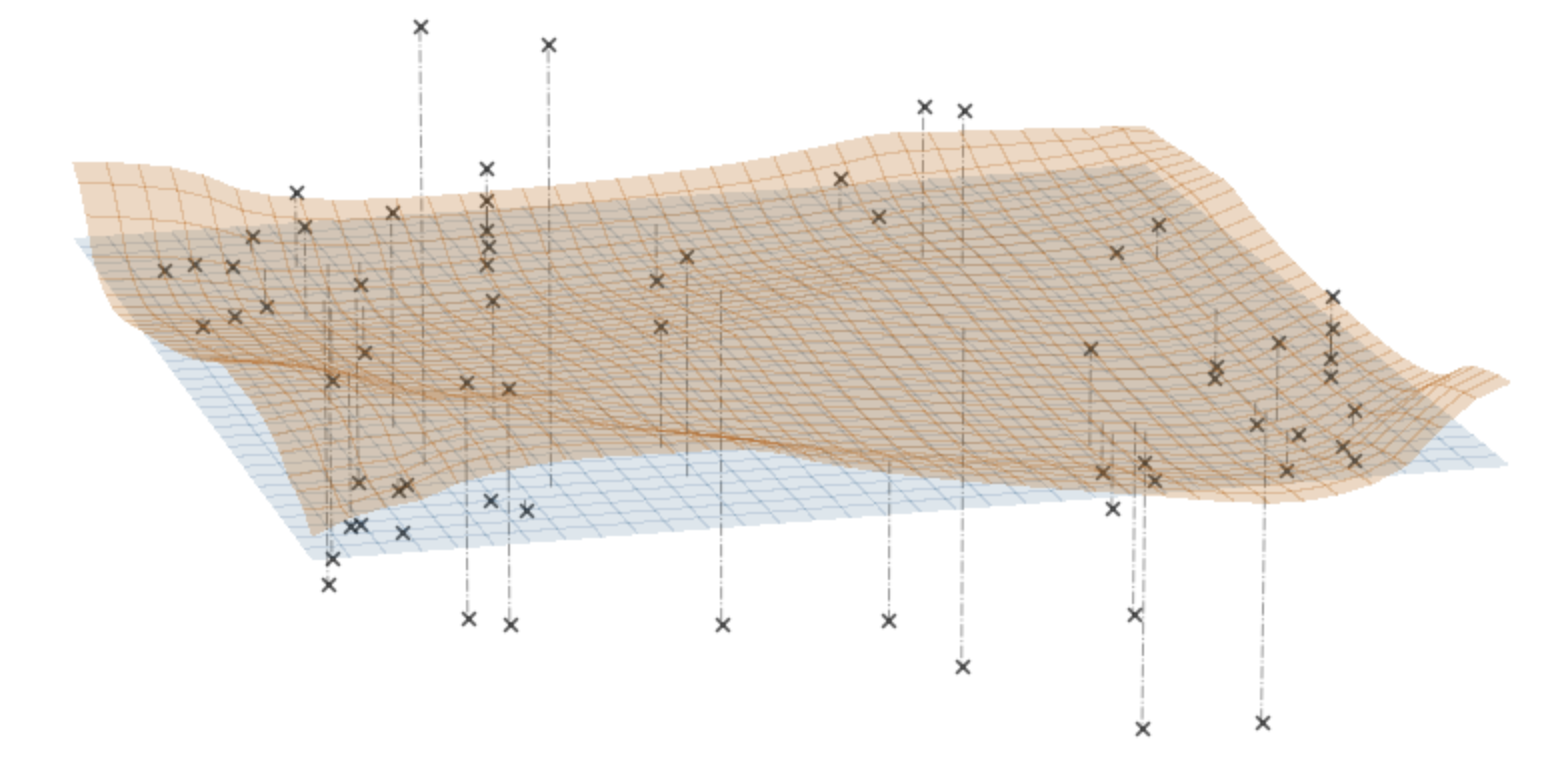}
  \caption{Estimated and true covariance surface -- in red and blue, respectively --  with 2D scatter plot of the off-diagonal product observations.}
  \label{fig_covest}
\end{figure}
In particular, if one is interested in estimating not only the latent covariance surface, but also its (partial) derivatives, a standard approach is to apply \textit{local linear smoothing}: see \citet{Fan1996}. Broadly speaking, the main idea of local linear smoothing is to fit local planes in two dimensions by means of weighted least squares. In particular, such an approach has the advantage of being fully non parametric.\\
However, the original approach due to  \citet{Yao2005} requires the covariance surface to be smooth to achieve consistent estimation, and  can not be straightforwardly applied in the setting of SDEs: indeed, in such case, the covariance possesses a singularity on the diagonal. 
A recent modification by \citet{jouzdani2021functional}, however, considers smoothing the covariance \textit{as a function defined on the closed lower} (upper) \textit{triangle}:
\begin{eqnarray}
\bigtriangleup : = \lbrace (t,s) \::\: 0\leq t\leq s \leq 1 \rbrace.
\end{eqnarray}
We will show that covariances corresponding to SDEs of the form \eqref{eq : our sde, integral form} are smooth on $\bigtriangleup$, provided some regularity of the drift and diffusion functionals: see condition \ref{cond:smoothness}. This, in turn, allows us to apply the methodology in \citet{jouzdani2021functional} to consistently estimate the covariance function from iid, sparse and noise corrupted samples, with uniform convergence rates. In fact, we will show how to extend such an estimation procedure to recover the partial derivatives of the covariance function as well: see Subsections \ref{subsec:The Estimators} and \ref{subsec:Asymptotics}.

\subsubsection*{Step 3 : Plug-in estimation}
 The method described in Step 2 allows us to consistently recover, with uniform convergence rates, the mean and covariance functions of the latent process, together with their derivatives.
Then, we may  finally  plug  our  mean/covariance  estimators  (and  their  derivatives)  into the \textit{averaged} PDE system \eqref{system2averaged}, allowing us to translate the global/pooled information into local information on the drift and diffusion. In short, we obtain plug-in estimators for the drift and diffusion via the PDE system: see Equation \eqref{estimate:system2} below.

\begin{remark}
In essence, we estimate drift and diffusion by plug-in estimators based on the mean and covariance of the process. Our estimators of the mean and the covariance \textit{are oblivious to the ordering of the time points $T_{ij}$}. In fact, we order them just for notational convenience, so that we may easily distinguish the pairs of times that lie in the upper (lower) triangle, by the corresponding index ordering. For instance, in the first sum appearing in \eqref{eq:local surface reg}, rather then summing over the pair of times $\{(\tilde T_{ij},\tilde T_{ik}) \::\: \tilde T_{ij}< \tilde T_{ik} \}$ we can simply sum over $j<k$.    
\end{remark}

\subsection{Mean, covariance, drift, diffusion: regularity and PDEs}

As outlined in the previous subsection, our methodology for estimating the drift $\mu(\cdot)$ and diffusion $\sigma(\cdot)$ functions builds on a family of systems of deterministic PDE -- derived from It\^{o}'s formula -- that relate the SDE parameters with the mean function $t \mapsto m(t) := \E X(t)$, the second moment function  $(t,s) \mapsto G(t,s) :=\E [X(t)X(s)]$,  and their (partial) derivatives. 
We will require the following assumption:
\begin{itemize}
    \item[\namedlabel{cond:smoothness}{\textbf{C(0)}}] 
    % The drift and diffusion functionals are $d$-times continuously differentiable on the unit interval i.e.   
    $\mu(\cdot), \sigma(\cdot)\in C^d([0,1],\mathbb{R})$ for some $d\geq 1$.
\end{itemize}

{Our starting point consists in the following system of ODEs, which is obtained by direct application of It\^{o}'s formula:
\begin{eqnarray}\label{system1}
\left\{ 
\begin{array}{ll}
   \partial m(t)= m(t)\mu(t), &   \\
   \sigma^2(t) = \partial D(t) -  2\mu(t) D(t),&
\end{array}
\qquad t\in [0,1],
\right.
\end{eqnarray}
{where $D(t) = G(t,t)$, $t \in [0,1]$. System \eqref {system1}  holds} whenever \ref{cond:smoothness} is satisfied. See for instance \citet{sarkka_applied_2019} or \citet{karatzas1991brownian}, where it is referred to as the ``linear system".
In spite of its appealing simplicity, the system \eqref{system1},
only involves marginal variances,  i.e. second order structure \textit{restricted to the diagonal segment}  $\left\{G(t,t) \; :\;0 \leq t \leq 1\right\}$. 

We seek to generalise the second equation appearing in \eqref{system1} by some deterministic PDE involving the \textit{entire} second order structure $\left\{G(t,s) \; :\;0 \leq t \leq s \leq 1\right\}$. Such a relation requires to consider the process at different points in time $s < t$, and hence cannot be obtained via direct application of It{\^o}'s formula. Nevertheless, we establish such result in  Proposition \ref{prop:system2}, by an application of stopping times to It{\^o}'s formula. Indeed, it would appear that Proposition \ref{prop:system2} bears some novelty in applying It{\^o}'s formula to functions with two dimensional time points, i.e. $\phi (x,t,s)$ with $s,t \in [0,1] , x \in \mathbb{R}$.

\begin{proposition}\label{prop:system2}
Assume \ref{cond:smoothness}. Then, $G(\cdot,\cdot )
\in \mathcal{C}^{d+1}(\bigtriangleup , \mathbb{R})$.  Moreover:
\begin{eqnarray}\label{eq:G_expansion}
G(t,s)&=& G(0,0) + 2\int_0^t G(u,u) \mu \left(u\right)du + \int_t^s G(t,u) \mu(u) du  + \int_0^t  \sigma^2 \left(u \right) du, \qquad  0 \leq t\leq s \leq 1.
\end{eqnarray}

In particular, the following family of PDE systems holds:
\begin{eqnarray} \label{system2}
\begin{cases}
\begin{array}{ll}
   \partial m(t)=  m(t)\mu(t),  \\
 \sigma^2(t)  =   \partial_t G(t,s)
 - \mu(t)G(t,t) - \int_t^s \mu \left(u \right) \partial_t G(t,u) du,
\end{array} \qquad  0 \leq t\leq s \leq 1.
\end{cases}
\end{eqnarray}
\end{proposition}

% \begin{remark}
Note that \eqref{eq:G_expansion} is a bona-fide generalisation of \eqref{system1}, as the latter may be obtained from the former by fixing $s=t$ and taking derivatives.
% \end{remark}

\medskip

{For every $s\in[0,1]$, \eqref{system2} defines a PDE   relating the diffusion $\sigma^2(\cdot)$ and the second order function $G(\cdot, \cdot)$ of the process.}
Having at our disposal a family of equations, we  wish to combine the information contained in each identity. Therefore, it is natural to consider the following equation:
 \begin{eqnarray} \label{system2averaged}
\begin{cases}
\begin{array}{ll}
   \partial m(t)= m(t)\mu(t), &   \\
    {\sigma}^2(t) = \frac{1}{1-t}\int_t^1\left(\partial_t G(t,s) - {\mu}(t) {G}(t,t) - \int_t^s {\mu} \left(u \right) \partial_t{G}(t,u) du \right)ds,&
\end{array}
\end{cases}
 \qquad  0 \leq t \leq 1.
\end{eqnarray}
which is obtained from \eqref{system2} by a simple averaging argument. That is, for every $t\in[0,1]$, we may average over  $s\geq t$ of the corresponding identity given by \eqref{system2}.

\medskip

A salient feature of \eqref{system2averaged} is the fact that it relates the intrinsically \textit{local} pointwise diffusion coefficient  -- interpretable as the infinitesimal local variation of the process at time $t$ -- to the \textit{global} behaviour of the process.
% , by appropriately integrating {the covariance kernel} of the process over its full domain. 

\medskip

 We will show that the PDE system in \eqref{system2averaged} yields consistent plug-in estimators for both the drift and the diffusion. In the simulations section (Section \ref{sec:smulation}) we  investigate the finite-sample behaviour of a plug-in estimator based on \eqref{system2averaged}, and compare its performance against an estimator based on the ``simple" system \eqref{system1}. We observe improved finite sample performance when using \eqref{system2averaged}.}

\subsection{Estimators}\label{subsec:The Estimators}

In the following we illustrate how we estimate  
$m$, $G$ and their derivatives by means of local polynomial smoothing (\citet{Fan1996}). 

The mean function  $m(\cdot)$  can be easily estimated by pooling all the observations: this is standard for sparsely observed functional data, and we refer the reader to \citet{li2010} and  \citet{Yao2005}. 

As for the covariance, we may pool the empirical second moments from all individuals on the closed lower (upper) \textit{triangle}:
\begin{eqnarray}
\bigtriangleup : = \lbrace (t,s) \::\: 0\leq t\leq s\leq 1 \rbrace.
\end{eqnarray}
and proceed by local polynomial smoothing.
This method is due to \citet{jouzdani2021functional}, and modifies the classical approach of \citet{Yao2005} to allow for processes with paths of low regularity, with a covariance function that is non-differentiable on the diagonal $\{(t,t) : t \in [0,1]\}$, but rather of the form: 
\begin{eqnarray}
\E (X(t)X(s)) = g(\min(s, t), \max(s, t)), 
\end{eqnarray}
for some smooth function $g$ that is at least of class {$C^{d+1}$} for some $d\geq 1$. Fortunately, under the assumption \ref{cond:smoothness}, this is satisfied for the SDE \eqref{eq : our sde, integral form} under consideration: see Proposition \ref{prop:system2}. 
{Consequently, one can employ local polynomial smoothing to obtain consistent estimators of the function and its (partial) derivatives, as required.} In detail, given any $0  \leq t \leq 1$ the pointwise local  polynomial smoothing  estimator of order $d$ for $m(t)$ and its derivative $\partial m(t)$ are obtained as 
\begin{equation} \label{eq:m,m'}
  \left(\widehat{m}(t) , \:h_m \cdot  \widehat{\partial m}(t) \right)^T = \left( (1,\underbrace{0,\ldots,0}_{d\text{ times}})^T \left( \widehat{\beta}_p (t) \right)_{0 \leq p \leq d}, (0,1,\underbrace{0,\ldots,0}_{d-1\text{ times}})^T \left( \widehat{\beta}_p (t)\right)_{0 \leq p \leq d}\right)^T
\end{equation}
where the vector $\left\{ \widehat{\beta}_p (t) \right\}_{0 \leq p \leq d}$ is the solution of  the following optimization problem 
\begin{equation}\label{eq:local linear reg}
  \underset{\left\{ \beta_p \right\}_{0 \leq p \leq d}}{\mathrm{argmin}}  \sum\limits_{i=1}^n \sum\limits_{j=1}^{r(n)}
    \left\{
        Y_{ij} - \sum_{p=0}^d \beta_p \left(T_{ij}-t\right)^p\right\}^2 K_{h^2_{m}}\left(T_{ij}-t\right),
\end{equation}
with $h_{m}$ a one--dimensional bandwidth parameter depending on $n$ and $K_{h^2_m} (\cdot) = h_{m}^{-1} K_{m} \left( h_{m}^{-1} \cdot  \right)$, for some univariate integrable kernel $K_{m}$.

Similarly, to estimate  $G(t,s)$ restricted to the lower triangle $\bigtriangleup = \{(t,s) \::\: 0\leq t\leq s\leq 1\}$ we apply a local {polynomial smoothing} regression method to the following 2D scatter plot:
\begin{equation} \label{eq:regression set}
    \lbrace \left( (T_{ik},T_{ij}),Y_{ij}Y_{ik}\right) \: :\: i = 1,...,n,\: k < j\rbrace.
\end{equation}
Note that  we excluded \textit{diagonal points} (squared observations with indices $j=k$) from the regression, as the measurement error in the observations causes such diagonal observations to introduce (asymptotically persistent) \textit{bias}, i.e. conditionally on the design array $\{T_{ij}\}$ we have:
\begin{eqnarray}\label{identifiability:issue}
\E \left[ Y_{ij}Y_{ik} \right] = G(T_{ik},T_{ij}) + \delta_{j,k}\cdot\nu^2.
\end{eqnarray}
In detail, for $t\leq s$, the {local polynomial smoothing} of order $d$ for $G(t,s)$ and its partial derivative $\partial_t G(t,s)$ are proposed to be
\begin{equation} \label{eq:G,partial G}
  \left(\widehat{G}(t,s) , \: h_G \cdot \widehat{ \partial_t G}(t,s) \right)^T = \left( (1,0,\ldots,0)^T \left( \widehat{\gamma}_{p,q}(t,s) \right))_{0 \leq p+q \leq d}, (0,1,0,\ldots,0)^T \left( \widehat{\gamma}_{p,q} (t,s)\right))_{0 \leq p+q \leq d}\right)^T
\end{equation}
where the vector $\left( \widehat{\gamma}_{0,0}(t,s),\widehat{\gamma}_{1,0}(t,s),\widehat{\gamma}_{0,1}(t,s), \ldots , \widehat{\gamma}_{d,0}(t,s),\widehat{\gamma}_{0,d} (t,s) \right)^T = \left( \widehat{\gamma}_{p,q}(t,s) \right)_{0 \leq p+q \leq d}^T$ is the minimizer  of the following problem:
\begin{align}\label{eq:local surface reg} 
    & \underset{\left( \gamma_{p,q} \right)}{\arg\min} \sum\limits_{i\leq n} \sum\limits_{k < j}
    \left\{
        Y_{ij}Y_{ik} -\sum_{0 \leq p+q \leq d}\gamma_{p,q}  \left(  T_{ij}-s\right)^p\left(  T_{ik}-t\right)^q
    \right\}^2
    K_{H_G}\left(T_{ij}-s,T_{ik}-t\right)\\ \nonumber
&=\underset{(\gamma_{p,q})}{\mathrm{argmin}}\sum_{i\leq n}\sum_{ k<j }
\left\{Y_{ik}Y_{ij}-\sum_{0 \leq p+q \leq d}\gamma_{p,q} h^{p+q}_G  \left( \frac{ T_{ij}-s}{h_G}\right)^p\left( \frac{ T_{ik}-t}{h_G}\right)^q\right\}^2 \\ \nonumber
& \hspace{6cm}\times K_{H_{G}}\left(\left( T_{ij}-s\right),\left( T_{ik}-t\right)\right).
\end{align} 
Here  $H_G^{1/2} = \operatorname{diag}(h_G,h_G)$ for some positive bandiwth parameter $h_G$, and $K_{H_G} (\cdot) =\vert H_G \vert^{-1/2} K_G\left( H_G^{-1/2} \times \cdot \right) $ for some bi-variate integrable kernel $K_G$.
A two-dimensional Taylor expansion at $(s,t)$ motivates estimating $G$ and its partial derivatives  as the vector minimizing \eqref{eq:local surface reg}. 

\medskip

{
Having defined our estimators for $m$, $G$ and their (partial) derivatives, we can combine this with Proposition \ref{prop:system2} to obtain the following simultaneous estimators $(\widehat{\mu},\widehat{\sigma}_T^2)$ for the drift and diffusion functions:

\begin{eqnarray} \label{estimate:system2}
\left\{
\begin{array}{ll}
    \widehat{\mu}(t) = \left( \widehat{ m}(t) \right)^{-1} \widehat{\partial m}(t) \mathbb{I}\left( \widehat{ m}(t) \neq 0\right),&  \\
    \widehat\sigma^2_T(t) = \frac{1}{1-t}\int_t^1\left(\widehat{\partial_t G}(t,s) - {\widehat\mu}(t) \widehat{G}(t,t) - \int_t^s {\widehat\mu} \left(u \right) \widehat{\partial_t G}(t,u) du \right)ds. & 
\end{array}
t \in [0,1] \right.
\end{eqnarray}

For the interest of comparison, we will also investigate the plug-in estimator based on the simpler ``linear system", i.e. \eqref{system1}:
\begin{eqnarray} \label{estimate:system1}
% \left\{
\begin{array}{ll}
    % \widehat{\mu}(t) = \left( \widehat{ m}(t) \right)^{-1} \widehat{\partial m}(t) \mathbb{I}\left( \widehat{ m}(t) \neq 0\right), &  \\
     \widehat{\sigma}^2_D(t) =  \widehat{\partial D}(t) - 2\widehat{\mu}(t) \widehat{D}(t),\;\;\; t \in [0,1] ,& 
\end{array}
% \right.
\end{eqnarray}
}

Note that both equations are well-defined as long as the mean function is bounded away from $0$. {For the cases considered in \eqref{eq:generalSDE}, this is easily avoidable and entails no serious restriction: 

\begin{itemize}
\item When $\alpha=1$ and $\beta\in \{0,1/2,1\}$, the mean function is available in closed form, $m(t) = m(0) \mathrm{exp}\left(\int_0^t \mu(v)dv \right)$, showing that the mean function is bounded away from zero,  $\inf_{t\in[0,1]} \vert m(t) \vert \neq 0$, if and only if  $m(0)\neq 0$.

\item When $\alpha=0$ and $\beta\in\{0,1/2,1\}$ the issue concerning drift estimation does not arise, as division by the mean function is not required to isolate it in the equation PDE: see Proposition \ref{prop:system2:alpha = 0}. 
However, when $\beta=1/2$, division by the mean function is required to estimate the diffusion, hence requiring a lower bound on its distance from zero. Still, the closed form representation of the mean $m(t) = m(0) + \int_0^t\mu(u)du$ shows that the condition is easily satisfied -- for instance when $m(0) \mu(\cdot)$ is always either positive or negative.
Similarly, when $\beta=1$, to isolate the diffusion in the PDE, we need to divide by the diagonal second-moment function $D(\cdot)$; this however, by Jensen's inequality and by the lower bound on the absolute value of the mean function, poses no issue.
\end{itemize}

\smallskip

In light of the above, the following assumption suffices to ensure that $m$ will be bounded away from zero, and the equations  \eqref{system2} and  \eqref{system1} are well-defined:
\begin{itemize}
    \item[\namedlabel{cond:meanBoundAway}{\textbf{C(1)}}] If $\alpha =1$, assume that $m(0)\neq 0$. If $\alpha =0$, assume that $\inf_{t\in[0,1]}\left \vert m(0) + \int_0^t\mu(s)ds\right \vert \neq 0 $.
\end{itemize}
}

\subsection{Asymptotic Theory}\label{subsec:Asymptotics}
In the present subsection we establish the consistency and convergence rate of the nonparametric estimators defined through \eqref{estimate:system2}. We will consider the limit as the number of replications diverges,  $n\rightarrow\infty$.  
Our rates are non-asymptotic in terms of $r(n)$ reflecting how different sampling frequency might affect the speed
of convergence.

Set  $K_{G}(\cdot,\cdot)  =W(\cdot)W(\cdot)$ and $K_{m}(\cdot) =  W(\cdot)$ for some appropriately chosen symmetric univariate  probability density kernel $W(\cdot)$, with finite $d$th moment.
 The kernel function $W(\cdot)$ depends, in general, on $n$. For the sake of simplicity we choose  $W(\cdot)$  to be in one of the forms $\left\{ \mathbf{W}_n \right\}$ or  $\left\{ \mathbb{W}_n \right\}$ below:
\begin{align*}
     \mathbf{W}_n(u) =  \mathrm{exp}\left( -\frac{u}{a_n} \right),
&&
\mathbb{W}_n (u)= 
\left\{
\begin{array}{ll}
\mathbf{W}_1 (u)    &\mathrm{if} \;\vert u \vert < 1,  \\
\mathbf{W}_n (u)   & \mathrm{if} \;\vert u \vert \geq 1,
\end{array}
\right.
\end{align*}
where  $\left\{ a_n \right\}_n$ is a sequences of positive numbers tending to zero sufficiently fast such that $W_n(1^+)  \lesssim O\left(h_G^{d+3}  \right)$. See Subsection 5.2 of \citet{jouzdani2021functional} for a   discussion around the choice of kernel function $W(\cdot)$.
%%%%%%%%########
 We additionally make the following assumptions: 
\begin{itemize}
    \item[\namedlabel{cond:design}{\textbf{C(2)}}] There exists some positive number $M$ for which we have {$0 < \mathbb{P}\left(T_{ij} \in [a,b] \right) \leq M (b-a)$}, for all $i,j$ and $0 \leq a < b \leq  1$. 
    
    \item[\namedlabel{cond:error}{\textbf{C(3)}}] $\mathbb{E}\vert U_{ij}      \vert^{\rho}< \infty$ and $\mathbb{E}\vert X(0)      \vert^{\rho}< \infty$.
\end{itemize}
Condition \ref{cond:design} concerns the sampling design, and asks that the domain be sampled minimally evenly. Condition \ref{cond:error} bounds the $\rho$-moment (for a $\rho$ to be specified as needed later) of the noise contaminant, as well as of the starting point (initial distribution) of the diffusion. The latter clearly holds whenever the initial point $X(0)$ follows a  degenerate  distribution i.e. $X(0) =  x $ a.s. for some $x$ in $\R$.

Remarkably, these two conditions alone are sufficient to prove consistency; 
any other regularity property needed to ensure convergence of the local polynomial regression method can be shown to be satisfied in the present setting as a consequence to the stochastic behaviour constrained by \eqref{eq : our sde, integral form}. Indeed, in light of Proposition  \ref{prop:system2}  we conclude  that $G$ satisfies the regularity conditions \textbf{C(2)}  of \citet{jouzdani2021functional}. Moreover,  Lemma \ref{lemma:sup} below justifies  condition \textbf{C(1)} of \citet{jouzdani2021functional}  (i.e. finiteness of $\underset{0\leq t \leq 1}{\mathrm{sup}} \mathbb{E} \vert X(t)\vert^{\rho}$)   as a result of finiteness of $\mathbb{E} \vert X(0)\vert^{\rho}$.
\begin{lemma}\label{lemma:sup} 
Let $\rho $  be a positive number,  then
 $\underset{0\leq t \leq 1}{\mathrm{sup}} \mathbb{E} \vert X(t)\vert^{\rho} < \infty $  holds if and only if   $\mathbb{E} \vert X(0)\vert^{\rho} < \infty$.
\end{lemma}
\noindent Now, let us define:
$$
\mathcal{R}(n) =  \left[ h^{-2}_{m} \frac{\mathrm{log}n}{n}\left(h^2_{m} + \frac{h_{m}}{r}  \right) \right]^{1/2}+ h^{d+1}_{m},
\quad \text{and} \quad
% and:
% $$
\mathcal{Q}(n) =   \left[h_{G}^{-4} \frac{\mathrm{log}n}{n}\left( h_G^4 + \frac{h_G^3}{r}+ \frac{h_G^2}{r^2} \right) \right]^{1/2}+h^{d+1}_{G}.
$$
We are now ready to present the asymptotic behaviour of the estimates proposed above.
 \begin{theorem}\label{thm : m, G}
 Assume conditions \ref{cond:smoothness}, \ref{cond:design} and \ref{cond:error} hold for $\rho > 2$, and let $\widehat{m}(\cdot)$ and $\widehat{\partial m}(\cdot)$ be the estimators defined in \eqref{eq:m,m'}. Then with probability $1$:
 \begin{align}\label{eq: cons m}
     \sup_{0\leq t \leq 1}\vert\widehat{m}(t) - m (t)\vert
 &=  O\left( \mathcal{R}(n) \right)\\ \label{eq: cons partial m}
 \sup_{0\leq t \leq 1}\vert\widehat{\partial m}(t) - \partial m (t)\vert
 &= h_m^{-1} O\left( \mathcal{R}(n) \right).
 \end{align}

\noindent If additionally \ref{cond:error} holds for $\rho > 4$,  then $\widehat{G}(\cdot,\cdot)$ and $\widehat{\partial_t G}(\cdot,\cdot)$ defined in \eqref{eq:G,partial G} satisfy with probability $1$:
 \begin{align}\label{eq: cons G}
     \underset{0 \leq t \leq s \leq 1}{\mathrm{sup}}\vert\widehat{G}(t,s) - G (t,s)\vert
 &=  O\left( \mathcal{Q}(n) \right)\\ \label{eq: cons partial G}
 \underset{0 \leq t \leq s \leq 1}{\mathrm{sup}}\vert\widehat{\partial_s G}(t,s) - \partial_t G (t,s)\vert
 &= h_G^{-1} O\left( \mathcal{Q}(n) \right), 
 \end{align}

 \end{theorem}

\begin{theorem}\label{thm : mu, sigma} 
Assume conditions \ref{cond:smoothness}, \ref{cond:meanBoundAway},  \ref{cond:design} and \ref{cond:error} hold for $\rho > 2$, in particular,  $m(0) \neq 0$. Let $\widehat{\mu}(\cdot)$ be the estimator defined in  \eqref{estimate:system2}. Then, with probability $1$:
\begin{eqnarray}\label{eq:cons:mu}
 \sup_{0\leq t \leq 1}\vert\widehat{\mu}(t) - \mu (t)\vert
 &=& h_m^{-1} O\left( \mathcal{R}(n) \right).
\end{eqnarray}

{If additionally \ref{cond:error} holds for $\rho > 4$, and  $\widehat\sigma(\cdot)$ is defined from either  \eqref{estimate:system2} or \eqref{estimate:system1}, then with probability $1$:}
\begin{eqnarray}\label{eq:cons:sigma}
 \sup_{0\leq t \leq 1}\vert\widehat{\sigma}^2(t) - \sigma^2 (t)\vert
 &=&   h_m^{-1}O\left( \mathcal{R}(n) \right)+ h_G^{-1}O\left( \mathcal{Q}(n) \right).
\end{eqnarray}

\end{theorem} 
\begin{corollary}\label{cor : dense case}
 Additionally to the conditions of Theorem \ref{thm : mu, sigma}, let us assume a dense sampling regime, i.e. $r(n) \geq M_n$ for some increasing sequence satisfying $M_n\uparrow\infty$. Furthermore, assume that the bandwidths both satisfy $$M_n^{-1} \lesssim h_{m} { \lesssim } \left( \frac{\mathrm{log}n}{n} \right)^{1/{(2(1+d))}}\qquad \& \qquad  M_n^{-1} \lesssim h_{G} { \lesssim } \left( \frac{\mathrm{log}n}{n} \right)^{1/{(2(1+d))}}.$$ 
 Then, with probability $1$:
 \begin{equation*}
  \begin{split}
   \sup_{0\leq t \leq 1}\vert\widehat{m}(t) - m (t)\vert
  &= O \left( \frac{\mathrm{log}n}{n} \right)^{1/2}\\
   \sup_{0\leq t \leq 1}\vert\widehat{\partial m}(t) - \partial m (t)\vert
  &= O \left( \frac{\mathrm{log}n}{n} \right)^{d/(2(1+d))}\\
   \sup_{0\leq t \leq 1}\vert\widehat{\mu}(t) - \mu (t)\vert
  &= O \left( \frac{\mathrm{log}n}{n} \right)^{d/(2(1+d))}
  \end{split}
\,\quad
  \begin{split}
    \underset{0 \leq t \leq s \leq 1}{\mathrm{sup}}\vert\widehat{G}(t,s) - G (t,s)\vert
  &=  O \left( \frac{\mathrm{log}n}{n} \right)^{1/2}\\
  \underset{0 \leq t \leq s \leq 1}{\mathrm{sup}}\vert\widehat{\partial_t G}(t,t) - \partial_t G (t,s)\vert
  &= O \left( \frac{\mathrm{log}n}{n} \right)^{d/(2(1+d))}\\
    \sup_{0\leq t \leq 1}\vert\widehat{\sigma}^2(t) - \sigma^2 (t)\vert
 &= O \left( \frac{\mathrm{log}n}{n} \right)^{d/(2(1+d))}
  \end{split}
\end{equation*}

 where $\widehat\sigma(\cdot)$ may be defined from either  \eqref{estimate:system2} or \eqref{estimate:system1}.
\end{corollary}

\begin{remark}
When the drift and diffusion functionals are (only) once continuously differentiable ( i.e.\  $d=1$ in condition \ref{cond:smoothness}),  the dense rate in Corollary \ref{cor : dense case} becomes: 
$$
  \sup_{0\leq t \leq 1}\vert\widehat{\mu}(t) - \mu (t)\vert
 = O\left( \frac{\mathrm{log}n}{n} \right)^{1/4} \qquad\&\qquad
 \sup_{0\leq t \leq 1}\vert\widehat{\sigma}^2(t) - \sigma^2 (t)\vert
 = O \left( \frac{\mathrm{log}n}{n} \right)^{1/4},
$$
where both equalities are valid with probability 1. On the other hand, when the drift and diffusion functions are infinitely many times continuously differentiable( i.e.\  $d=\infty$ in condition \ref{cond:smoothness}), the dense rate in Corollary \ref{cor : dense case} can approach arbitrarily close to the parametric rate. That is, for arbitrarily small positive $\varepsilon$ we can perform $d$  local polynomial smoothing (for any $d \geq d(\varepsilon) = \lfloor \frac{1}{2\epsilon} \rfloor$), yielding with probability 1:
$$
  \sup_{0\leq t \leq 1}\vert\widehat{\mu}(t) - \mu (t)\vert
 = O\left( \frac{\mathrm{log}n}{n} \right)^{1/2-\varepsilon}
 \qquad\&\qquad
 \sup_{0\leq t \leq 1}\vert\widehat{\sigma}^2(t) - \sigma^2 (t)\vert
 = O \left( \frac{\mathrm{log}n}{n} \right)^{1/2 - \epsilon}.
$$
\end{remark}

\noindent The proofs of the results presented in this subsection are established in the Appendix.

\begin{remark}
 \citet{hall2006} observe  that fixed and regular sampling strategy does not affect the consistency result in estimating the FDA functionals $m$ and $G$, provided the (per path) sample size $r(n)$  grows sufficiently fast. Examining the proof of {Theorem \ref{thm : mu, sigma}},  one may speculate an analogous conclusion for the estimates of the (partial) derivatives of the FDA functionals, and consequently for the drift and diffusion estimates. We also conjecture a similar consequence for the sparse sampling regime when the design points are not common among the paths but  shifted in time such that collectively there are  \textit{enough} observations in  $h_m$(or $h_G$)-neighborhoods of each point in the corresponding domain.
\end{remark}

{
\begin{remark}
    Our proof of the uniform convergence result for the diffusion estimator $\widehat\sigma(\cdot)$ in Theorem \ref{thm : mu, sigma} proceeds as follows. First, we show that each equation in the family, {indexed by $s$}, in \eqref{system2} yields uniformly consistent plug-in estimators, {over $ t \leq s$}. Then, we essentially complete our proof by dominated convergence, showing that the {integration over $s$} preserves the uniform consistency result, and hence proving the uniform convergence result for the \textit{averaged} estimator \eqref{system2averaged}. We remark that, from the perspective of asymptotic theory, such \textit{averaging} does not come at any benefit nor expense: each plug-in estimator based on some equation in \eqref{system2} benefits from the same rate of convergence, and this translates to the same rate for the estimator based on the averaged equation as well. In particular, this implies that, in our proof, we show that a plug-in estimator based on the simple linear equation \eqref{system1} manifests the same rate of convergence as does the plug-in estimator based on  the more elaborate averaged equation \eqref{system2averaged}. However, the expectation is that the \emph{finite-sample behaviour} of those estimators as manifested in practice will be different -- namely, we expect that an estimator based on \eqref{system2averaged} to perform better, as it somehow incorporates at least as much as does the simple linear system. This expectation is observed throughout our simulations \ref{sec:smulation}. Possibly such an effect may be hidden in the constants of the convergence rates.  
\end{remark}
}

\section{Extensions}\label{sec:Extension}

\subsection{An ``easier" problem: integrated volatility estimation}

In the previous sections we developed a consistent methodology for estimating the (drift and) diffusion functional \textit{pointwise}. In many applications, though, the primary interest may lie in estimating the \textit{integrated} diffusion instead, sometimes referred to as integrated volatility: \begin{equation}
    t\mapsto  \int_0^t \sigma^2(u)du,\quad 0 \leq t \leq 1.
\end{equation}
This quantity naturally arises for instance when one is interested in assessing the variance of stock returns for portfolio optimization or for quantifying their exposure to risk, during a certain time period.
Indeed, there are many instances in the literature where focus is fully dedicated to the problem of estimating the integrated volatility, assuming various stochastic models and under different observational setting: see for instance \citet{woerner2005estimation}, \citet{zhang2005tale} or \citet{jacod2014efficient}.

\medskip

The system of equations we derived in Section \ref{sec:methodology and theory} includes ``population" identities involving the integrated diffusion function.
For instance, from the second equation in \eqref{system1} we may obtain the following identity:
\begin{equation}\label{eq : int_volatility D}
\int_0^t \sigma^2(u)du = D(t) - D(0) - 2\int_0^t\partial m(u)\frac{D(u)}{m(u)} du, 
\qquad 0\leq t\leq 1.
\end{equation}
 In turn, if one has estimates $\widehat m(\cdot),\widehat {\partial m }(\cdot)$ and $\widehat D(\cdot)$  for $m(\cdot)$, $\partial m(\cdot)$ and $D(\cdot)$, respectively, then such an identity clearly leads to the definition of an estimator for the integrated diffusion by plug-in estimation. Simulation results suggest good asymptotic behaviour of such estimator, with (as might be hoped) faster convergence rates with compared to the pointwise estimator of the integrand: see Section \ref{sec:smulation}. However, since the derivative of the mean function appears in the identity \eqref{eq : int_volatility D} -- which is estimated at a \textit{slower rate} by our local polynomial smoothing approach -- it is not straightforward why {the integrated volatility }should grant better convergence rates. A simple heuristic argument shows that the rate of convergence might even be \textit{optimal}. Indeed, suppose that $\widehat{\partial  m} \approx \partial \widehat{m}$. Then:
\begin{align*}
    \Big\vert  D(t) - & D(0) - 2\int_0^t \left(\partial m(u)\frac{D(u)}{m(u)}\right)du -  \widehat D(t) - \widehat D(0) - 2\int_0^t \left( \widehat{\partial m}(u)\frac{\widehat D(u)}{\widehat m(u)}\right)du  \Big\vert  \\
  \approx \: &  \Big\vert  D(t) - D(0) - 2\int_0^t \left(\partial m(u)\frac{D(u)}{m(u)}\right)du -  \widehat D(t) - \widehat D(0) - 2\int_0^t \left( \partial\widehat{m}(u)\frac{\widehat D(u)}{\widehat m(u)}\right)du  \Big\vert  \\
    \leq  \: & 2   \sup_{0\leq u \leq 1} \vert \widehat{D}(u) - D (u)\vert
    + 2 \Big\lvert \int_0^t \partial m(u)\left(\frac{D(u)}{m(u)} - \frac{\widehat D(u)}{\widehat m(u)}\right)du \Big\vert 
     + 2 \Big\lvert \int_0^t \left(\partial m(u) - \partial \widehat{m}(u)\right)\frac{\widehat D(u)}{\widehat m(u)} du \Big\vert  \\
    \leq  \: & \mathcal{O}(1) \left( \sup_{0\leq u \leq 1} \vert \widehat{D}(u) - D   (u)\vert  +   \sup_{0\leq u \leq 1} \vert \widehat{m}(u) - m (u)\vert \right) 
\end{align*}
which grants optimal rates of convergence of the estimator for the integrated diffusion, provided the estimators $\widehat D(\cdot)$ and $\widehat m(\cdot)$ are uniformly asymptotically optimal.
We additionally remark that the (reasonable) condition $\widehat{\partial  m} \approx \partial \widehat{m}$ is in fact satisfied by a large class of non parametric estimators, such as spline smoothing estimators and general projection methods. 

\medskip
A similar line of reasoning can be applied to the following expression for the integrated diffusion:
\begin{equation}\label{eq : int_volatility T}
     \int_0^t  \sigma^2 \left(u \right) du = - G(0,0) - 2\int_0^t \partial m (u) \frac{G(u,u)}{m(u)} du + \frac{1}{1-t} \int_{t}^1 \left( G(t,s) 
     - \int_t^s  \partial m (u) \frac{G(t,u)}{m(u)} du \right) ds 
      , 
     \qquad  0 \leq t \leq 1,
\end{equation}
which is derived from \eqref{eq:G_expansion}. The subsequent heuristics are also very similar, as is the tentative statement regarding the optimality of the plug-in estimator defined implicitly via the equation.

\subsection{{Different Time-Inhomogeneous Models}}

The methodology developed for the specific case of  \eqref{eq : our sde, integral form} (i.e. a time-inhomogeneous Ornstein-Uhlenbeck process) can be  extended to the following wider class of models:
$$dX(t) = \mu(t)X(t)^{\alpha}dt + \sigma(t)X(t)^{\beta}dB(t), 
\quad \text{with} \quad
\alpha \in \{0,1\}, \: \beta \in \{0,1/2,1\}.$$

{ We now show precisely how to extend our method, with a detailed treatment of each case.}

\subsubsection*{Case I : Time-inhomogeneous Geometric Brownian Motion ($\alpha =1 $ and $\beta  =1 $):} We  first provide   counterparts of the  results above to the case  $\alpha =1 $ and $\beta  =1 $ i.e.
\begin{eqnarray}\label{eq:timedep:GBM}
dX(t) = \mu(t)X(t)dt + \sigma(t)X(t)dB(t), \qquad t \in [0,1]
\end{eqnarray}

\begin{proposition}\label{prop:system2:GBM}
Let the stochastic process $\{X(t)\}_{t \geq 0}$ satisfy \eqref{eq:timedep:GBM} and  assume \ref{cond:smoothness} holds. Then,  $m(\cdot) \in \mathcal{C}^{d+1}([0,1],\mathbb{R})$ and $G(\cdot,\cdot )
\in \mathcal{C}^{d+1}(\bigtriangleup , \mathbb{R})$.  Moreover:
\begin{eqnarray}\label{eq:G_expansion:GBM}
G(t,s)&=&  G\left( 0,0\right)+ 2\int_0^t D(u)  \mu \left(u\right)du  + \int_t^s G\left(t,u \right) \mu(u) du+ \int_0^t  \sigma^2 \left(u \right) D(u) du
\end{eqnarray}
In particular:
\begin{eqnarray}\label{eq: G decomp:GBM}
 \partial_t G(t,s) &=& \mu(t)G(t,t) + \int_t^s \mu \left(u \right)  \partial_t G(t,u) du + \sigma^2(t)G(t,t)
\end{eqnarray}
and the following system of PDEs holds:
\begin{eqnarray} \label{system2:GBM}
\begin{cases}
\begin{array}{ll}
   \partial m(t)= m(t)\mu(t), &   \\
    \partial_t G(t,s)=
 \mu(t)D(t) + \int_t^s \mu \left(u \right)  \partial_t G(t,u) du + \sigma^2(t)D(t),&
\end{array}
\end{cases}\qquad (t,s) \in \bigtriangleup.
\end{eqnarray}
\end{proposition}

\begin{lemma}\label{lemma:sup:GBM} 
Let the stochastic process $\{X(t)\}_{t \geq 0}$ be defined as \eqref{eq:timedep:GBM} and  $\rho $  be a positive number,  then condition 
 $\underset{0\leq t \leq 1}{\mathrm{sup}} \mathbb{E} \vert X(t)\vert^{\rho} < \infty $  holds if and only if   $\mathbb{E} \vert X(0)\vert^{\rho} < \infty$.
\end{lemma}

By Lemma \ref{lemma:sup:GBM}, we can obtain  analogous results for the case of \eqref{eq:timedep:GBM} as obtained in Theorems \ref{thm : m, G} and \ref{thm : mu, sigma}.
Indeed, local polynomial smoothing entails the same rates of convergence for the mean function $m(\cdot)$ (and its derivative) and for the second moment function $G(\cdot, \cdot)$ (and its partial derivatives) as stated in  Theorems \ref{thm : m, G}, provided \ref{cond:smoothness}, \ref{cond:design} and \ref{cond:error} hold with $\rho >2$ and $\rho > 4$, respectively.
 Moreover, defining $\widehat{\mu}$, $\widehat{\sigma}^2_D$ and $\widehat{\sigma}^2_T$ by plug-in estimation in terms of \eqref{system2:GBM}:
\begin{eqnarray*}
 \left\{
\begin{array}{ll}
    \widehat{\mu}(t) =  \widehat{\partial m}(t)(\widehat{ m}(t) )^{-1},&  \\
    \widehat\sigma^2_D(t) =\widehat{\partial D}(t)(\widehat{ D}(t))^{-1} {-2 \widehat{\mu}(t)},&  \\
    \widehat\sigma^2_T(t) =(\widehat{D}(t))^{-1} \frac{1}{1-t}\int_t^1\left(\widehat{\partial_t G}(t,s) - {\widehat\mu}(t) \widehat{D}(t) - {\int_t^s {\widehat\mu} \left(u \right) \widehat{\partial_t G}(t,u) du} \right)ds. & 
\end{array}
t \in [0,1] \right.
\end{eqnarray*}
 we can show that the same rates of convergence hold as stated in Theorem  \ref{thm : mu, sigma}, 
provided \ref{cond:smoothness}, \ref{cond:meanBoundAway}, \ref{cond:design} and \ref{cond:error} hold with $\rho >2$ and $\rho > 4$ for the drift and diffusion estimators, respectively.

\subsubsection*{Case II : Time-inhomogeneous Cox--Ingersoll--Ross model ($\alpha =1 $ and $\beta  =1/2 $):} 
Next we focus on the case  $\alpha = 1 $ and $\beta = 1/2$, i.e.:
\begin{eqnarray}\label{eq:timedep:CIR}
dX(t) = \mu(t)X(t)dt + \sigma(t)X^{1/2}(t)dB(t), \qquad  t \in [0,1].
\end{eqnarray}

\begin{proposition}\label{prop:system2:CIR}
Let the stochastic process $\{X(t)\}_{t \geq 0}$ satisfies \eqref{eq:timedep:CIR} and  assume \ref{cond:smoothness} holds. Then,  $m(\cdot) \in \mathcal{C}^{d+1}([0,1],\mathbb{R})$ and  $G(\cdot,\cdot )
\in \mathcal{C}^{d+1}(\bigtriangleup , \mathbb{R})$.  Moreover:
\begin{eqnarray}\label{eq:G_expansion:CIR}
G(t,s)&=& G\left( 0,0\right)+ 2\int_0^t D(u)  \mu \left(u\right)du  + \int_t^s G\left(t,u \right) \mu(u) du+ \int_0^t  \sigma^2 \left(u \right) m(u) du, 
\end{eqnarray}
In particular:
\begin{eqnarray}\label{eq: G decomp:CIR}
 \partial_t G(t,s)  &=&
  \mu(t) D(t)  + \int_t^s \mu \left(u \right) 
\partial_t G(t,u) du + \sigma^2(t) m(t),
\end{eqnarray}
and the following system of PDEs holds:
\begin{eqnarray} \label{system2:CIR}
\begin{cases}
\begin{array}{ll}
   \partial m(t)= m(t)\mu(t), &   \\
    \partial_t G(t,s)=
 \mu(t)D(t) + \int_t^s \mu \left(u \right) \partial_t G(t,u) du + \sigma^2(t)m(t),&
\end{array}
\end{cases}\qquad (t,s) \in \bigtriangleup.
\end{eqnarray}
\end{proposition}

To deduce analogue results for the case of \eqref{eq:timedep:CIR} we require the following stronger assumption  \ref{cond:error:CIR}: 
\begin{itemize}
        \item[\namedlabel{cond:error:CIR}{\textbf{C(4)}}] $\mathbb{E}\vert U_{ij} \vert^{\rho}< \infty$ and
        $\underset{0\leq t \leq 1}{\mathrm{sup}} \mathbb{E} \vert X(t)\vert^{\rho} < \infty $.
\end{itemize}
See Lemma \ref{lemma:sup:CIR}. Indeed, local polynomial smoothing entails the same rates of convergence for the mean function $m(\cdot)$ (and its derivative) and for the second moment function $G(\cdot , \cdot)$ (and its partial derivatives) as stated in  Theorems \ref{thm : m, G}, provided \ref{cond:smoothness}, \ref{cond:design} and \ref{cond:error:CIR} hold with $\rho >2$ and $\rho > 4$, respectively.
 Moreover, defining $\widehat{\mu}$ and {$\widehat{\sigma}^2_D$} or $\widehat{\sigma}^2_T$ by plug-in estimation in terms of \eqref{system2:CIR}: 
 \begin{eqnarray*}
 \left\{
\begin{array}{ll}
    \widehat{\mu}(t) =  \widehat{\partial m}(t)(\widehat{ m}(t) )^{-1},&  \\
    \widehat\sigma^2_D(t) = {\widehat{\partial D}(t)(\widehat{ m}(t))^{-1} -2 \widehat{\mu}(t) \widehat{ D}(t)(\widehat{ m}(t))^{-1}},&  \\
    \widehat\sigma^2_T(t) = (\widehat{m}(t))^{-1} \frac{1}{1-t}\int_t^1\left(\widehat{\partial_t G}(t,s) - {\widehat\mu}(t) \widehat{D}(t) - \int_t^s {\widehat\mu} \left(u \right) \widehat{\partial_t G}(t,u) du \right)ds. & 
\end{array}
t \in [0,1] \right.
\end{eqnarray*}
we can show that the same rates of convergence hold as stated in Theorem  \ref{thm : mu, sigma}, provided  $m(0)\neq0$ (see \eqref{eq:mean:CIR}), \ref{cond:smoothness}, \ref{cond:design} and \ref{cond:error:CIR} hold with $\rho >2$ and $\rho >4$ for the drift and diffusion estimators, respectively.

\begin{lemma}\label{lemma:sup:CIR} 
Let the stochastic process $\{X(t)\}_{t \geq 0}$ be defined as \eqref{eq:timedep:CIR} and  $\rho $  be a natural power of $2$ i.e. $\rho \in \{2,4,8,16,\cdots\}$,  then condition 
 $\underset{0\leq t \leq 1}{\mathrm{sup}} \mathbb{E} \vert X(t)\vert^{\rho} < \infty $  holds if and only if   $\mathbb{E} \vert X(0)\vert^{\rho} < \infty$.
\end{lemma}
\subsubsection*{Case III : $\alpha = 0 $  and $\beta \in \{0,1/2,1 \}$}
The remaining cases corresponding to $\alpha = 0 $  and $\beta \in \{0,1/2,1 \}$  can be handled in unison: 
\begin{align}\label{eq:timedep:alpha = 0}
    dX(t) =& \mu(t) dt + \sigma(t) X^{\beta}(t)dB(t), \quad 0 \leq t \leq 1, \;\beta \in \{0,1/2,1 \}.
\end{align}

\begin{proposition}\label{prop:system2:alpha = 0}
Let the stochastic process $\{X(t)\}_{t \geq 0}$ satisfy \eqref{eq:timedep:alpha = 0} and  assume \ref{cond:smoothness} holds. Then, $m(\cdot) \in \mathcal{C}^{d+1}([0,1],\mathbb{R})$ and  $G(\cdot,\cdot )
\in \mathcal{C}^{d+1}(\bigtriangleup , \mathbb{R})$.  Moreover:
\begin{eqnarray}\label{eq:G_expansion:alpha = 0}
G(t,s)&=&  G\left(0,0 \right)+ 2\int_0^t m \left( u \right) \mu \left(u\right)du  + m \left( t \right) \int_t^s  \mu(u) du+ \int_0^t  \sigma^2 \left(u \right)\mathbb{E}( X^{2\beta}(u)) du
\end{eqnarray}
In particular:
\begin{eqnarray}\label{eq: G decomp:alpha = 0}
 \partial_t G(t,s) &=& 
  m(t) \mu(t)  + \partial m(t) \int_t^s  \mu(u) du   + \sigma^2(t) \mathbb{E}( X^{2\beta}(t)) 
\end{eqnarray}
and the following system of PDEs holds:
\begin{eqnarray} \label{system2:alpha = 0}
\begin{cases}
\begin{array}{ll}
   \partial m(t)= \mu(t), &   \\
    \partial_t G(t,s)= m(t) \mu(t)  + \partial m(t) \int_t^s  \mu(u) du   + \sigma^2(t) \mathbb{E}( X^{2\beta}(t)) &
\end{array}
\end{cases}\qquad (t,s) \in \bigtriangleup.
\end{eqnarray}
\end{proposition}
Notice that the expression $\mathbb{E}\left( X^{2\beta}(t)\right)$ appearing in  \eqref{system2:alpha = 0} equals to $1$, $m(t)$ and $D(t)$ depending on the value of $\beta = 0$, $\beta = 1/2$ and $\beta = 1$,  respectively. 
Again, we get the same rates of convergence for the mean function $m(\cdot)$ (and its derivative) and for the second moment function $G(\cdot, \cdot)$ (and its partial derivatives) as stated in  Theorems \ref{thm : m, G}, provided \ref{cond:smoothness}, \ref{cond:design} and \ref{cond:error:CIR} hold with $\rho >2$ and $\rho >4$, respectively.
 Moreover, similarly defining $\widehat{\mu}$, {$\widehat{\sigma}_D^2$} and $\widehat{\sigma}_T^2$ by plug-in estimation in terms of \eqref{system2:alpha = 0}:
 \begin{eqnarray*}
 \left\{
\begin{array}{ll}
    \widehat{\mu}(t) =  {\widehat{\partial m}(t)},&  \\
    \widehat\sigma^2_D(t) =  {(\widehat{\xi^{(\beta)}}(t))^{-1} \left( \widehat{\partial D}(t) - 2{\widehat{m}} (t)\widehat{\partial m}(t) - 2{\widehat{m}}(0) \widehat{\partial m}(t
    )\right)},&  \\
    \widehat\sigma^2_T(t) =  {(\widehat{\xi^{(\beta)}}(t))^{-1} \frac{1}{1-t}\int_t^1\left(\widehat{\partial_t G}(t,s) - {\widehat\mu}(t) \widehat{m}(t) - \widehat{\partial m}(t)\int_t^s {\widehat\mu} \left(u \right)  du \right)ds}. & 
\end{array}
t \in [0,1] \right.
\end{eqnarray*}
where $\widehat{\xi^{\beta}}(t)$ is $1, \widehat m(t)$ or $\widehat D(t)$ if $\beta =0,1/2,1$ respectively, we can show that the same rates of convergence hold as stated in Theorem  \ref{thm : mu, sigma}, provided   \ref{cond:smoothness}, \ref{cond:design} and \ref{cond:error:CIR} hold with $\rho >2$ and $4$ for the drift and diffusion estimators, respectively.

\section{Simulation Studies} \label{sec:smulation}

In this section we illustrate how our methodology can be used to estimate the drift and diffusion functions from multiple i.i.d. diffusions observed sparsely, on an random grid, with stochastic noise corruption.

We simulated the $n$ iid trajectories of the diffusions $\{X_i(t)\}_{i=1,\dots,n}$ by the Euler-Maruyama approximation scheme (see e.g. \citet{lord_introduction_2014-1}) with step-size $dt = 10^{-3}$. For each path $X_i$, we selected $r$ random locations $\{T_{ij}\}_{j=1,\dots,r}$ (uniformly) in the unit interval  and ordered them increasingly. Finally, we considered the value of each trajectory $X_i$ at the sampled, ordered times $T_{ij}$, with a random additive noise error. That is, we consider:
$$
Y_{ij} = X_i(T_{ij}) + U_{ij}, \qquad i=1,\dots,n, \:\: j = 1,\dots, r
$$
where $U_{ij}$ are iid mean-zero Gaussian measurement errors with finite variance $\nu^2$.

\begin{figure}[h!]
\centering
\includegraphics[width=1\textwidth]{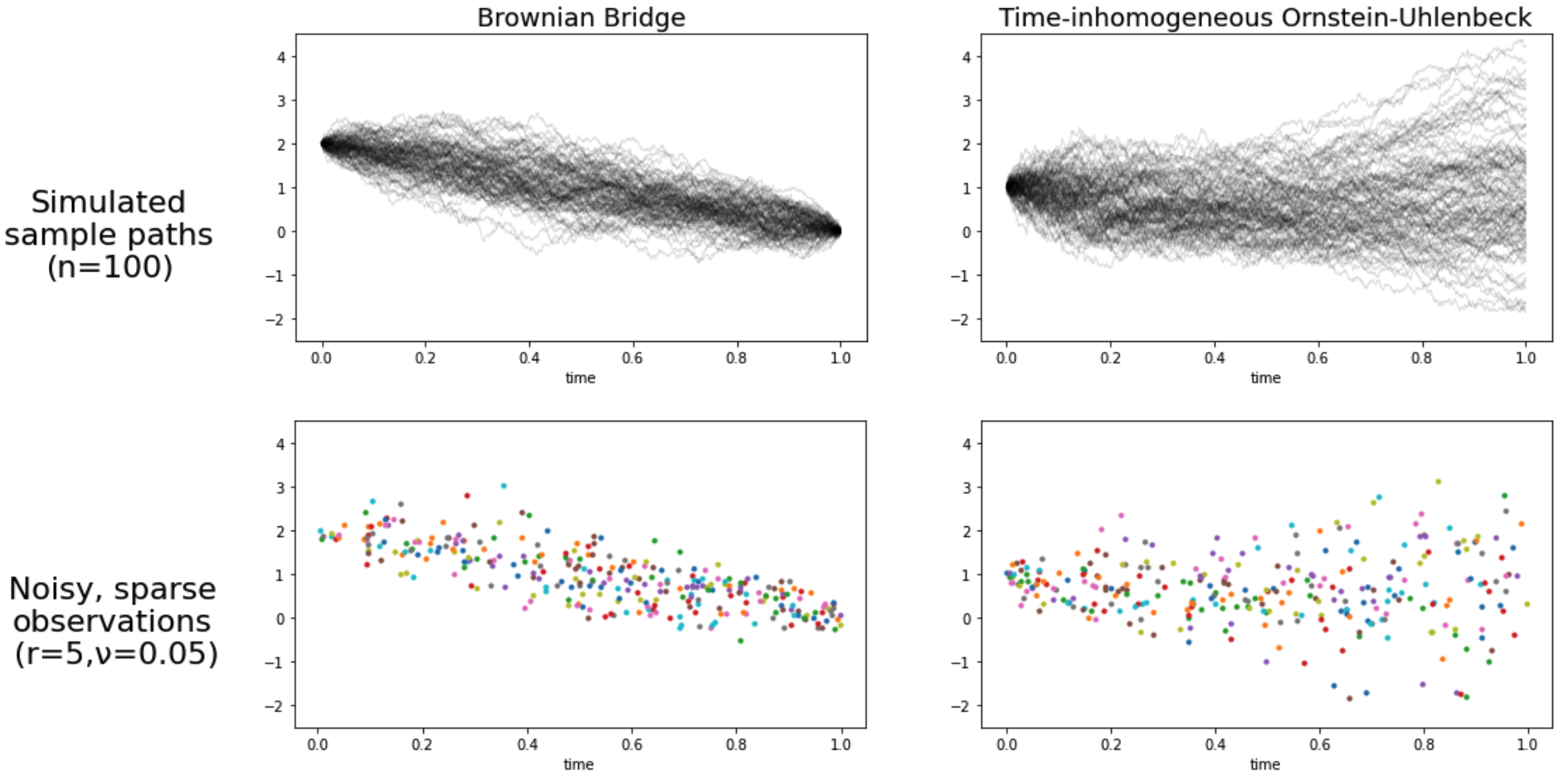}
\caption{The top pair of figures depict $n=100$ simulated sample paths of the two time-inhomogeneous diffusion processes we consider in our analysis in Subsections \ref{simulation_study_BB} and \ref{simulation_study_CX2}, respectively. The bottom pair depicts the discretely measured versions with $r=5$  observations per sample path, sampled uniformly on $[0,1]$, with additive mean-zero Gaussian noise with standard error $\nu = 0.05$.}
\end{figure}

We estimated mean and covariance of the processes by local linear and surface smoothing ($d=2$) and consequently drift and diffusion estimators by plug-in estimation. We investigated the performance of our estimation methodology across different values of $n$ -- the number of sampled trajectories, ranging from $100$ to $1000$ -- and $r$ -- the number of observed locations per path, ranging from $2$ to $10$. In addition, we explored the performance of our methodology across different values for the standard deviation of the measurement error: $\nu \in \{0, 0.05, 0.1\}$. Note that the latter two choices for $\nu$ correspond to a signal-to-ratio of roughly 24dB and 30dB. For both linear and surface smoothing we used the Epanechnikov kernel, with bandwidth $h=(n\cdot r)^{-(1/5)}$. Evaluation of the estimated functional values was done on a discretized grid over $[0,1]$, which was subdivided into $25$ sub-intervals.

We investigate two examples of time-varying SDEs: the Brownian-Bridge, which has a time-varying drift and constant diffusion coefficient: see Subsection \ref{simulation_study_BB}; and a time-inhomogeneous Ornstein-Uhlenbeck process, with sinusoidal time-varying drift and negative exponential time-varying diffusion: see Subsection \ref{simulation_study_CX2}.

{In our simulation study, we investigate the finite-sample behaviour of our drift estimator, appearing in the first line of \eqref{estimate:system2}. Furthermore, we compare two different estimators for the diffusion:
\begin{itemize}
    \item a plug-in estimator obtained by second line in the 
    simple linear equation \eqref{system1}, or equivalently by taking $s=t$ in the second line of \eqref{estimate:system2}:
     $$\widehat\sigma^2_D(t) = \widehat{\partial_t G}(t,t) +\widehat{\partial_s G}(t,t) + 2\widehat m(t)\widehat{\partial m}(t) -  2\mu(t) (\widehat G(t,t) + \widehat m(t)^2),$$

    \item  the plug-in estimator obtained by averaging the family of plug-in estimators in \eqref{estimate:system2}:
    $$
    \widehat\sigma^2_T(t) = \frac{1}{1-t}\int_t^1\left(\widehat{\partial_t G}(t,s) - {\widehat\mu}(t) \widehat{G}(t,t) - \int_t^s {\widehat\mu} \left(u \right) \widehat{\partial_t G}(t,u) du \right)ds.
    $$
\end{itemize}
   }

For every triple $(n, r, \nu)$, $100$ Monte Carlo iterations were run. The performance of the proposed estimators  was evaluated by computing the average square root of the integrated squared error {(RISE)} obtained over $100$ Monte Carlo runs, i.e. the average $L^2$-distance between each functional estimator and the corresponding true function. That is, every Monte Carlo run accounted for computing the errors $\| \mu - \widehat \mu\|_{L^2([0,1])}, \: \| \sigma^2 - \widehat \sigma_D^2\|_{L^2([0,1])}$ and $\| \sigma^2 - \widehat \sigma_T^2\|_{L^2([0,1])}$, and we looked into the distribution of these errors over $100$ iterations.

\subsection{Example 1: Brownian-Bridge}
\label{simulation_study_BB}
The first model we consider is a Brownian-Bridge:
\begin{equation} \label{eq : sde, Brownnian Bridge}
\begin{cases}
 dX(t) =   \displaystyle - \frac{1}{1-t} X(t)dt + dB(t),\qquad t\in [0,1].\\
 X(0) = 1.
\end{cases}
\end{equation}
i.e.\  \eqref{eq : our sde, integral form} with time-varying drift and constant diffusion:
$\mu(t) = -1/(1-t), \sigma(t) = 1.$ Note that the process is defined in $1$ by its left-hand limit (well known to be a.s. $0$). The fact that  the drift term is {not well-defined} at $t = 1$ implies that \eqref{eq : sde, Brownnian Bridge} does not satisfy our assumptions. However, a quick look at the closed form expressions for the first and second moment of the Brownian Bridge will show that our methodology is successful in dealing with such a case as well. Indeed, the key feature needed for local linear (surface) regression estimates to satisfy the convergence rates in Theorem \ref{thm : m, G}  is the \textit{smoothness} of $m$ and $G$.

\begin{figure}[h!]
    \centering
    \includegraphics[height=.13\textheight]{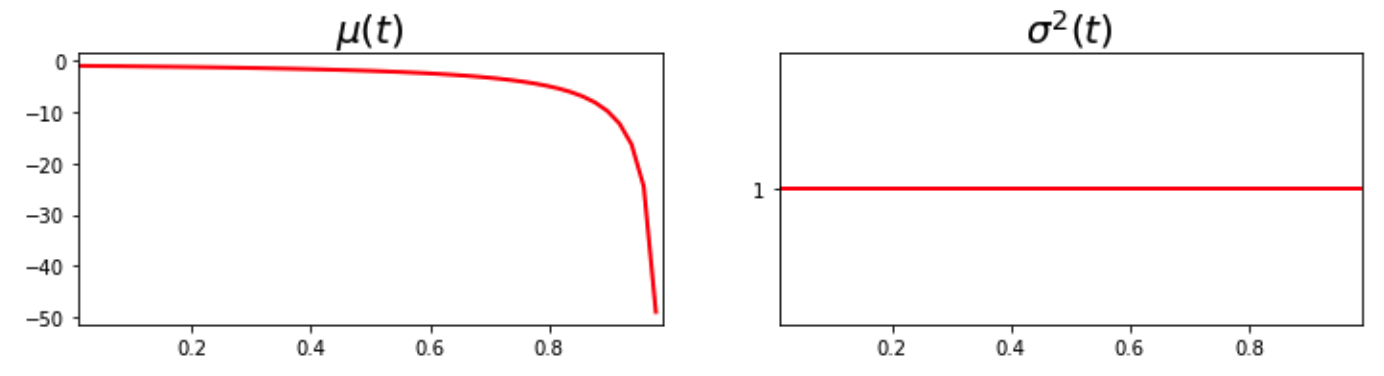} 
    \caption{Example 1. Drift and diffusion functions for the Brownian Bridge process  \eqref{eq : sde, Brownnian Bridge}}
    \label{fig:BBdriftdiff}
\end{figure}

\begin{figure}[h!]%
    \centering
    \includegraphics[width=1\textwidth]{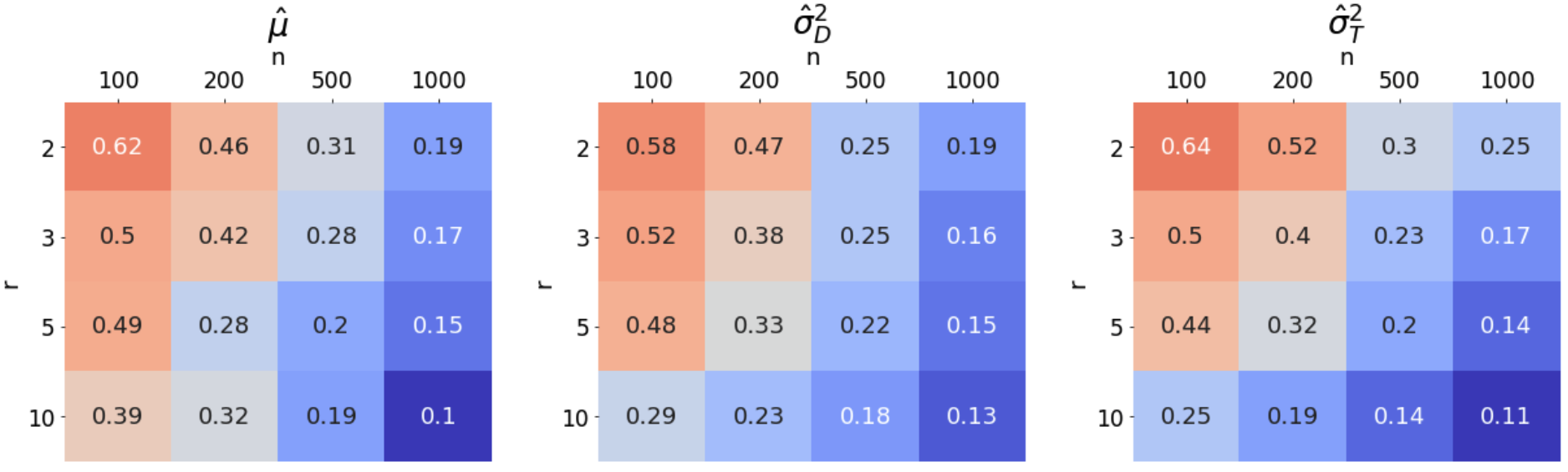} %
    \caption{Example 1. Heatmap illustrating the median RISE of our estimators $\widehat \mu, \widehat \sigma^2_D$ and $\widehat \sigma^2_T$ over 100 independent experiments, for different values of $(n, r)$. Observational error variance was set at $\nu = 0.05$. Dark blue shades indicate a low median RISE, while dark red shades indicate high median RISE. The heat-maps show the convergence of the estimators in the sparse regime, as $n$ increases for different values of $r$: that is, for every fixed value of $r$ (row) we see how increasing $n$ (column) leads to progressively more accurate estimates, and that convergence is achieved faster with greater values of $r$.}%
    \label{fig:BBheatmap}%
\end{figure}
We explored the behaviour of the drift estimator $\widehat \mu$ and of the two diffusion estimators $\widehat \sigma_D, \widehat \sigma_T$
across different values for the number of paths $n\in\{100,200,500,1000\}$ and for the number of observed points per path $r\in \{2,3,5,10\}$. The average and median RISE are illustrated in the heatmaps and boxplots in Figure \ref{fig:BBheatmap} and \ref{fig:BBboxplots}, respectively. The effect of measurement noise is investigated in Figure \ref{fig:BBnoise}, where we considered Gaussian additive noise with standard deviation $\nu\in\{0,0.05,0.1\}$.

\begin{figure}[h!]%
\centering 
\includegraphics[width=1\textwidth]{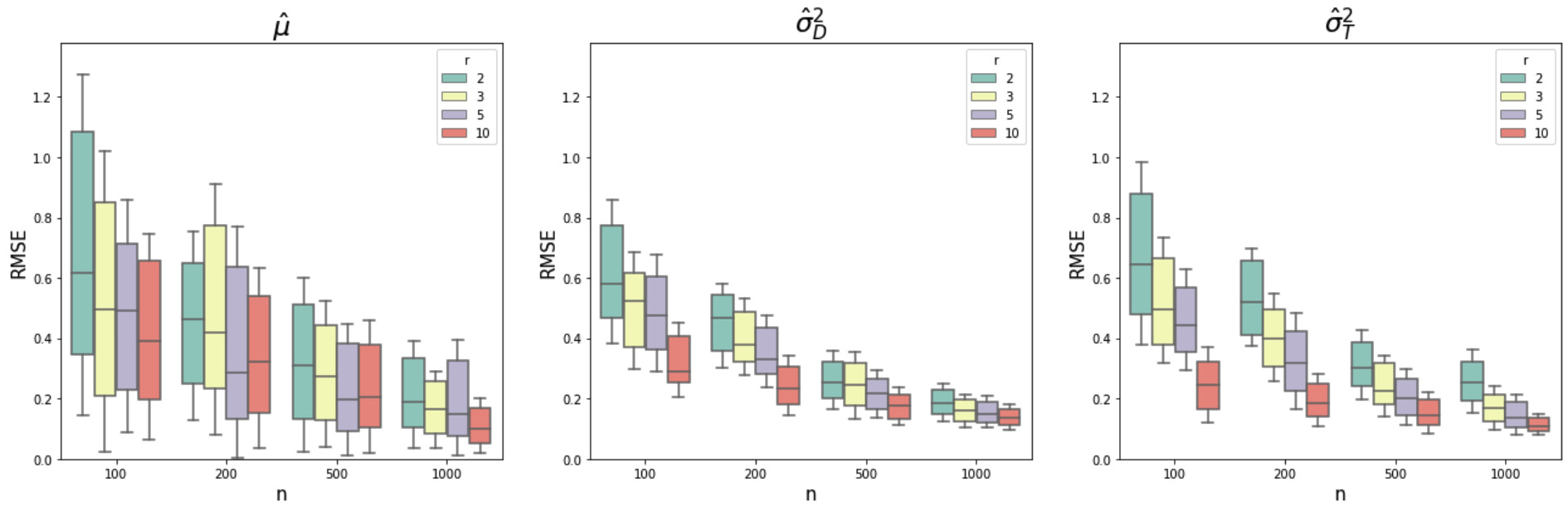}
\caption{Example 1. Boxplots of the RISE scores for different values of $(n, r)$.}%
\label{fig:BBboxplots}
\end{figure}
We also investigated the effect of noise contamination on our estimation  procedure, and -- in particular -- we conducted a comparison between estimation performance when assuming knowledge that the observations are error-less (hence including diagonal observations in the local surface regression for $G$) \emph{vs} being agnostic as to the presence of random measurement error (and hence excluding diagonal observations): the plots in Figure \ref{fig:BBnoise} show the average RISE for fixed $r=5$ as $n$ grows to infinity, with 3 different values for the standard deviation of the error $\nu \in \{0,0.05, 0.5\}$.
\begin{figure}[h!]
\centering
\includegraphics[width=1\textwidth]{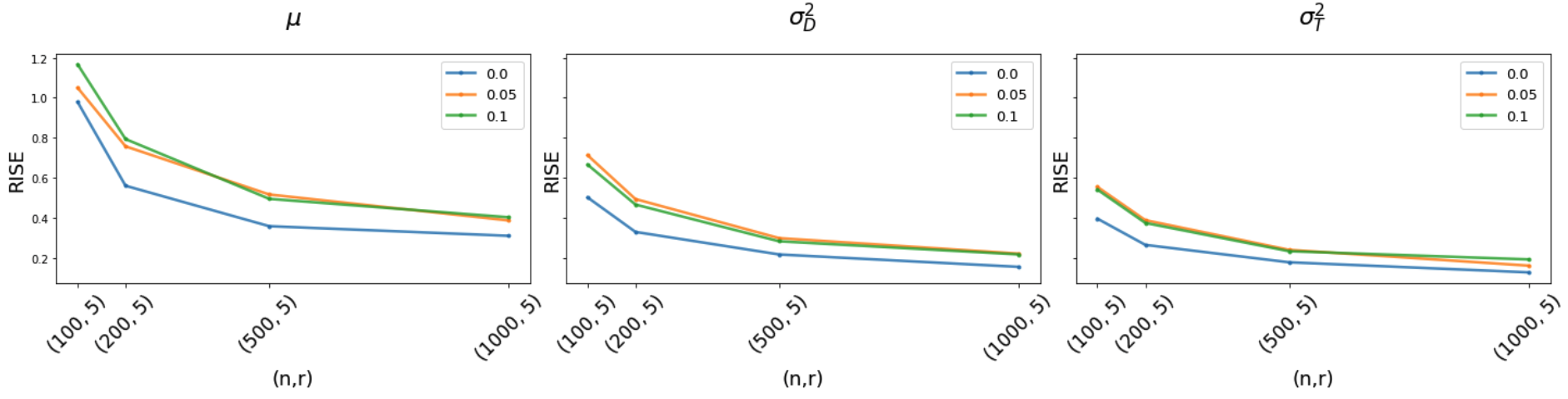}
\caption{Example 1. Comparing the average RISE with different values for the observational error standard-deviation: $\nu = 0,0.05, 0.5$. We consider $r=5$ observations per path, and $n=100,200,500,1000$ paths.}
\label{fig:BBnoise}
\end{figure}
Considering that the estimator $\widehat\sigma_T$ is linked to the entire covariance surface, whereas  $\widehat\sigma_D$ only to its diagonal values, we expected the first to achieve a lower average RISE, although the rates of convergence proven are the same. And indeed, our analysis shows that the estimator $\sigma_T^2$ seems to suffer less from  boundary issues, and appears to have lower variance, compared to $\sigma_D^2$: this is for instance very evident by comparing the $95\%$ confidence bands of the estimated trajectories for the diffusion in Figure \ref{fig:BBCI}

\begin{figure}[h!]
\centering
\includegraphics[width=1\textwidth]{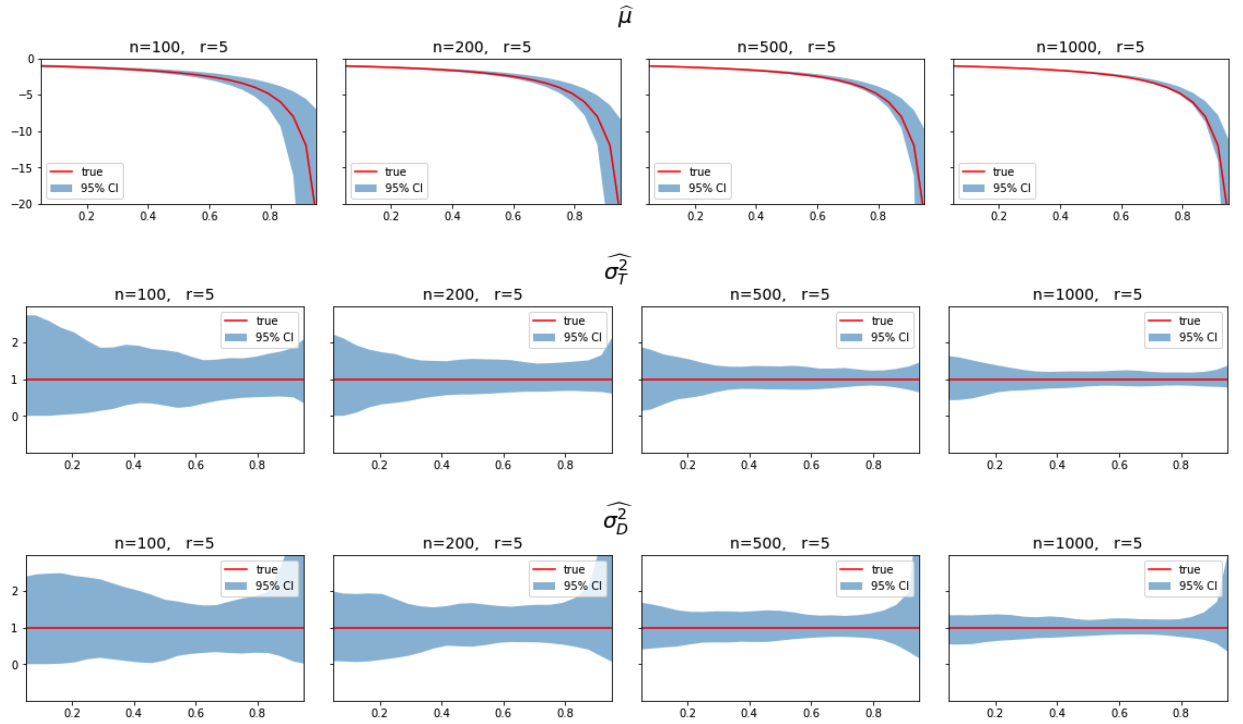}
\caption{Example 1. $95\%$ confidence bands of the estimated drift and diffusion functions in $100$ Monte carlo runs. We assumed no error contamination, $r=5$ and $n = 100,200,500,1000$.}
\label{fig:BBCI}
\end{figure}

 We finally considered estimation of the integrated diffusion function $t\mapsto \int_0^t\sigma^2(u)du$ for $t\in[0,1]$ following the discussion and heuristic presented with the extensions in Section \ref{sec:Extension}. We emphasize that, as expected, estimation appears to be more stable indeed, and convergence appears to be achieved faster. In fact, Figure \ref{fig:BBID-heatmap_boxplots} was necessarily presented in a different $y$-scale, as the average errors are of a different order of magnitude.

\begin{figure}[h!]
\includegraphics[height=.027\textheight]{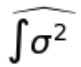}\\
\begin{minipage}{.4\textwidth}
\includegraphics[height=.195\textheight]{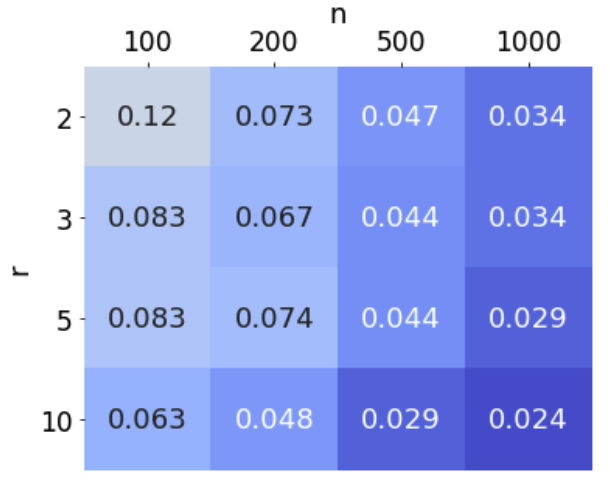}
\end{minipage}\hspace{.02\textwidth}
\begin{minipage}{.4\textwidth}
\includegraphics[height=.22\textheight]{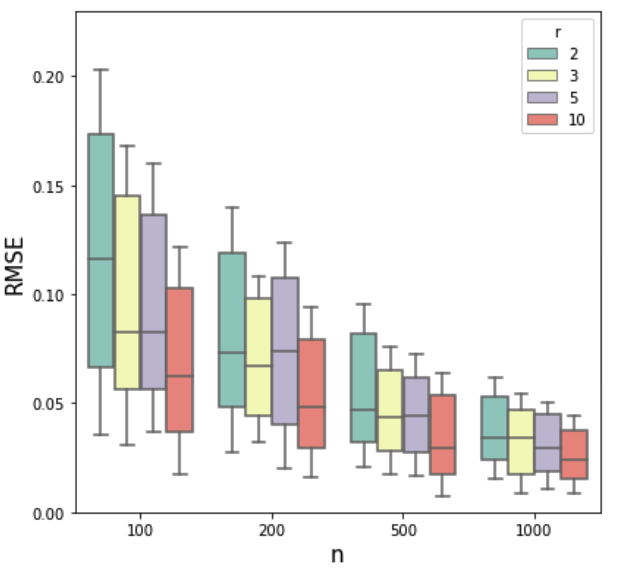} 
\end{minipage}
\caption{Example 1.  Heatmap and boxplots of the (median) RISE scores for different values of the number of paths $n\in\{100,200,500,1000\}$ and number of observed points per path $r\in \{2,3,5,10\}$. }% }
\label{fig:BBID-heatmap_boxplots}
\end{figure}

\begin{figure}[h!]
\centering
\includegraphics[width=1.0\textwidth]{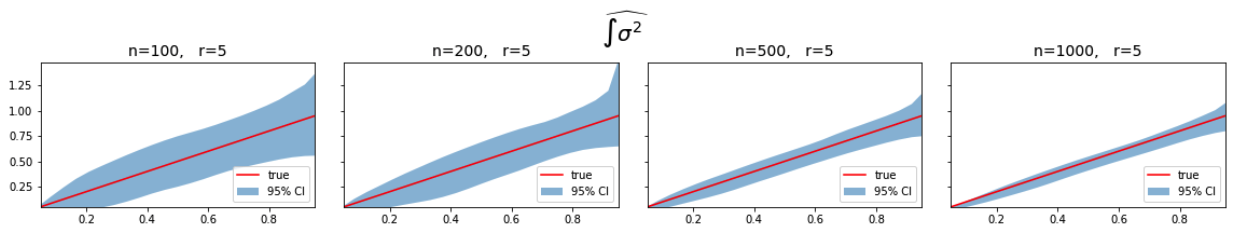}
\caption{Example 1. $95\%$ confidence bands of the estimated integrated volatility Monte carlo runs. We assumed no error contamination, $r=5$ and $n = 100,200,500,1000$.}
\label{fig:BBID-CI}
\end{figure}

\newpage

\subsection{Example 2: Time-inhomogeneous Ornstein-Uhlenbeck process}
\label{simulation_study_CX2}
The second case we consider is that of a time-inhomogeneous Ornstein-Uhlenbeck process defined as:
\begin{equation} \label{eq : sde, cx2}
\begin{cases}
% -2*(0.2+0.8*np.sin(2*np.pi*t))
% np.sqrt(np.exp(1)**(1-t)**2)
 dX(t) =   \displaystyle - \frac{1}{5}(1+\sin(2\pi t)) X(t)dt 
 + \sqrt{e^{(1-t)^2}} dB(t)
 ,\qquad t\in [0,1].\\
 X(0) = 2.
\end{cases}
\end{equation}
i.e.\  \eqref{eq : our sde, integral form} with time-varying drift and diffusion:
$
\mu(t) = -\frac{1}{5}(1+\sin(2\pi t)), 
\sigma(t) = \sqrt{e^{(1-t)^2}}$.

\begin{figure}[h!]
    \centering
    \includegraphics[height=.12\textheight]{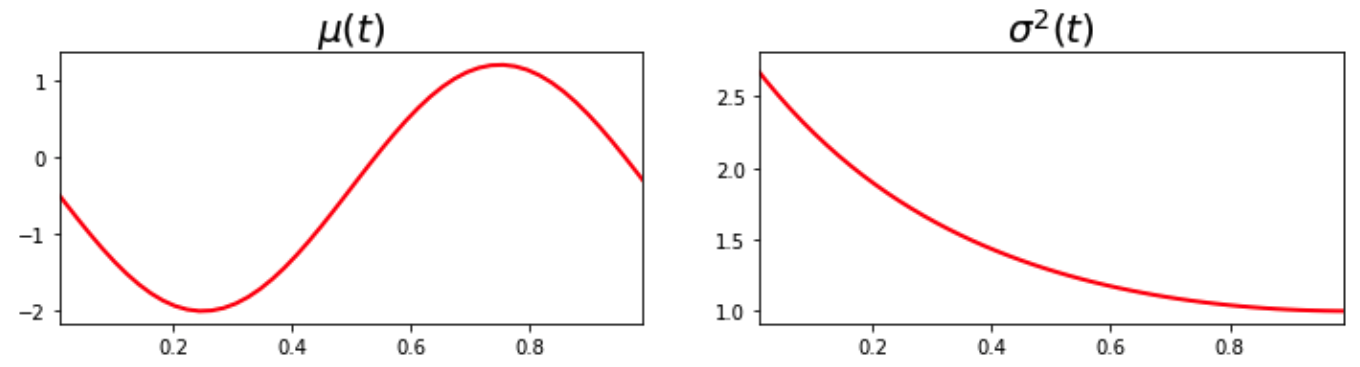}
    \caption{Example 2. Drift and diffusion functions for the time-inhomogenous Ornstein-Uhlenbeck process  \eqref{eq : sde, cx2}}
    \label{fig:CX2driftdiff}
\end{figure}

This particular choice of sinusoidal drift and negative-exponential diffusion was made in order to highlight the ability of the described methodology to recover \textit{non-trivially} time-varying drift and diffusions. We considered the same simulation framework as in Example 1 (the values of $n,r,\nu$ as well as the number of iterations), and observed very similar asymptotic behaviour concerning the convergence of our estimators (the reader is referred to the corresponding plots). This may be considered as further indication confirming the flexibility of our methodology, as well as its ability to capture local variability and estimate non-parametric time-varying drift and diffusion functionals. Since the qualitative conclusions are essentially the same as in Example 1, we do not repeat them in detail.

\begin{figure}[h!]%
    \centering
    \includegraphics[width=0.95\textwidth]{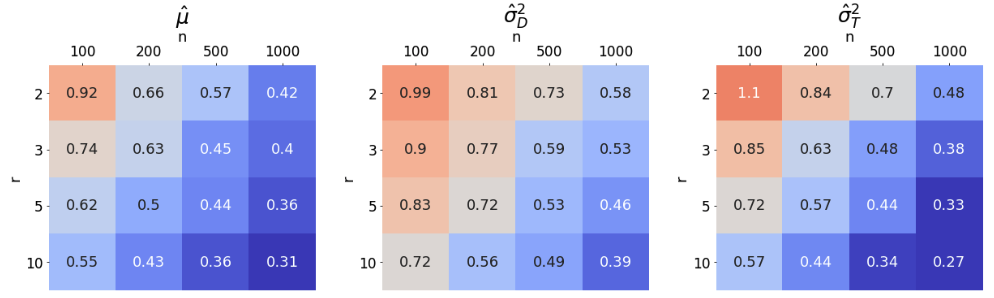} %
    \caption{Example 2. Heatmap illustrating the median RISE of our estimators $\widehat \mu, \widehat \sigma^2_D$ and $\widehat \sigma^2_T$ over 100 independent experiments, for different values of $(n, r)$. Observational error variance was set at $\nu = 0.05$. Dark blue shades indicate a low average RISE, while dark red shades indicate high average RISE. The heat-maps show the convergence of the estimators in the sparse regime, as $n$ increases for different values of $r$.}%
    \label{fig:CX2heatmap}%
\end{figure}

\begin{figure}[h!]%
\centering 
\includegraphics[width=1\textwidth]{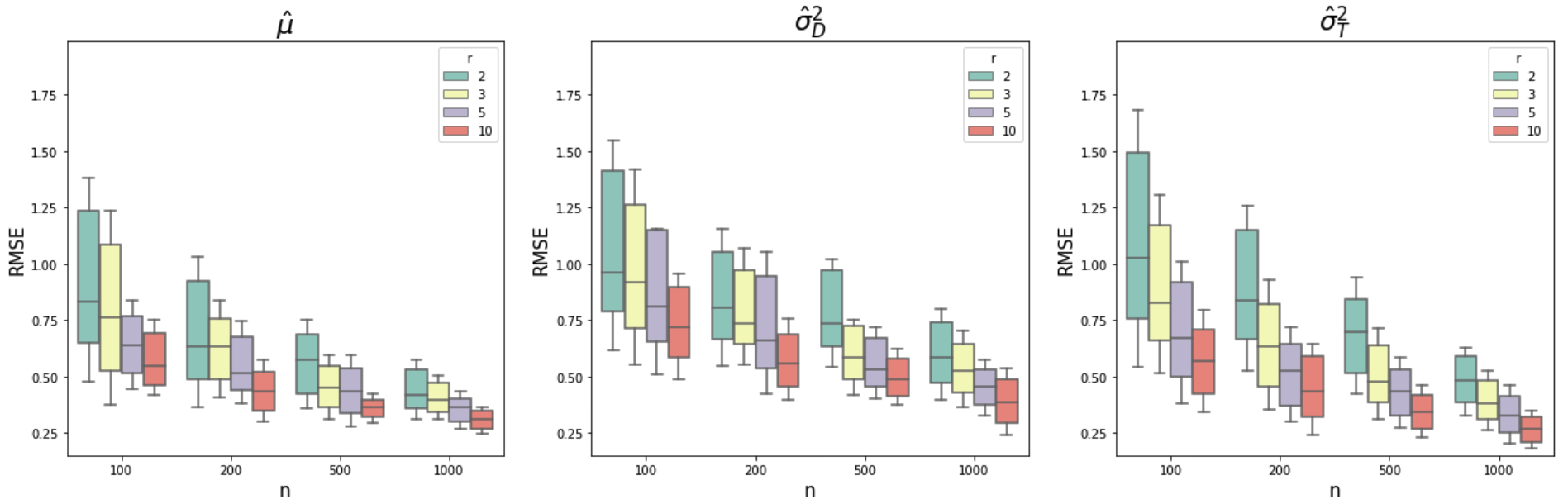}
\caption{Example 2. Boxplots of the RISE scores for different values of $(n, r)$.}%
\label{fig:CX2boxplots}
\end{figure}

\begin{figure}[h!]
\centering
\includegraphics[width=1\textwidth]{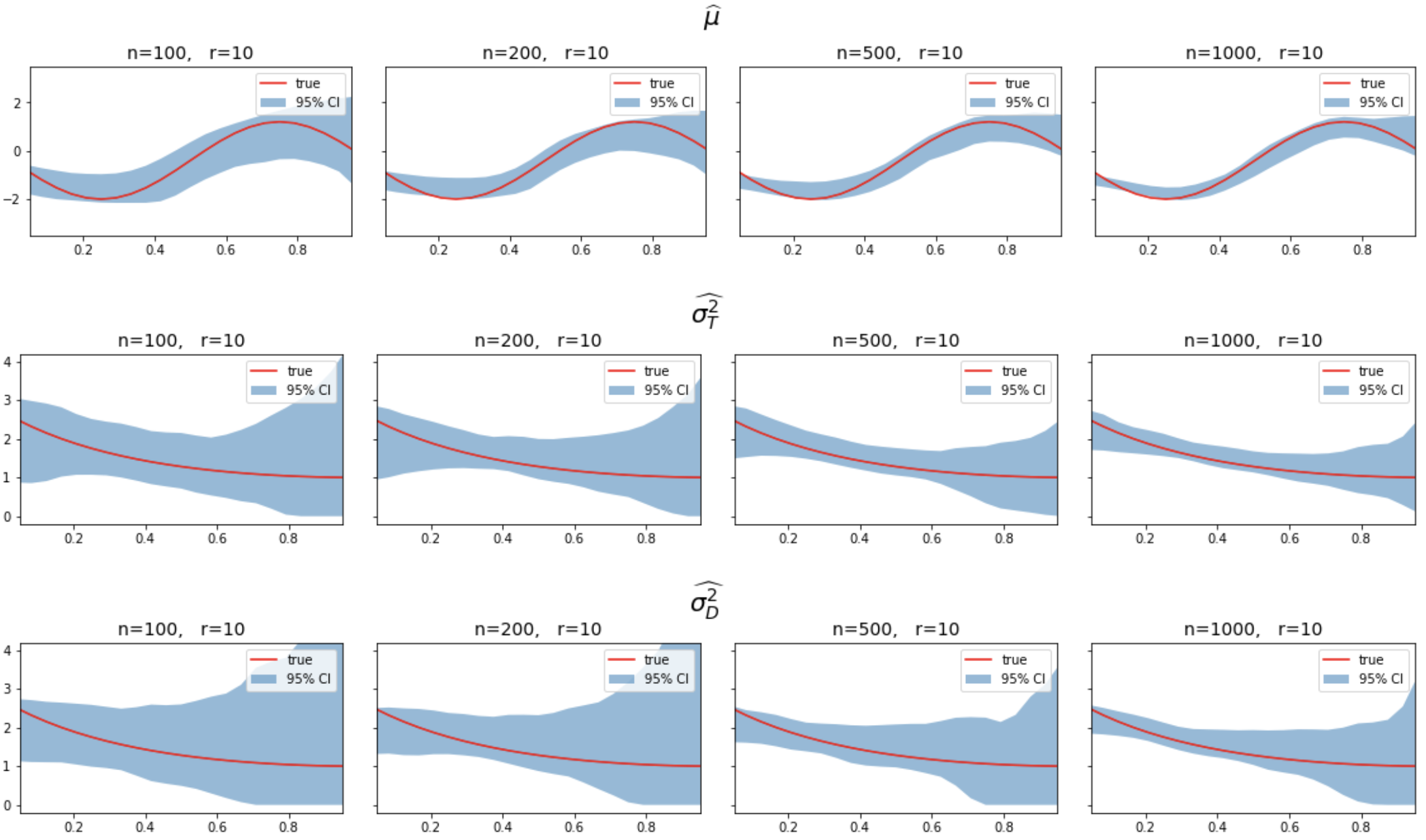}
\caption{Example 2. $95\%$ confidence bands of the estimated drift and diffusion functions in $100$ Monte carlo runs. We assumed no error contamination, $r=10$ and $n = 100,200,500,1000$.}
\label{fig:CX2CI}
\end{figure}

\begin{figure}[h!]
\includegraphics[height=.027\textheight]{graphics2/___header.png}\\
\begin{minipage}{.4\textwidth}
\includegraphics[height=.195\textheight]{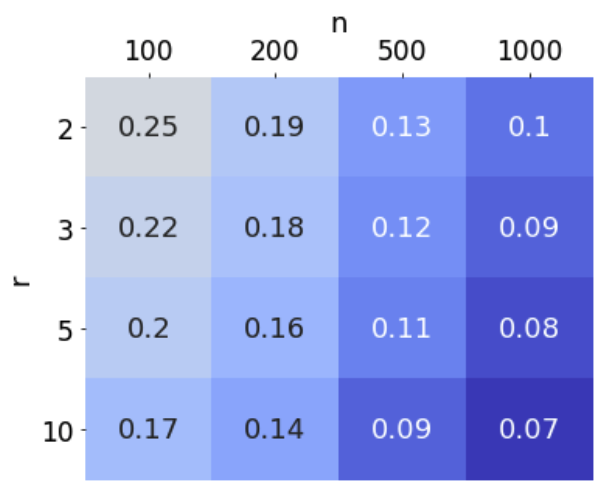}
\end{minipage}\hspace{.02\textwidth}
\begin{minipage}{.4\textwidth}
\includegraphics[height=.22\textheight]{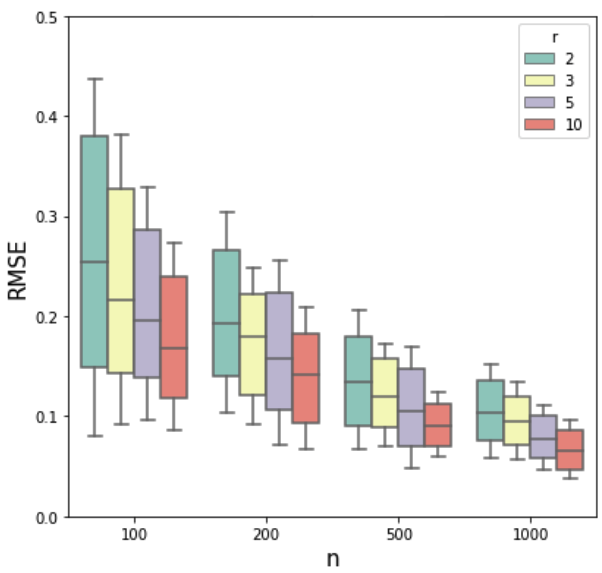} 
\end{minipage}
\caption{Example 2.  Heatmap and boxplots of the (median) RISE scores for different values of the number of paths $n\in\{100,200,500,1000\}$ and number of observed points per path $r\in \{2,3,5,10\}$. }% }
\end{figure}

\begin{figure}[h!]
\centering
\includegraphics[width=1.0\textwidth]{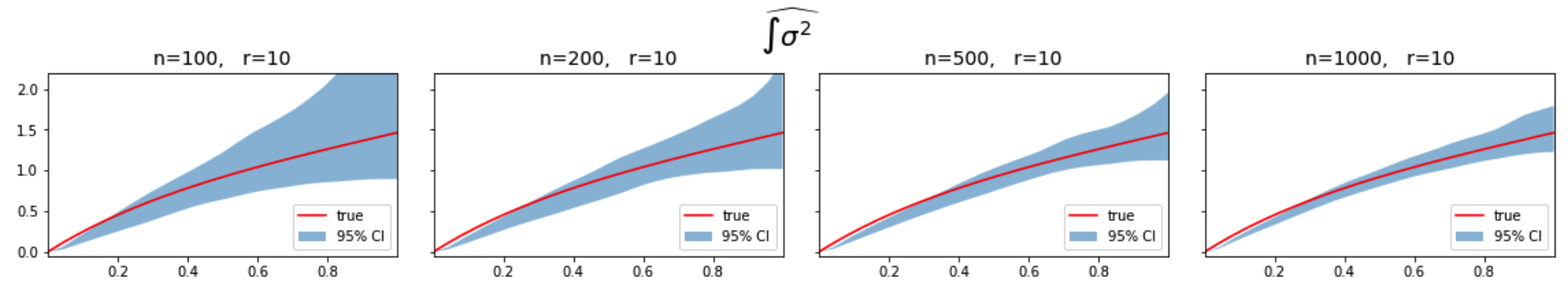}
\caption{Example 2. $95\%$ confidence bands of the estimated integrated volatility Monte carlo runs. We assumed no error contamination, $r=5$ and $n = 100,200,500,1000$.}
\end{figure}

% %%%%%%%%%%%%%

\hfill
\newpage
\hfill
\newpage

\section*{Appendix}\label{sec:appendix}
\setcounter{equation}{0}
\renewcommand{\theequation}{A.\arabic{equation}}

\subsection*{Proof of the Results of Section \ref{sec:methodology and theory}}
In the sequel we use the fact that the stochastic process $\left\{ X(t) \right\}$ satisfying equation \eqref{eq : our sde, integral form} admits the following closed form representation
\begin{eqnarray} \label{closed form}
X(t) &=& X(0) \cdot\mathrm{exp}\left(\int_0^t \mu(v)dv \right) + \int_0^t  \mathrm{exp}\left(\int_v^t \mu(u)du \right)  \sigma(v)dB(v)\\ \label{closed form 2}
&=& X(s)\cdot\mathrm{exp}\left(\int_s^t \mu(v)dv \right) + \int_s^t  \mathrm{exp}\left(\int_v^t \mu(u)du \right)  \sigma(v)dB(v),\;\;\; 0 \leq s \leq t \leq 1,
\end{eqnarray}
see \citet{Oskendal03}, p. 98.
% \begin{proof}[Proof of Proposition \ref{prop:system1}]

\medskip
For completeness,we provide a proof of \eqref{system1}.
In light of the closed form solution \eqref{closed form}, the functions $m(\cdot)$ and  $D(\cdot)$ admit the following forms:
\begin{eqnarray} \label{mean:function}
m(t) &=& m(0)\cdot\mathrm{exp}\left(\int_0^t \mu(v)dv \right),
\end{eqnarray}
and
{\begin{eqnarray} \nonumber
D(t) &=& D(0)\cdot \mathrm{exp} \left(2\int_0^t \mu(v)dv \right) + \int_0^t \left(  \mathrm{exp}\left(\int_v^t \mu(u)du \right) \sigma(v)\right)^2 dv,
\end{eqnarray}}
and hence the desired smoothness property  $m(\cdot),D(\cdot) \in \mathcal{C}^{d+1}([0,1] , \mathbb{R})$ is deduced via the assumption that $\mu(\cdot),\sigma(\cdot) \in \mathcal{C}^d([0,1]) $.

\begin{proof}[Proof of Proposition \ref{prop:system2}]
First, observe that the closed form representation \eqref{closed form 2} implies
\begin{eqnarray}
\label{covariance:function}
G(s,t) &=& D(s)\cdot \mathrm{exp}\left(\int_s^t \mu(v)dv \right) ,\;\;\; 0 \leq s \leq t \leq 1.
\end{eqnarray}
This entails the desired smoothness of the second moment function  $G(\cdot,\cdot)$ on $\bigtriangleup$. 
\\ For $s\geq0$, define the stopped process:
$$
Z(t) = X(t) \mathbb{I}\{t \leq s\} + X(s) \mathbb{I}\{s<t\},\;\;\;\; t\geq 0.
$$
Consequently:
\begin{eqnarray}\label{eq:X&Z}
\left(
\begin{array}{c}
   X(t) \\
   Z(t)
\end{array}
\right) &=& \left(
\begin{array}{c}
   X(0) \\
   X(0)
\end{array}
\right) + \int_0^t \left(
\begin{array}{c}
   \mu(u) X(u)  \\
   \mu(u) X(u) \mathbb{I}\{u \leq s\}
\end{array}
\right)du + \int_0^t \left(
\begin{array}{c}
   \sigma(u) \\
   \sigma(u) \mathbb{I}\{u \leq s\}
\end{array}
\right)dB(u),
\end{eqnarray}
which gives a well-defined and unique (up to indistinguishablility) It{\^o} diffusion process. 
Applying It{\^o}'s integration by parts formula (see \citet{Oskendal03}) to \eqref{eq:X&Z}, we obtain the following representation for $\mathbb{E}\left( Z(t) X(t) \right)$:
\begin{eqnarray*}
\mathbb{E}\left( Z(t) X(t) \right) &=& \mathbb{E}\left( X(0) X(0) \right)  + \mathbb{E} \int_0^t X(u) dZ(u)  + \mathbb{E} \int_0^t Z(u) dX(u)  + \mathbb{E} \int_0^t dX(u) dZ(u).
\end{eqnarray*}
In particular, by the definition of the stopped process $\{Z(t)\}$, we conclude:
\begin{eqnarray*}
\mathbb{E} \left( X(s) X(t) \right) &=& 
\mathbb{E}\left( X(0) X(0) \right)  + 2\mathbb{E} \int_0^s X^2(u) \mu \left(u\right) du  + \mathbb{E} \int_s^t X(s) X(u) \mu\left(u \right) du  + \mathbb{E} \int_0^s \sigma^2 \left(u \right) du.
\end{eqnarray*}
By the linear growth condition, the continuity of mean and mean squared functions, and the Fubini-Tonelli Theorem, we can interchange integral and expectation to obtain:
\begin{eqnarray*}
\label{EX_sX_t}
\mathbb{E}\left( X(s) X(t) \right)&=& \mathbb{E}\left( X(0) X(0) \right)+ 2\int_0^s \mathbb{E} \left( X^2(u) \right) \mu \left(u\right)du  + \int_s^t \mathbb{E} \left( X(s) X(u) \right) \mu(u) du+ \int_0^s  \sigma^2 \left(u \right) du, 
\end{eqnarray*}
as claimed in \eqref{eq:G_expansion}. 
Taking partial derivatives with respect to $s$ we obtain:
\begin{eqnarray*}
\frac{\partial}{\partial s} G(s,t) &=&
 2 \mu(s) G(s,s)  - \mu(s) G(s,s) + \int_s^t \mu \left(u \right) 
\frac{\partial}{\partial s} G(s,u) du + \sigma^2(s),
\end{eqnarray*}
as desired. Integrating with respect to $t$ leads to the  second equation appearing in system \eqref{system2}. Combining this with the first equation in \eqref{system1}  completes the proof.
\end{proof}

\begin{proof}[Proof of Lemma \ref{lemma:sup}]
Recall that the stochastic process $\left\{ X(t) \right\}$ satisfying equation \eqref{eq : our sde, integral form} admits the representation \eqref{closed form}.
In order to prove finiteness of  $\underset{0\leq s \leq 1}{\mathrm{sup}} \mathbb{E} \vert X(s)\vert^{\rho} $, we investigate  the summands appearing in the right hand side of \eqref{closed form} separately. We first observe that 
\begin{eqnarray}
\nonumber
   \left\vert X(0) \right\vert^{\rho} \cdot\mathrm{exp}\left\{\rho\int_0^t \mu(v)dv \right\}  &\leq &  \left\vert X(0)  \right\vert^{\rho} \cdot\mathrm{exp}\left\{\rho\int_0^t \left\vert\mu(v)  \right\vert dv \right\} \\ \nonumber
   &\leq &  \left\vert X(0)  \right\vert^{\rho} \cdot\mathrm{exp}\left\{\rho \int_0^1 \left\vert\mu(v)  \right\vert dv \right\}.
\end{eqnarray}
Hence $\underset{0\leq t \leq 1}{\mathrm{sup}} \mathbb{E} \left\vert X(0) \right\vert^{\rho} \cdot\mathrm{exp}\left\{\rho \int_0^t \mu(v)dv \right\}  $ is dominated by $\mathbb{E} \left\vert X(0) \right\vert^{\rho} \cdot\mathrm{exp}\left\{\rho \int_0^1 \vert\mu(v)\vert dv \right\} $ which in turn is finite as long as $\mathbb{E} \left\vert X(0) \right\vert^{\rho}$ is finite. For the second term appearing in \eqref{closed form}, note that:
\begin{eqnarray} 
\nonumber
 \underset{0\leq t \leq 1}{\mathrm{sup}} \mathbb{E} \left\vert \int_0^t  \mathrm{exp}\left(\int_v^t \mu(u)du \right)  \sigma(v)dB(v)\right\vert^{\rho} &=&  \underset{0\leq t \leq 1}{\mathrm{sup}} \mathbb{E} \left\vert \mathrm{exp}\left(\int_0^t \mu(u)du \right) \int_0^t  \mathrm{exp}\left(-\int_0^v \mu(u)du \right)  \sigma(v)dB(v)\right\vert^{\rho} 
 \\ \nonumber
 &=&  \underset{0\leq t \leq 1}{\mathrm{sup}} \mathrm{exp}\left(\rho\int_0^t \mu(u)du \right)  \mathbb{E} \left\vert \int_0^t  \mathrm{exp}\left(-\int_0^v \mu(u)du \right)  \sigma(v)dB(v)\right\vert^{\rho}
 \\ \nonumber
 &\leq& \mathrm{exp}\left(\rho\int_0^1 \vert \mu(u) \vert du \right)   \underset{0\leq t \leq 1}{\mathrm{sup}}  \mathbb{E} \left\vert \int_0^t  \mathrm{exp}\left(-\int_0^v \mu(u)du \right)  \sigma(v)dB(v)\right\vert^{\rho}
 \\ \nonumber
 &\leq& \mathrm{exp}\left(\rho\int_0^1 \vert \mu(u) \vert du \right) \mathbb{E}  \underset{0\leq t \leq 1}{\mathrm{sup}}   \left\vert \int_0^t  \mathrm{exp}\left(-\int_0^v \mu(u)du \right)  \sigma(v)dB(v)\right\vert^{\rho} 
 \\ \label{BDG:ineq}
 &\leq& C \mathbb{E} 	\langle  \int_0^{\cdot}  \mathrm{exp}\left(-\int_0^v \mu(u)du \right)  \sigma(v)dB(v)	\rangle_1^{\rho/2}
 \\ \label{quad:var:W}
 &=& C \mathbb{E} \left\{  \int_0^1  \mathrm{exp}\left(-2\int_0^v \mu(u)du \right)  \sigma^2(v)dv \right\}^{\rho/2} 
 \\ \nonumber
 &<& \infty.
\end{eqnarray}
Inequality \eqref{BDG:ineq} is a consequence of  Burkholder-Davis-Gundy inequality, see Theorem 5.1 in \citet{gall_brownian_2016}. 
The integrand appearing in the stochastic integral $\int_0^t  \mathrm{exp}\left(-\int_0^v \mu(u)du \right)  \sigma(v)dB(v)$ is deterministic, and hence adapted and measurable. By the continuity of the drift and diffusion functions $\mu(\cdot)$  and  $\sigma(\cdot)$, the integrand is also bounded. Hence, one may apply Proposition 3.2.17  of  \citet{KaratzasIoannis1988BMaS} to conclude  well-definiteness of the quadratic variation process $\langle  \int_0^{\cdot}  \mathrm{exp}\left(-\int_0^v \mu(u)du \right)  \sigma(v)dB(v)	\rangle$. Equality \eqref{quad:var:W} is a consequence of the same result.
\end{proof}

\begin{proof}[Proof of Theorem \ref{thm : m, G}] The proof of Theorem \ref{thm : m, G} mimics the lines of the proof of Theorem 1 in \citet{jouzdani2021functional} which we reproduce here for the sake of completeness. We rigorously justify the convergence rate for the function $G$ and its first order partial derivative $\partial G$ (\eqref{eq: cons G} and \eqref{eq: cons partial G}). Relations \eqref{eq: cons m} and \eqref{eq: cons partial m} can be obtained by a straightforward modification.

\medskip

\noindent The proof follows the following line of argument:

\begin{itemize}
    \item[I.] We first obtain explicit closed-form expressions for the estimators $\widehat{G}$ and $\widehat{\partial G}$: see \eqref{eq : MORE explicit expression Ghat and partialGhat}.
    \item[II.] We use these expressions to decompose the supremum distance between $\widehat{G}$ and $G$ as well as $\widehat{\partial G}$ into several different terms: see \eqref{Ghat - G : decomposition} and \eqref{partialGhat - partialG : decomposition}. 
    \item[III.] We investigate each term appearing in these decompositions, and reduce the problem to the analysis of five terms (named $A_1-A_5$), which consitutes a bias-variance decomposition: terms $A_1-A_5$ representing variance and term $A5$ representing bias.
    \item[IV.]  We handle the variance  terms  $A_1-A_4$. This is done by further decomposing each of these four terms in a sub-decomposition of 4 further terms. Fortunately, the sub-decompositions are structurally similar and can be analysed jointly by representing them in the same general form. We show that each term admits the rate $\left[h_{G}^{-4} \frac{\mathrm{log}n}{n}\left( h_G^4 + \frac{h_G^3}{r}+ \frac{h_G^2}{r^2} \right) \right]^{1/2}$ in an appropriate sense: see \eqref{rate:varianc}.
    
    \item[V.]  We then handle the remaining term $A_5$, corresponding to the bias, and show that it admits the rate $h^{d+1}_{G}$: see \eqref{rate:bias}.

    \item[VI.] Combining the results from steps (I-IV) we conclude the result of the theorem as stated.
\end{itemize}

\medskip

Before we proceed, we first introduce some additional notation, necessary due to the complexity of the expressions. For each subject $i$, $1 \leq i \leq n$, and pair $(s,t) \in \bigtriangleup$ we define $d^{\ast} = \frac{1}{2}(d+1)(d+2)$ vectors $T_{i,(p,q)}(s,t)$ corresponding to all possible pairs $(p,q)$ such that $0 \leq p+q \leq d$ in the following form
\begin{align}\nonumber
    T_{i,(p,q)}(s,t) &= \left[ \left( \frac{T_{ij}-s}{h_G} \right)^{p} \left( \frac{T_{ik}-t}{h_G} \right)^{q}  \right]_{1 \leq k < j \leq r}.
\end{align}
We then bind the  $d^{\ast}$ vectors $T_{i,(p,q)}(s,t)$ column-wise to obtain $T_{i}(s,t)$. And finally we bind the matrices $T_{i}(s,t)$ row-wise to obtain the design matrix $\mathbb{T}(s,t)$. We also define the diagonal weight matrix $\mathbf{W}(s,t)$ in the form
\begin{align}\nonumber
    \mathbf{W}(s,t) = \mathrm{diag}\left\{ K_{H_G}\left(T_{(1,1)}(s,t)\right)\right\},
\end{align}
where $T_{(1,1)}(s,t)$ is the column of the design matrix $\mathbb{T}(s,t)$ corresponding to the pair $(p,q) = (1,1)$.

We can now proceed with steps (I-VI), as announced:

\medskip
\noindent\textsc{I. Closed form expression.}\\ 
It is straightforward to show that the local least square problem \eqref{eq:local surface reg}
admits the following closed-form solution
\begin{eqnarray}\label{a_0}
( \widehat{\gamma}_{p,q}(s,t))_{0 \leq p+q\leq d}^T = \left( h_G^{p+q}\widehat{\partial^{p+q}_{s^pt^q} G}(s,t) \right)_{0 \leq p+q\leq d}^T=  \left(\mathbb{T}_{(s,t)}^{T}\mathbf{W}_{(s,t)}\mathbb{T}_{(s,t)}\right)^{-1}\mathbb{T}_{(s,t)}^{T}\mathbf{W}_{(s,t)} \mathbb{Y}.
\end{eqnarray}
The equation \eqref{a_0} in turn can be reduced to 
\begin{eqnarray}\label{Ghat=A*S}
\label{eq:G = AS}
\left[
\begin{array}{c}
    \widehat{G}(s,t)   \\
   h_G \widehat{\partial_s G}(s,t)  \\
    h_G \widehat{\partial_t G}(s,t) \\
    \vdots\\
    h^d_G \widehat{\partial_{s^d} G}(s,t)\\
    h^d_G \widehat{\partial_{t^d} G}(s,t)
\end{array}
\right]_{d^{\ast} \times 1}  &=&  
\left[
\begin{array}{cccccc}
\mathcal{A}_{0,0}  & \mathcal{A}_{1,0}  &  \mathcal{A}_{0,1} & \ldots & \mathcal{A}_{d,0}  &  \mathcal{A}_{0,d}  \\
 \mathcal{A}_{1,0} & \mathcal{A}_{2,0}  & \mathcal{A}_{1,1} & \ldots &  \mathcal{A}_{d+1,0} &   \mathcal{A}_{1,d}\\
 \mathcal{A}_{0,1} & \mathcal{A}_{1,1}  & \mathcal{A}_{0,2} & \ldots  &  \mathcal{A}_{0,d+1}  &   \mathcal{A}_{d,1}\\
 \vdots & \vdots & \vdots & \ddots & \vdots & \vdots\\
   \mathcal{A}_{d,0}   &  \mathcal{A}_{d+1,0}& \mathcal{A}_{d,1} & \ldots & \mathcal{A}_{2d^{\ast}-2,0} & \vdots\\
   \mathcal{A}_{0,d}  & \mathcal{A}_{1,d}  & \mathcal{A}_{0,d+1}  & \ldots & \ldots &  \mathcal{A}_{0,2d^{\ast}-2}
\end{array} 
\right]^{-1}_{ d^{\ast}\times d^{\ast} }
\left[
\begin{array}{c}
    S_{0,0}   \\
    S_{1,0}  \\
     S_{0,1} \\
     \vdots \\
      S_{d,0}  \\
     S_{0,d}
\end{array}
\right]_{d^{\ast} \times 1}.
\end{eqnarray}
This gives the following closed for expressions for $\widehat G$ and $\widehat{\partial G}$:
\begin{equation}\label{eq : explicit expression Ghat and partialGhat}
\begin{split}
\widehat{G}(s,t) &=\frac{1}{\mathcal{D}(s,t)}
\left[
\begin{array}{cccccc}
 \mathcal{A}^{0,0}  & \mathcal{A}^{1,0}  &  \mathcal{A}^{0,1} & \ldots & \mathcal{A}^{d,0}  &  \mathcal{A}^{0,d} 
\end{array}
\right] \left[
\begin{array}{cccccc}
    S_{0,0}   &
    S_{1,0}  &
     S_{0,1} &
     \vdots &
     S_{d,0}  &
     S_{0,d}
\end{array}
\right]^T
,\\ 
h_G\widehat{\partial_s G}(s,t) &= \frac{1}{\mathcal{D}(s,t)}
\left[
\begin{array}{cccccc}
\mathcal{A}^{1,0}  & \mathcal{A}^{2,0}  &  \mathcal{A}^{1,1} & \ldots & \mathcal{A}^{d+1,0}  &  \mathcal{A}^{1,d} 
\end{array}
\right] \left[
\begin{array}{cccccc}
    S_{0,0}   &
    S_{1,0}  &
     S_{0,1} &
     \vdots &
     S_{d,0}  &
     S_{0,d}
\end{array}
\right]^T
,
\end{split}
\end{equation}
where  $\mathcal{A}^{p,q}$ are the arrays of the inverse matrix appearing in \eqref{eq:G = AS},
\begin{eqnarray*}\label{partialGhat=A*S}
 \mathcal{A}_{p,q} =\frac{1}{nh^{2}_G} \sum_{i=1}^{n} \frac{2}{r(r-1)}\sum_{1\leq k < j \leq r} W\left( \frac{T_{ij}-s}{h_G}\right) W\left( \frac{T_{ik}-t}{h_G}\right) \left( \frac{T_{ij}-s}{h_G} \right)^{p} \left( \frac{T_{ik}-t}{h_G} \right)^{q} ,\quad 0 \leq p+q \leq 2d^{\ast}-2,
\end{eqnarray*}
\begin{align}\nonumber
 S_{p,q} &= \frac{1}{nh^2_G}\sum_{i=1}^{n}\frac{2}{r(r-1)}\sum_{1\leq k<j \leq r} \left[ W\left( \frac{T_{ij}-s}{h_G}\right) W\left( \frac{T_{ij}-s}{h_G}\right)
  \left(\frac{T_{ij}-s}{h_G}\right)^p\left(\frac{T_{ik}-t}{h_G}\right)^q Y_{ij}Y_{ik}\right], \quad 0 \leq p+q \leq d,
\end{align}
and
\begin{align}\label{eq:determinant}
\mathcal{D}(s,t) &= \det \left( \left[
\begin{array}{cccccc}
\mathcal{A}_{0,0}  & \mathcal{A}_{1,0}  &  \mathcal{A}_{0,1} & \ldots & \mathcal{A}_{d,0}  &  \mathcal{A}_{0,d}  \\
 \mathcal{A}_{1,0} & \mathcal{A}_{2,0}  & \mathcal{A}_{1,1} & \ldots &  \mathcal{A}_{d+1,0} &   \mathcal{A}_{1,d}\\
 \mathcal{A}_{0,1} & \mathcal{A}_{1,1}  & \mathcal{A}_{0,2} & \ldots  &  \mathcal{A}_{0,d+1}  &   \mathcal{A}_{d,1}\\
 \vdots & \vdots & \vdots & \ddots & \vdots & \vdots\\
   \mathcal{A}_{d,0}   &  \mathcal{A}_{d+1,0}& \mathcal{A}_{d,1} & \ldots & \mathcal{A}_{2d^{\ast}-2,0} & \vdots\\
   \mathcal{A}_{0,d}  & \mathcal{A}_{1,d}  & \mathcal{A}_{0,d+1}  & \ldots & \ldots &  \mathcal{A}_{0,2d^{\ast}-2}
\end{array}
\right]\right).
\end{align}
These expressions allow us to write out \eqref{eq : explicit expression Ghat and partialGhat} more explicitly and concisely as follows:
\begin{equation}
\begin{split}\label{eq : MORE explicit expression Ghat and partialGhat}
\widehat{G}(s, t) &= \frac{\sum_{ 0 \leq p+q \leq d}\mathcal{A}^{p,q}(s,t) S_{p,q}(s,t)}{\mathcal{D}(s,t)} \\
\widehat{\partial_s G}(s, t) 
 &= \frac{\sum_{0 \leq p+q \leq d}\mathcal{A}^{1+p,q}(s,t) S_{p,q}(s,t)}{\mathcal{D}(s,t)} 
\end{split}
\end{equation}

In the following we drop the index $n$ for the sake of notational brevity. 

\medskip

\noindent\textsc{II. Decomposition of the estimation error.}\\ 
According to \eqref{a_0}, and observing that
\begin{eqnarray*}
\begin{bmatrix}
         G(s,t)\\
         h_G \partial_s G(s,t)\\
         h_G \partial_t G(s,t)\\
          \vdots\\
         h_G^d \partial_{s^d} G(s,t)\\
         h_G^d \partial_{t^d}(s,t)
         \end{bmatrix} =  \left(\mathbb{T}_{(s,t)}^{T}\mathbf{W}_{(s,t)}\mathbb{T}_{(s,t)}\right)^{-1}\left(\mathbb{T}_{(s,t)}^{T}\mathbf{W}_{(s,t)}\mathbb{T}_{(s,t)}\right)\begin{bmatrix}
          G(s,t)\\
         h_G \partial_s G(s,t)\\
         h_G \partial_t G(s,t)\\
          \vdots\\
         h_G^d \partial_{s^d} G(s,t)\\
         h_G^d \partial_{t^d}(s,t)
         \end{bmatrix},
\end{eqnarray*}
we obtain:
\begin{align}\label{Ghat - G : decomposition}
 \underset{0 \leq s \leq t \leq 1}{\mathrm{sup}}\left|\widehat{G}(s, t) - G(s, t)\right|
 &= \underset{0 \leq s \leq t \leq 1}{\mathrm{sup}}\left|\frac{\sum_{ 0 \leq p+q \leq d}\mathcal{A}^{p,q}(s,t) \widetilde{S}_{p,q}(s,t)}{\mathcal{D}(s,t)} \right|,
 \end{align}
 \begin{align}\label{partialGhat - partialG : decomposition}
 \underset{0 \leq s \leq t \leq 1}{\mathrm{sup}}\left|h_G\widehat{\partial_s G}(s, t) - h_G\partial_s G(s, t)\right|
 &=  \underset{0 \leq s \leq t \leq 1}{\mathrm{sup}}\left|\frac{\sum_{ 0 \leq p+q \leq d}\mathcal{A}^{p+1,q}(s,t) \widetilde{S}_{p,q}(s,t)}{\mathcal{D}(s,t)} \right| ,
 \end{align}
 where
 \begin{align*}
 \widetilde{S}_{p,q}(s,t) &=S_{p,q}(s,t)- \sum_{0 \leq p'+q'\leq d}  \mathcal{A}_{p'+p,q'+q} h^{p'+q'}_G    \partial G_{s^{p'}t^{q'}}   (s,t) \\ \nonumber
 &=: S_{p,q}(s,t) - \Lambda_{p,q}(s,t), \quad 0 \leq p+q \leq d.
 \end{align*}

 \medskip

\noindent \textsc{III. Reduction to a bias-variance decomposition}\\ 
We now investigate the different terms appearing in \eqref{Ghat - G : decomposition} and \eqref{partialGhat - partialG : decomposition} separately. \\
{First, we observe that $\mathcal{A}_{p,q}$, for $0 \leq p+q \leq d$, and hence the determinants \eqref{eq:determinant} are uniformly convergent. Moreover, for each $s,t$, \eqref{eq:determinant} is bounded away from zero. See equation (3.13) of \citet{Fan1996} and the discussion thereof. Hence, the rate of convergence of the vector $$\left[
\begin{array}{cccccc}
    \widehat{G}(s,t)  &
   h_G \widehat{\partial_s G}(s,t) &
    h_G \widehat{\partial_t G}(s,t) &
    \ldots&
    h_G \widehat{\partial_{s^d} G}(s,t)&
    h^d_G \widehat{\partial_{t^d} G}(s,t)
\end{array}
\right]^T $$  is determined by the rest of the terms i.e. $\widetilde{S}_{p,q}$, $0 \leq p+q \leq d$}. All these terms can be analysed in the same way, and we will only study $\widetilde{S}_{0,0}$ in full detail.
\begin{eqnarray*}
  \widetilde{S}_{0,0}(s,t) &=&\frac{1}{nh^2_G}\sum_{i=1}^{n}\frac{2}{r(r-1)}\sum_{1\leq k<j \leq r} \left[ W\left( \frac{T_{ij}-s}{h_G}\right) W\left( \frac{T_{ik}-t}{h_G}\right)Y_{ij}Y_{ik}\right] - \Lambda_{0,0}(s,t)\\ \nonumber
  &=& \frac{1}{nh^2_G}\sum_{i=1}^{n}\frac{2}{r(r-1)}\sum_{1\leq k<j \leq r} \left[ W\left( \frac{T_{ij}-s}{h_G}\right) W\left( \frac{T_{ik}-t}{h_G}\right)\left(X_{ij} + U_{ij}\right)\left(X_{ik} + U_{ik}\right)\right] \\ \nonumber
  &&- \Lambda_{0,0}(s,t)\\ \nonumber
  &=&\frac{1}{nh^2_G}\sum_{i=1}^{n}\frac{2}{r(r-1)}\sum_{1\leq k<j \leq r} \left[ W\left( \frac{T_{ij}-s}{h_G}\right) W\left( \frac{T_{ik}-t}{h_G}\right) U_{ij}U_{ik}\right]\\ \nonumber
  &&+\frac{1}{nh^2_G}\sum_{i=1}^{n}\frac{2}{r(r-1)}\sum_{1\leq k<j \leq r} \left[ W\left( \frac{T_{ij}-s}{h_G}\right) W\left( \frac{T_{ik}-t}{h_G}\right) X_{ij}U_{ik}\right]\\ \nonumber
  &&+\frac{1}{nh^2_G}\sum_{i=1}^{n}\frac{2}{r(r-1)}\sum_{1\leq k<j \leq r} \left[ W\left( \frac{T_{ij}-s}{h_G}\right) W\left( \frac{T_{ik}-t}{h_G}\right) U_{ij}X_{ik}\right]\\ \nonumber
  &&+\frac{1}{nh^2_G}\sum_{i=1}^{n}\frac{2}{r(r-1)}\sum_{1\leq k<j \leq r} \left[ W\left( \frac{T_{ij}-s}{h_G}\right) W\left( \frac{T_{ik}-t}{h_G}\right) \left(X_{ij}X_{ik}-G\left(T_{ij},T_{ik} \right)\right)\right]\\ \nonumber
  &&+\frac{1}{nh^2_G}\sum_{i=1}^{n}\frac{2}{r(r-1)}\sum_{1\leq k<j \leq r} \left[ W\left( \frac{T_{ij}-s}{h_G}\right) W\left( \frac{T_{ik}-t}{h_G}\right) G\left(T_{ij},T_{ik} \right)\right]- \Lambda_{0,0}(s,t)\\ \label{eq:A1--A5}
  &=:& A_1+A_2+A_3+A_4+A_5.
\end{eqnarray*}
The expressions $A_1-A_4$ collectively represent the variance term, and will be analysed jointly next, in step IV. The term $A5$ represents the bias term and will be analysed separately in step V via a Taylor expansion.
\medskip

\noindent \textsc{IV. The variance terms ($A1-A4$).}\\
We remark that each of the terms $A_1-A_4$ can be written in the same general form, namely
\begin{align*}
% \nonumber 
 % &\frac{1}{nh^2_G}\sum_{i=1}^{n}\frac{2}{r(r-1)}\sum_{1\leq k<j \leq r} \left[ W\left( \frac{T_{ij}-s}{h_G}\right) W\left( \frac{T_{ik}-t}{h_G}\right) Z_{ijk}\right] \\  \nonumber 
&  \frac{1}{nh^2_G}\sum_{i=1}^{n}\frac{2}{r(r-1)}\sum_{1\leq k<j \leq r} \left[ W\left( \frac{T_{ij}-s}{h_G}\right) W\left( \frac{T_{ik}-t}{h_G}\right) Z_{ijk}\right] \\ 
% \label{sum:Z_ijk:c:c} 
& \hspace{4.1cm}\times \left[\mathbb{I}\left(  \left(T_{ij}, T_{ik}\right) \in [s-h_G , s+h_G]^{c} \times [t-h_G , t+h_G]^{c} \right)\right.\\  
% \label{sum:Z_ijk:c:} 
& \hspace{4.5cm}+ \left.\mathbb{I}\left(  \left(T_{ij}, T_{ik}\right) \in [s-h_G , s+h_G]^{c} \times [t-h_G , t+h_G] \right)\right.\\  \label{sum:Z_ijk::c} 
& \hspace{4.5cm}+ \left.\mathbb{I}\left(  \left(T_{ij}, T_{ik}\right) \in [s-h_G , s+h_G] \times [t-h_G , t+h_G]^{c} \right)\right.\\ 
 % \label{sum:Z_ijk} 
& \hspace{4.5cm}+ \left.\mathbb{I}\left(  \left(T_{ij}, T_{ik}\right) \in [s-h_G , s+h_G] \times [t-h_G , t+h_G] \right)\right]\\
% \nonumber
&=: \bar{Z}_{1,1} (s,t) +\bar{Z}_{1,0} (s,t)+\bar{Z}_{0,1} (s,t)+\bar{Z}_{0,0} (s,t),
\end{align*}
where $Z_{ijk}$ are mean zero random variables in all cases. We can thus analyse this general expression to account for all four terms $A_1-A_4$.   For 
% \eqref{sum:Z_ijk:c:c},
$\bar{Z}_{1,1} (s,t) $
observe that
\begin{equation}
\label{Z_11}
\begin{split}
\bar{Z}_{1,1} (s,t) &= \frac{1}{nh^2_G}\sum_{i=1}^{n}\frac{2}{r(r-1)}\sum_{1\leq k<j \leq r} \left[ W\left( \frac{T_{ij}-s}{h_G}\right) W\left( \frac{T_{ik}-t}{h_G}\right) Z_{ijk}\right] \\  
 &\hspace{4.1cm}\times \mathbb{I}\left(  \left(T_{ij}, T_{ik}\right) \in [s-h_G , s+h_G]^{c} \times [t-h_G , t+h_G]^{c} \right)\\ 
 &\leq  W^{2}\left( 1^{+} \right)\frac{1}{nh^2_G}\sum_{i=1}^{n}\frac{2}{r(r-1)}\sum_{1\leq k<j \leq r} \vert Z_{ijk} \vert \\ 
&= O\left( h_G^{-2} \right) W^{2}\left( 1^{+} \right),\;\;\;\; \mathrm{\;a.s\;uniformly\; on\;\;}  0 \leq t \leq s \leq 1 \\
&=  O\left( h_G^{d+1} \right) ,\;\;\;\; \mathrm{\;a.s\;uniformly\; on\;\;}  0 \leq t \leq s  \leq 1.
\end{split}
\end{equation}
For $\bar{Z}_{1,0} (s,t) $ (and similarly for $\bar{Z}_{0,1} (s,t) $ we have that:
% For \eqref{sum:Z_ijk:c:} (similarly \eqref{sum:Z_ijk::c}) we have
\begin{equation}
\begin{split}
\label{Z_10}
\bar{Z}_{1,0} (s,t) &= \frac{1}{nh^2_G}\sum_{i=1}^{n}\frac{2}{r(r-1)}\sum_{1\leq k<j \leq r} \left[ W\left( \frac{T_{ij}-s}{h_G}\right) W\left( \frac{T_{ik}-t}{h_G}\right) Z_{ijk}\right] \\ 
&\hspace{4.1cm} \times \mathbb{I}\left(  \left(T_{ij}, T_{ik}\right) \in [s-h_G , s+h_G]^{c} \times [t-h_G , t+h_G] \right)\\
&\leq  \left(\int \left\vert W^{\dagger}\left( u\right)\right\vert du \right)  W\left( 1^{+} \right)   \frac{1}{nh^2_G}\sum_{i=1}^{n}\frac{2}{r(r-1)}\sum_{1\leq k<j \leq r} \vert Z_{ijk} \vert\\
&= O\left( h_G^{-2} \right) \left(\int \left\vert W^{\dagger}\left( u\right)\right\vert du \right)  W\left( 1^{+} \right),\;\;\;\; \mathrm{\;a.s\;uniformly\; on\;\;}  0 \leq t \leq s  \leq 1 \\
&=  O\left( h_G^{d+1} \right) ,\;\;\;\; \mathrm{\;a.s\;uniformly\; on\;\;}  0 \leq t \leq s  \leq 1, 
\end{split}
\end{equation}
where $W^{\dagger}$ indicates the Fourier transform of the kernel function $W$. For $\bar{Z}_{0,0} (s,t) $
% \eqref{sum:Z_ijk} 
we have 
\begin{align*}
    \bar{Z}_{0,0} (s,t)&=
 \frac{1}{n}\sum_{i=1}^{n}\frac{2}{r(r-1)}\sum_{1\leq k<j \leq r} \left[ Z_{ijk} \iint e^{-\mymathbb{i}us-\mymathbb{i}vt+\mymathbb{i}uT_{ij}+\mymathbb{i}vT_{ik}}W^{ \dagger} (h_Gu)W^{ \dagger} (h_Gv)dudv\right]  \\  \nonumber &\hspace{4cm}\times \mathbb{I}\left(  \left(T_{ij}, T_{ik}\right) \in [s-h_G , s+h_G] \times [t-h_G , t+h_G] \right).
\end{align*}
Observing that $\mathbb{E}\bar{Z}_{0,0}(s,t) = 0 $, for all $0 \leq t \leq s \leq 1$, we have
 \begin{align}
  \nonumber
&\vert \bar{Z}_{0,0} (s,t)- \mathbb{E}\bar{Z}_{0,0}(s,t)\vert \\\nonumber \quad
 &\leq \left(\int \left\vert W^{\dagger}\left( h_G u\right)\right\vert du \right)^{2}
 \times
 %\underset{  0<t<s<1  }{\mathrm{sup}}
 \frac{1}{n}\sum_{i=1}^{n}\frac{2}{r(r-1)}\sum_{1\leq k<j \leq r}\left\vert Z_{ijk} \right\vert \mathbb{I}\left(  \left(T_{ij}, T_{ik}\right) \in [s-h_G , s+h_G] \times [t-h_G , t+h_G] \right)
 \\ \label{rate*h^4}
 & =  O\left( h^{-2}_{G}\right)
  \times
% \underset{  0<t<s<1  }{\mathrm{sup}}
 \frac{1}{n}\sum_{i=1}^{n}\frac{2}{r(r-1)}\sum_{1\leq k<j \leq r}\left\vert Z_{ijk} \right\vert \mathbb{I}\left(  \left(T_{ij}, T_{ik}\right) \in [s-h_G , s+h_G] \times [t-h_G , t+h_G] \right).
 \end{align}
The remaining argument for $\bar{Z}_{0,0} (s,t) $
% \eqref{sum:Z_ijk}
is similar to the proof of Lemma (8.2.5) of \citet{hsing_theoretical_2015}. In more detail, for the summation term appearing in \eqref{rate*h^4} we have
 \begin{eqnarray*}
  &&\frac{1}{n}\sum_{i=1}^{n}\frac{2}{r(r-1)}\sum_{1\leq k<j \leq r}\left\vert Z_{ijk} \right\vert \mathbb{I}\left(  \left(T_{ij}, T_{ik}\right) \in [s-h_G , s+h_G] \times [t-h_G , t+h_G] \right)\\
  &=& \frac{1}{n}\sum_{i=1}^{n}\frac{2}{r(r-1)}\sum_{1\leq k<j \leq r}  
  \left\vert Z_{ijk}  \right\vert    \mathbb{I}\left(   \left\vert Z_{ijk}  \right\vert  \geq   Q_n \cup \left\vert Z_{ijk}  \right\vert  < Q_n   \right)\\
  && \hspace{4cm} \times \mathbb{I}\left(  \left(T_{ij}, T_{ik}\right) \in [s-h_G , s+h_G] \times [t-h_G , t+h_G] \right)\\
  &=& \underbrace{\frac{1}{n}\sum_{i=1}^{n}\frac{2}{r(r-1)}\sum_{1\leq k<j \leq r}  
  \left\vert Z_{ijk}  \right\vert    \mathbb{I}\left(   \left\vert Z_{ijk}  \right\vert  \geq   Q_n    \right)\times \mathbb{I}\left(  \left(T_{ij}, T_{ik}\right) \in [s-h_G , s+h_G] \times [t-h_G , t+h_G] \right)}_{:=B_1}
   \\ &&\qquad + \:
  \underbrace{\frac{1}{n}\sum_{i=1}^{n}\frac{2}{r(r-1)}\sum_{1\leq k<j \leq r}  
  \left\vert Z_{ijk}  \right\vert    \mathbb{I}\left(  \cup \left\vert Z_{ijk}  \right\vert  < Q_n   \right)\times \mathbb{I}\left(  \left(T_{ij}, T_{ik}\right) \in [s-h_G , s+h_G] \times [t-h_G , t+h_G] \right)}_{:=B_2}.
  % &=:& B_1 + B_2.
  \end{eqnarray*}
  We analyse the two terms $B_1,B_2$ separately.
Employing \ref{cond:error} and Lemma \ref{lemma:sup}, and choosing $Q_n$ in such a way that $$
\left[ \frac{\mathrm{log}n}{n}\left( h_G^4 + \frac{h_G^3}{r(n)}+ \frac{h_G^2}{r^2(n)} \right)\right]^{-1/2} Q_n^{1-\alpha} = O(1),
$$ we may obtain that $B_1 = O \left( \left[ \frac{\mathrm{log}n}{n}\left( h_G^4 + \frac{h_G^3}{r(n)}+ \frac{h_G^2}{r^2(n)} \right)\right]^{1/2} \right)$, almost surely uniformly. Indeed:
\begin{eqnarray}
%\label{B_1}
\nonumber
 B_1  & = & \frac{1}{n}\sum_{i=1}^{n}\frac{2}{r(r-1)}\sum_{1\leq k<j \leq r}
  \left\vert Z_{ijk}  \right\vert ^{1-\alpha+\alpha}   \mathbb{I}\left(   \left\vert Z_{ijk}  \right\vert  \geq   Q_n  \right)  \mathbb{I}\left(  \left(T_{ij}, T_{ik}\right) \in [s-h_G , s+h_G] \times [t-h_G , t+h_G] \right)\\ \nonumber
  & \leq & \frac{1}{n}\sum_{i=1}^{n}\frac{2}{r(r-1)}\sum_{1\leq k<j \leq r}
  \left\vert Z_{ijk}  \right\vert ^{\alpha}  Q_n^{1-\alpha}   
   = O(1) Q_n^{1-\alpha} 
  =   
  O\left( \left[ \frac{\mathrm{log}n}{n}\left( h_G^4 + \frac{h_G^3}{r(n)}+ \frac{h_G^2}{r^2(n)} \right)\right]^{1/2} \right)
  \end{eqnarray}
  almost surely, uniformly on $0<s<t<1.$
For $B_2$, first, define
  \begin{eqnarray}
  \nonumber
  B_2 &=& \frac{1}{n}\sum_{i=1}^{n}\frac{2}{r(r-1)}\sum_{1\leq k<j \leq r}  
  \left\vert Z_{ijk}  \right\vert    \mathbb{I}\left(  \left\vert Z_{ijk}  \right\vert  \leq Q_n   \right)
 \times \mathbb{I}\left(  \left(T_{ij}, T_{ik}\right) \in [s-h_G , s+h_G] \times [t-h_G , t+h_G] \right)\\ \nonumber
 &=:& \frac{1}{n}\sum_{i=1}^{n}\frac{2}{r(r-1)}\sum_{1\leq k<j \leq r}  \mathcal{Z}_{ijk}(s,t).
  \end{eqnarray}
We then apply Bennett's concentration inequality (see \citep{boucheron_concentration_2013}) to complete the proof of this part. We  obtain a uniform upper bound for  $\mathrm{Var\left( \frac{2}{r(r-1)}\sum_{1\leq k<j \leq r}  \mathcal{Z}_{ijk}(s,t) \right) }$, $i=1,2,\ldots,n$, in the following way
\begin{eqnarray}
\nonumber
 \mathrm{Var\left( \frac{2}{r(r-1)}\sum_{1\leq k<j \leq r}  \mathcal{Z}_{ijk}(s,t) \right) } 
  &=& \left(\frac{2}{r(r-1)}\right)^{2}
 \sum_{1\leq k_1<j_1 \leq r}  \sum_{1\leq k_2<j_2 \leq r} \mathrm{Cov}\left( \mathcal{Z}_{ij_1k_1}(s,t) , \mathcal{Z}_{ij_2k_2}(s,t) \right)\\ \label{var:Bern}
 & \leq & c \left( h_G^4 + \frac{h_G^3}{r(n)}+ \frac{h_G^2}{r^2(n)} \right),
\end{eqnarray}
for some positive constant $c$ which  depends neither on $(s,t)$ nor on $i$. Note that inequality \eqref{var:Bern} is a direct consequence of conditions \ref{cond:design}  and \ref{cond:error} and Lemma \ref{lemma:sup}. 
Finally, applying Bennett's inequality and choosing $Q_n = \left(\frac{\mathrm{log}n}{n}\right)^{-1/2}\left( h_G^4 + \frac{h_G^3}{r(n)}+ \frac{h_G^2}{r^2(n)} \right)^{1/2}$ we have, for any positive number $\eta$,
\begin{align}
\nonumber
 &\mathbb{P}\left( \frac{1}{n}\sum_{i=1}^{n}\frac{2}{r(r-1)}\sum_{1\leq k<j \leq r}  \mathcal{Z}_{ijk}(s,t) \geq \eta \left[ \frac{\mathrm{log}n}{n}\left( h_G^4 + \frac{h_G^3}{r(n)}+ \frac{h_G^2}{r^2(n)} \right) \right]^{1/2} \right) \\ \nonumber
 &\leq\mathrm{exp}\left\{ -\frac{\eta^2 n^2 \left[ \frac{\mathrm{log}n}{n}\left( h_G^4 + \frac{h_G^3}{r(n)}+ \frac{h_G^2}{r^2(n)} \right) \right]}{2nc\left( h_G^4 + \frac{h_G^3}{r(n)}+ \frac{h_G^2}{r^2(n)} \right)+\frac{2}{3}\eta n\left( h_G^4 + \frac{h_G^3}{r(n)}+ \frac{h_G^2}{r^2(n)} \right)} \right\} \\\label{sum:Bern}
 &= \mathrm{exp}\left\{-\frac{\eta^2   \mathrm{log}n }{ 2c+\frac{2}{3}\eta} \right\} 
 = n^{-\frac{\eta^2 }{ 2c+\frac{2}{3}\eta} }\;, \;\;\;\; \forall 0 \leq t \leq s \leq 1.
\end{align}
Choosing $\eta$ large enough we conclude summability of \eqref{sum:Bern}. This result  combined with the Borel-Cantelli lemma completes the proof of this part. In other words we conclude there exists a subset $\Omega_0 \subset \Omega$ of full probability measure such that for each $\omega \in \Omega_0$ there exists $n_0 = n_0(\omega) $ with 
\begin{align}\label{rate:varianc}
 \frac{1}{n}\sum_{i=1}^{n}\frac{2}{r(r-1)}\sum_{1\leq k<j \leq r}  \mathcal{Z}_{ijk}(s,t) \leq \eta \left[ \frac{\mathrm{log}n}{n}\left( h_G^4 + \frac{h_G^3}{r(n)}+ \frac{h_G^2}{r^2(n)} \right) \right]^{1/2},\;\;\;\; n \geq n_0.
\end{align}
This concludes the analysis of the variance term ($A_1--A_4$).

\medskip

\noindent\textsc{V. The bias term ($A_5$).}\\
We now turn to the bias term $A_5$ and investigate its convergence. Observe that
\begin{align}
\nonumber
 A_5 &= \frac{1}{nh^2_G}\sum_{i=1}^{n}\frac{2}{r(r-1)}\sum_{1\leq k<j \leq r} \left[ W\left( \frac{T_{ij}-s}{h_G}\right) W\left( \frac{T_{ik}-t}{h_G}\right) G\left(T_{ij},T_{ik} \right)\right]- \Lambda_{0,0}(s,t)\\ \nonumber
 &=\frac{1}{nh^2_G}\sum_{i=1}^{n}\frac{2}{r(r-1)}\sum_{1\leq k<j \leq r} W\left( \frac{T_{ij}-s}{h_G}\right) W\left( \frac{T_{ik}-t}{h_G}\right)\left[  G\left(T_{ij},T_{ik} \right) - G(s,t) -\left( T_{ij}-s\right) \partial_{s} G(s,t) \right.\\ \nonumber
 & \hspace{4cm}
 \left.
  - \left( T_{ik}-t\right) \partial_{t} G(s,t) - \dots - \left( T_{ik}-t\right)^d \partial_{t^d} G(s,t) - \left( T_{ij}-s\right)^d\partial_{s^d} G(s,t) \right]
 \\ \nonumber
 &= \frac{1}{nh^2_G}\sum_{i=1}^{n}\frac{2}{r(r-1)}\sum_{1\leq k<j \leq r} W\left( \frac{T_{ij}-s}{h_G}\right) W\left( \frac{T_{ik}-t}{h_G}\right) \left[  \sum_{m+m'=d+1}\frac{1}{m!m'!}\partial_{s^mt^{m'}} G(s,t)\left( T_{ij}-s\right)^m\left( T_{ik}-t\right)^{m'}  \right]\\ \nonumber   
 & = \frac{h_G^{d+1}}{nh^2_G}\sum_{i=1}^{n}\frac{2}{r(r-1)}\sum_{1\leq k<j \leq r} W\left( \frac{T_{ij}-s}{h_G}\right) W\left( \frac{T_{ik}-t}{h_G}\right) \left[  \sum_{m+m'=d+1}\frac{1}{m!m'!}\partial_{s^mt^{m'}} G(s,t)\left(\frac{ T_{ij}-s}{h_G}\right)^m\left(\frac{ T_{ik}-t}{h_G}\right)^{m'}  \right]
 \\ \label{rate:bias}
 & =  O\left( h^{d+1}_G \right),\;\;\;\; \mathrm{a.s\;uniformly \; on \;} 0 \leq t \leq s \leq 1.
\end{align}
The last expression is a consequence of the smoothness condition \ref{cond:smoothness}. 

\medskip

\noindent\textsc{VI. Conclusion.}\\
Taking together, the relations \eqref{rate:varianc} and \eqref{rate:bias} complete the proof for the term $\widetilde{S}_{0,0}$ appearing in decomposition \eqref{Ghat - G : decomposition}.  However, as remarked, it is immediate to see that the remaining terms terms $\widetilde{S}_{p,q}$ are structurally same and can be handled with the same approach. 
This together with decompositions \eqref{Ghat - G : decomposition} and \eqref{partialGhat - partialG : decomposition} entail the rate $\mathcal{Q}(n)$ for $\widehat{G}(s,t)$  and $h_G \widehat{\partial_s G}(s,t)$ (in general  for the components of the vector appearing on the left hand side of \eqref{Ghat=A*S}) as desired.
\end{proof}
The following Lemma investigates the rate preserving property of the reciprocal function and paves the way of proving the results in Theorem \ref{thm : mu, sigma}.
\begin{lemma} \label{consistency:m^-1}[Consistency and rate of convergence of  $\widehat{m}^{-1}(\cdot)$] \label{consistency:inverse} Assume conditions \ref{cond:smoothness}, \ref{cond:design} and \ref{cond:error} hold for $\rho > 2$ and {$m(0) \neq 0$}.  Let $\widehat{m}(\cdot)$ be the estimators defined in \eqref{eq:m,m'} and set $\widehat{m}_{\ast}(t) = \widehat{m}(t) \mathbb{I}\left( \widehat{m}(t)\neq 0\right) + \mathbb{I}\left( \widehat{m}(t) = 0\right)$. Then with probability $1$:
\begin{eqnarray}
\underset{t \in {[0,1]}}{\mathrm{sup}} \left \vert \frac{1}{ \widehat{m}_{\ast}(t) } - \frac{1}{m(t)} \right \vert  &=&  O\left( \mathcal{R}(n) \right), 
\end{eqnarray}
where  $ \mathcal{R}(n)$ equals $\left[ h^{-2}_{m} \frac{\mathrm{log}n}{n}\left(h^2_{m} + \frac{h_{m}}{r}  \right) \right]^{1/2}+ h^{d+1}_{m} $ and reduces to $\left( \frac{\log n}{n} \right)^{1/2}$ under dense sampling regime, i.e. when $r(n) \geq M_n$, for some increasing sequence $M_n\uparrow\infty$,   and taking $M_n^{-1} \lesssim h_{m}  \simeq \left( \frac{\mathrm{log}n}{n} \right)^{1/{(2(1+d))}} $.
\end{lemma}
\begin{proof}
Let $\left(\Omega, \mathcal{F}, \mathbb{P}  \right)$ be the underlying probability space and  $\Omega_0$ be a subset of $\Omega$ with full measure, $\mathbb{P}(\Omega_0) = 1$, on which we have \eqref{eq: cons m}. For $\omega$ in  $\Omega_0$, %fix $\epsilon >0$ and 
choose $n_0$ and $b_0$ such that:
\begin{eqnarray} \nonumber
 \frac{\sup_{0\leq t \leq 1}\left\vert \widehat{m}(t,\omega) - m(t) \right\vert}{\mathcal{R}(n)}  <  b_0,\;\;\; \forall n_0<n,
\end{eqnarray}
or equivalently:
\begin{eqnarray} \label{n_0}
 \frac{\left\vert \widehat{m}(t,\omega) - m(t) \right\vert}{\mathcal{R}(n)} <  b_0,\;\;\; \forall n_0<n,\; \forall t \in [0,1].
\end{eqnarray}
In the following, we use that the mean function is bounded away from zero across the time interval considered, $M^{\ast} :=  \underset{t \in [0,1]}{\inf} \vert m(t) \vert]>0$, if and only if $m(0) \neq 0$, which is satisfied by the assumptions. Moreover, there exists $n_1$ such that:
\begin{eqnarray} \nonumber
\left\vert \vert\widehat{m}(t,\omega)\vert -  \vert m(t) \vert \right\vert  \leq \left\vert \widehat{m}(t,\omega) - m(t) \right\vert < \frac{1}{2} M^{\ast} <   \frac{1}{2} \vert m(t) \vert, \;\;\; n_1 < n,\;\forall t \in [0,1],
\end{eqnarray}
as a result of uniform convergence \eqref{eq: cons m}. Consequently, 
\begin{eqnarray} \nonumber
0 < \frac{1}{2} M^{\ast} \leq \frac{1}{2} \left\vert  m(t) \right\vert < \left\vert \widehat{m}(t,\omega)  \right\vert ,\;\;\; \forall n_1 <n,\; \forall t \in [0,1],
\end{eqnarray}
and hence
\begin{eqnarray}\label{n_1}
 0 < \frac{1}{2} M^{\ast 2} \leq \frac{1}{2} \left\vert  m(t) \right\vert^2 \leq \left \vert  \widehat{m}(t,\omega) m(t) \right \vert ,\;\;\; \forall n_1 <n,\; \forall t \in [0,1].
\end{eqnarray}
Combining \eqref{n_0}, \eqref{n_1} and $\widehat{m}(t,\omega) = \widehat{m}_{\ast}(t,\omega)$  for $\omega \in \Omega_0$ and sufficiently large $n$ ($n_1 < n$) we  observe that:
\begin{eqnarray} \nonumber
\frac{ \left \vert \frac{1}{  \widehat{m}_{\ast}(t,\omega) } - \frac{1}{m(t)} \right \vert}{\mathcal{R}(n)} &=&\frac{1}{\left \vert  \widehat{m}_{\ast}(t,\omega) m(t) \right \vert}\frac{ \left \vert  \widehat{m}_{\ast}(t,\omega)  -m(t) \right \vert}{ \mathcal{R}(n)} \leq \frac{2 }{M^{\ast 2}}b_0, \;\; \max (n_0,n_1) < n  ,\; \forall t \in [0,1].
\end{eqnarray}
That establishes the result.
\end{proof}
We are now ready to establish  Theorem \ref{thm : mu, sigma}.
\begin{proof}[Proof of Theorem \ref{thm : mu, sigma}]
Regarding Lemma \ref{consistency:inverse} and Theorem \ref{thm : m, G}, for  any positive number  $b_2$  there exists  $n_2$ such that with probability 1:
\begin{eqnarray}  \label{n_2:m inverse}
 \frac{\left\vert \widehat{m}^{-1}_{\ast}(t,\omega) - m^{-1}(t) \right\vert}{\mathcal{R}(n)} <  b_2 < \infty,\;\;\; \forall n_2<n,\; \forall t \in [0,1],
\end{eqnarray}
and 
\begin{eqnarray} \label{n_2:m'}
 \frac{\left\vert \widehat{\partial m}(t,\omega) - \partial m(t) \right\vert}{h_m^{-1}\mathcal{R}(n)} <  b_2 < \infty,\;\;\; \forall n_2<n,\; \forall t \in [0,1] .
\end{eqnarray}
We now have:
\begin{eqnarray} \nonumber
&&\frac{\vert \widehat{m}^{-1}_{\ast}(t,\omega) \widehat{\partial m}(t,\omega) - m^{-1}(t) \partial m(t) \vert}{h_m^{-1}\mathcal{R}(n)}\\ \nonumber
&\leq& \frac{\vert \widehat{\partial m}(t,\omega)\vert \vert \widehat{m}^{-1}_{\ast}(t,\omega)  - m^{-1}(t)  \vert}{h_m^{-1}\mathcal{R}(n)} 
 + \frac{\vert m^{-1}(t) \vert \vert  \widehat{\partial m}(t,\omega) -  \partial m(t) \vert}{h_m^{-1}\mathcal{R}(n)} \\ \nonumber
 &\leq& \left(\vert \widehat{\partial m}(t,\omega) - \partial m(t)\vert + \vert \partial m(t) \vert \right)\frac{ \vert \widehat{m}^{-1}_{\ast}(t,\omega)  - m^{-1}(t)  \vert}{h_m^{-1}\mathcal{R}(n)} 
 + \vert m^{-1}(t) \vert \frac{ \vert  \widehat{\partial m}(t,\omega) -  \partial m(t) \vert}{h_m^{-1}\mathcal{R}(n)}, \\ \nonumber
\end{eqnarray}
Uniform consistency of $\widehat{\partial m}(t,\omega)$, uniform boundedness of  $m^{-1}(t)$ and  $\partial m(t)$ on $[0,1]$ together with \eqref{n_2:m inverse}  and \eqref{n_2:m'} yield the desired rate for $\underset{0 \leq t \leq 1}{\sup}\vert \widehat{m}^{-1}_{\ast}(t,\omega) \widehat{\partial m}(t,\omega) - m^{-1}(t) \partial m(t) \vert$. Comparing the definitions of $\widehat{m}^{-1}_{\ast}(\cdot) \widehat{\partial m}(\cdot) $ and $\widehat{\mu}(\cdot)$ justifies \eqref{eq:cons:mu}. 
Similar arguments lead to the uniform rate 
$h_m^{-1}O\left( \mathcal{R}(n) \right)+ h_G^{-1}O\left( \mathcal{Q}(n) \right) $ for either products $\widehat{\mu}(t) \widehat{D}(t)$, $\widehat{\mu}(s) \widehat{G}(s,s)$ or $\widehat{\mu}(s) \widehat{\partial_s }G(s,u)$. This yields \eqref{eq:cons:sigma}. 
\end{proof}
%%%%%%%%%
\subsection*{Proof of the Results of Section \ref{sec:Extension}}
\begin{proof}[Proof of Proposition \ref{prop:system2:GBM}]
We first claim that the unique solution of equation \eqref{eq:timedep:GBM} admits the following form
\begin{eqnarray}\label{closed form:GBM}
X(t)
&=& X(0) \exp \left\{ \int_{0}^{t}\left( \mu(v) -\frac{1}{2}\sigma^2(v) \right)dv \right\}\exp \left\{  \int_{0}^{t} \sigma(v) dB(v)\right\} \\ \nonumber
&=& X(s) \exp \left\{ \int_{s}^{t}\left( \mu(v) -\frac{1}{2}\sigma^2(v) \right)dv + \int_{s}^{t} \sigma(v) dB(v)\right\} 
, \quad 0 \leq s \leq t \leq 1.
\end{eqnarray}
Indeed, the (stochastic) differential of the process $\{X(t)\}_{t \geq 0}$ appearing in \eqref{closed form:GBM} reads
\begin{align}
\nonumber
dX(t) \: = \: & X(0) \exp \left\{ \int_{0}^{t}\left( \mu(v) -\frac{1}{2}\sigma^2(v) \right)dv\right\} \left(\mu(t)dt -\frac{1}{2}\sigma^2(t)dt \right) \exp \left\{  \int_{0}^{t} \sigma(v) dB(v)\right\} \\ \nonumber
&+ X(0) \exp \left\{ \int_{0}^{t}\left( \mu(v) -\frac{1}{2}\sigma^2(v) \right)dv\right\} \exp \left\{  \int_{0}^{t} \sigma(v) dB(v)\right\} \left( \sigma(t)dB(t) + \frac{1}{2}\sigma^2(t)dt\right) \\
\nonumber
= \:& X(0) \exp \left\{ \int_{0}^{t}\left( \mu(v) -\frac{1}{2}\sigma^2(v) \right)dv+\int_{0}^{t} \sigma(v) dB(v)\right\} \left(\mu(t)dt + \sigma(t)dB(t) \right)\\ \label{eq:conj:true}
= \:& X(t) \left(\mu(t)dt + \sigma(t)dB(t) \right).
\end{align}
Equation \eqref{eq:conj:true} confirms the claim. 
%\\
Taking the expectation of both sides of \eqref{closed form:GBM} gives:
\begin{eqnarray} 
\nonumber
m(t) &=&  m(0)  \exp \left\{ \int_{0}^{t}\left( \mu(v) -\frac{1}{2}\sigma^2(v) \right)dv \right\}\exp \left\{ \int_{0}^{t}\left( \frac{1}{2}\sigma^2(v) \right)dv \right\}\\ \label{extension:m:GBM}
 &=&  m(0) \exp \left\{ \int_0^t \mu(v) dv \right\}.
\end{eqnarray}
Since the random variable $   \int_{0}^{t} \sigma(v) dB(v)$  is identical to the random variable $B(\int_{0}^{t} \sigma^2(v) dv)$ i.e Brownian motion at time $\int_{0}^{t} \sigma^2(v) dv$, similar arguments lead to:
\begin{eqnarray} \label{extension:D:GBM}
D(t) &=&  D(0) \exp \left\{ \int_0^t (2\mu(v) + \sigma^2(v)) dv \right\}.
\end{eqnarray}
Equations \eqref{extension:m:GBM} and  \eqref{extension:D:GBM}  confirm the smoothness conditions $m(\cdot), D(\cdot) \in \mathcal{C}^{d+1}([0,1],\R)$.   
%\\
Equation \eqref{closed form:GBM}, in particular,  gives:
\begin{align}
G(s,t) &= G(s,s) \exp \left\{ \int_{s}^{t}\left( \mu(v) -\frac{1}{2}\sigma^2(v) \right)dv \right\} \mathbb{E}\exp \left\{ \int_{s}^{t} \sigma(v) dB(v)\right\}\\ \nonumber
&= G(s,s) \exp \left\{ \int_{s}^{t}\left( \mu(v) -\frac{1}{2}\sigma^2(v) \right)dv \right\} \exp \left\{ \int_{s}^{t} \frac{1}{2}\sigma^2(v) dv \right\}\\ \nonumber
&= D(s) \exp \left\{ \int_{s}^{t}\left( \mu(v)  \right)dv \right\}.
\end{align}
This together with the regularity of the drift function $\mu(\cdot)$ (as assumed) and of the second order function $D(\cdot)$ lead to $G(\cdot , \cdot) \in C^{d+1}(\bigtriangleup , \mathbb{R})$. The rest of the  argument is similar to the proof of Proposition \ref{prop:system2}. For $s\geq0$, define the stopped process:
$$
Z(t) = X(t) \mathbb{I}\{t \leq s\} + X(s) \mathbb{I}\{s<t\},\;\;\;\; t\geq 0.
$$
and the corresponding coupled process
\begin{align}\label{eq:X&Z:GBM}
\left(
\begin{array}{c}
   X(t) \\
   Z(t)
\end{array}
\right) &\:=& \left(
\begin{array}{c}
   X(0) \\
   X(0)
\end{array}
\right) + \int_0^t \left(
\begin{array}{c}
   \mu(u) X(u)  \\
   \mu(u) X(u) \mathbb{I}\{u \leq s\}
\end{array}
\right)du + \int_0^t \left(
\begin{array}{c}
   \sigma(u) X(u) \\
   \sigma(u)  X(u)\mathbb{I}\{u \leq s\}
\end{array}
\right)dB(u),
\end{align}
which gives a well-defined and unique (up to indistinguishablility) It{\^o} diffusion process. 
Applying It{\^o}'s integration by parts formula (see \citet{Oskendal03}) %to \eqref{eq:X&Z:CIR}, 
gives the following representation for $\mathbb{E}\left( Z(t) X(t) \right)$ ($= \mathbb{E}\left( X(s) X(t) \right)$) with $s<t$:
\begin{eqnarray*}
\mathbb{E}\left( Z(t) X(t) \right) &=& \mathbb{E}\left( X(0) X(0) \right)  + \mathbb{E} \int_0^t X(u) dZ(u)  + \mathbb{E} \int_0^t Z(u) dX(u)  + \mathbb{E} \int_0^t dX(u) dZ(u).
\end{eqnarray*}
By definition of the stopped process $\{Z(t)\}$, this implies:
\begin{eqnarray*}
\mathbb{E} \left( X(s) X(t) \right) &=& 
\mathbb{E}\left( X(0) X(0) \right)  + 2\mathbb{E} \int_0^s X^2(u) \mu \left(u\right) du  + \mathbb{E} \int_s^t X(s) X(u) \mu\left(u \right) du  + \mathbb{E} \int_0^s \sigma^2 \left(u \right) X^2(u) du.
\end{eqnarray*}
By the linear growth condition, the continuity of mean and mean squared functions, and the Fubini-Tonelli Theorem, we can interchange integral and expectation to obtain:
\begin{eqnarray*}
%\label{EX_sX_t}
\mathbb{E}\left( X(s) X(t) \right)&=& \mathbb{E}\left( X(0) X(0) \right)+ 2\int_0^s \mathbb{E} \left( X^2(u) \right) \mu \left(u\right)du  + \int_s^t \mathbb{E} \left( X(s) X(u) \right) \mu(u) du+ \int_0^s  \sigma^2 \left(u \right) \mathbb{E}( X^2(u)) du, 
\end{eqnarray*}
that is \begin{eqnarray*}
G\left(s,t \right)&=& G\left(0,0 \right)+ 2\int_0^s G \left( u, u \right) \mu \left(u\right)du  + \int_s^t G \left( s,u \right) \mu(u) du+ \int_0^s  \sigma^2 \left(u \right) G(u,u) du, 
\end{eqnarray*}
as claimed in \eqref{eq:G_expansion:GBM}. 
Taking partial derivatives with respect to $s$  of equation above we obtain:
\begin{eqnarray*}
\partial_s G(s,t) &=&
 2 \mu(s) G(s,s)  - \mu(s) G(s,s) + \int_s^t \mu \left(u \right) 
\partial_s G(s,u) du + \sigma^2(s) G(s,s),
\end{eqnarray*}
which implies \eqref{eq: G decomp:GBM} and the second equality in \eqref{system2:GBM}. 
\end{proof}
%%%%%%%%%
\begin{proof}[Proof of Lemma \ref{lemma:sup:GBM}]
In order to prove finiteness of  $\underset{0\leq t \leq 1}{\mathrm{sup}} \mathbb{E} \vert X(t)\vert^{\rho} $ as a result of finiteness of $ \mathbb{E} \vert X(0)\vert^{\rho}$, we take advantage  of representation \eqref{closed form:GBM} to obtain
\begin{eqnarray}
\mathbb{E}\left(X^{\rho}(t)\right)
&=& \exp \left\{ \rho\int_{0}^{t}\left( \mu(v) -\frac{1}{2}\sigma^2(v) \right)dv\right\} \mathbb{E}\left( X^{\rho}(0)\right)  \mathbb{E}\left(\exp \left\{  \rho \int_{0}^{t} \sigma(v) dB(v)\right\}\right) ,
\end{eqnarray}
where the product of the expectations on the right hand side  comes from the fact that the initial distribution is independent of the filtration generated by the Brownian motion $\{B(t)\}_{t\geq 0}$. Furthermore, the random variable $   \int_{0}^{t} \sigma(v) dB(v)$  is identical to the random variable $B(\int_{0}^{t} \sigma^2(v) dv)$ i.e Brownian motion at time $\int_{0}^{t} \sigma^2(v) dv$. This leads to 
\begin{eqnarray}
\mathbb{E}\left(X^{\rho}(t)\right)
&=& \exp \left\{ \rho\int_{0}^{t}\left( \mu(v) -\frac{1}{2}\sigma^2(v) \right)dv\right\} \mathbb{E}\left( X^{\rho}(0)\right)  \exp \left\{  \frac{1}{2}\rho^2 \int_{0}^{t} \sigma^2(v) dv\right\} ,
\end{eqnarray}
and hence 
\begin{eqnarray}
\underset{0 \leq t \leq 1}{\sup }\mathbb{E}\left(X^{\rho}(t)\right)
&\leq & \exp \left\{ \rho\int_{0}^{1}\left(\vert \mu(v) \vert + \frac{1}{2}\sigma^2(v) \right)dv +   \frac{1}{2}\rho^2 \int_{0}^{1} \sigma^2(v) dv\right\} \mathbb{E}\left( X^{\rho}(0)\right).
\end{eqnarray}
This completes the proof of the claim: finiteness of  $\underset{0\leq t \leq 1}{\mathrm{sup}} \mathbb{E} \vert X(t)\vert^{\rho} $ is  a result of finiteness of $ \mathbb{E} \vert X(0)\vert^{\rho}$.
\end{proof}

\begin{proof}[Proof of Proposition \ref{prop:system2:CIR}]
 Multiplying each side of  the time-inhomogeneous Cox--Ingersoll--Ross model
\begin{eqnarray}\label{eq:CIR}
dX(t) = \mu(t)X(t)dt + \sigma(t)X^{1/2}(t)dB(t), \qquad  t \in [0,1],
\end{eqnarray}
with $\exp \{\Lambda(t)\} := \exp\{-\int_0^t\mu(s)ds\}$ we obtain
\begin{eqnarray}\nonumber
\exp \{\Lambda(t)\}\left (dX(t) - \mu(t)X(t)dt\right) = \exp \{\Lambda(t)\}\sigma(t)X^{1/2}(t)dB(t), \qquad  t \in [0,1].
\end{eqnarray}
Consequently,
\begin{eqnarray}\nonumber
d\left (\exp \{\Lambda(t)\}X(t)\right) = \exp \{\Lambda(t)\}\sigma(t)X^{1/2}(t)dB(t), \qquad  t \in [0,1],
\end{eqnarray}
or equivalently
\begin{eqnarray}\nonumber
\exp \{\Lambda(t)\}X(t) = \exp \{\Lambda(0)\}X(0) + \int_0^t\exp \{\Lambda(s)\}\sigma(s)X^{1/2}(s)dB(s), \qquad  t \in [0,1].
\end{eqnarray}
This gives the closed form representation 
\begin{eqnarray}\label{eq:closed:CIR}
X(t) = \exp \{-\Lambda(t)\}X(0) + \exp \{-\Lambda(t)\}\int_0^t\exp \{\Lambda(s)\}\sigma(s)X^{1/2}(s)dB(s), \qquad  t \in [0,1].
\end{eqnarray}
Taking expectations yields:
\begin{eqnarray}
\nonumber
m(t) &=& \exp \{-\Lambda(t)\}m(0)\\ \label{eq:mean:CIR}
&=& \exp\left\{\int_0^t\mu(s)ds\right\} m(0), \qquad  t \in [0,1].
\end{eqnarray}
This justifies that $m(\cdot) \in \mathcal{C}^{d+1}([0,1],\mathbb{R})$ as well as the first equation appearing in  \eqref{system2:CIR}.

By the closed form equation \eqref{eq:closed:CIR}  we see that $G(\cdot,\cdot) \in C^{d+1}(\bigtriangleup, \mathbb{R})$. Indeed, for $s<t$ 
\begin{align}
\nonumber
G(s,t) = & \mathbb{E}(X(t) X(s)) = \exp \{\Lambda(t)\Lambda(s)\}\mathbb{E}\left(X^2(0)\right) +
\exp \{\Lambda(t)\Lambda(s)\} \mathbb{E} \left(\int_0^s\exp \{\Lambda(u)\}\sigma(u)X^{1/2}(u)dB(u)\right)^2\\ \nonumber
& + \exp \{\Lambda(t)\Lambda(s)\} \mathbb{E} \left(\int_0^s\exp \{\Lambda(u)\}\sigma(u)X^{1/2}(u)dB(u) \int_s^t\exp \{\Lambda(u)\}\sigma(u)X^{1/2}(u)dB(u)\right)\\ \nonumber
= & \exp \{\Lambda(t)\Lambda(s)\} D(0) + \exp \{\Lambda(t)\Lambda(s)\} \int_0^s\exp \{2\Lambda(u)\}\sigma^2(u)m(u) du + 0,
\end{align}
where the last equation follows from It{\^o}'s isometry. Smoothness of drift and diffusion functions imply the desired regularity of $G(\cdot,\cdot)$. 
The rest of the  argument is similar to the proof of Proposition \ref{prop:system2}. For $s\geq0$, define the stopped process:
$$
Z(t) = X(t) \mathbb{I}\{t \leq s\} + X(s) \mathbb{I}\{s<t\},\;\;\;\; t\geq 0.
$$
and the corresponding coupled process
\begin{align}\label{eq:X&Z:CIR}
\left(
\begin{array}{c}
   X(t) \\
   Z(t)
\end{array}
\right) &=& \left(
\begin{array}{c}
   X(0) \\
   X(0)
\end{array}
\right) + \int_0^t \left(
\begin{array}{c}
   \mu(u) X(u)  \\
   \mu(u) X(u) \mathbb{I}\{u \leq s\}
\end{array}
\right)du + \int_0^t \left(
\begin{array}{c}
   \sigma(u) X^{1/2}(u) \\
   \sigma(u)  X^{1/2}(u) \mathbb{I}\{u \leq s\}
\end{array}
\right)dB(u),
\end{align}
which gives a well-defined and unique (up to indistinguishablility) It{\^o} diffusion process. 
Here again, we apply It{\^o}'s integration by parts formula (see \citet{Oskendal03}) %to \eqref{eq:X&Z:CIR}, 
to obtain the following representation for $\mathbb{E}\left( Z(t) X(t) \right)$ ($= \mathbb{E}\left( X(s) X(t) \right)$) with $s<t$:
\begin{eqnarray*}
\mathbb{E}\left( Z(t) X(t) \right) &=& \mathbb{E}\left( X(0) X(0) \right)  + \mathbb{E} \int_0^t X(u) dZ(u)  + \mathbb{E} \int_0^t Z(u) dX(u)  + \mathbb{E} \int_0^t dX(u) dZ(u).
\end{eqnarray*}
The definition of the stopped process $\{Z(t)\}$ implies that:
\begin{eqnarray*}
\mathbb{E} \left( X(s) X(t) \right) &=& 
\mathbb{E}\left( X(0) X(0) \right)  + 2\mathbb{E} \int_0^s X^2(u) \mu \left(u\right) du  + \mathbb{E} \int_s^t X(s) X(u) \mu\left(u \right) du  + \mathbb{E} \int_0^s \sigma^2 \left(u \right) X(u) du.
\end{eqnarray*}
By the linear growth condition, the continuity of mean and mean squared functions, and the Fubini-Tonelli Theorem, we can interchange integral and expectation to obtain:
\begin{eqnarray*}
%\label{EX_sX_t}
\mathbb{E}\left( X(s) X(t) \right)&=& \mathbb{E}\left( X(0) X(0) \right)+ 2\int_0^s \mathbb{E} \left( X^2(u) \right) \mu \left(u\right)du  + \int_s^t \mathbb{E} \left( X(s) X(u) \right) \mu(u) du+ \int_0^s  \sigma^2 \left(u \right) \mathbb{E}( X(u)) du, 
\end{eqnarray*}
that is \begin{eqnarray*}
G\left(s,t \right)&=& G\left(0,0 \right)+ 2\int_0^s D \left( u \right) \mu \left(u\right)du  + \int_s^t G \left( s,u \right) \mu(u) du+ \int_0^s  \sigma^2 \left(u \right) m(u) du, 
\end{eqnarray*}
as claimed in \eqref{eq:G_expansion:CIR}. 
Taking partial derivatives with respect to $s$  of equation above we obtain:
\begin{eqnarray*}
\partial_s G(s,t) &=&
 2 \mu(s) G(s,s)  - \mu(s) G(s,s) + \int_s^t \mu \left(u \right) 
\partial_s G(s,u) du + \sigma^2(s) m(s),
\end{eqnarray*}
as desired in equation \eqref{eq: G decomp:CIR} and the second equality in \eqref{system2:CIR}.
\end{proof}
\begin{proof}[Proof of Lemma \ref{lemma:sup:CIR}]
We prove  the assertion by an inductive argument on $\rho$. Step 1 below justifies the result for $\rho = 2$ i.e. $\underset{0 \leq t \leq 1}{\sup} D(t)$ is finite if and only if $D(0)$ is finite. Step 2  assumes the conclusion of Lemma \ref{lemma:sup:CIR} to be true  for $\rho $ and proves the conclusion for  $2\rho $. In other words, Step 2 assumes   $\underset{0 \leq t \leq 1}{\sup} \mathbb{E}\vert X(t) \vert^{\rho}$  is finite if and only if $\mathbb{E}\vert X(0) \vert^{\rho}$ is finite and then obtains the equivalency of  finiteness of $\underset{0 \leq t \leq 1}{\sup} \mathbb{E}\vert X(t) \vert^{2\rho}$ and $\mathbb{E}\vert X(0) \vert^{2\rho}$.
The representation \eqref{eq:closed:CIR} and an application of It{\^o}'s isometry give
\begin{eqnarray}
 \label{eq:D:CIR}
D(t) &=& \exp \{-2\Lambda(t)\}D(0) + \exp \{-2\Lambda(t)\}  \int_0^t\exp \{2\Lambda(s)\}\sigma^2(s)m(s)ds \\ \nonumber 
&=& \exp \left\{2\int_0^t\mu(s)ds\right\} D(0) + \exp \left\{2\int_0^t\mu(s)ds\right\}  \int_0^t\exp \left\{-2\int_0^s\mu(v)dv\right\}\sigma^2(s)m(s)ds.
\end{eqnarray}
 \\
\textbf{Step 1:} Set $\rho = 2$, equation \eqref{eq:D:CIR} leads to 
\begin{eqnarray}
\underset{0 \leq t \leq 1}{\sup}D(t) &\leq& D(0)  \underset{0 \leq t \leq 1}{\sup}\exp \{-2\Lambda(t)\}+ \underset{0 \leq t \leq 1}{\sup}\exp \{-2\Lambda(t)\}  \int_0^t\exp \{2\Lambda(s)\}\sigma^2(s)m(s)ds.
\end{eqnarray}
and smoothness of  the functions appearing on the right hand side  delivers the result  for $\rho = 2$.
\\
\textbf{Step 2:} Assume that $\underset{0 \leq t \leq 1}{\sup} \mathbb{E}\vert X(t) \vert^{\rho}$  is finite if and only if $\mathbb{E}\vert X(0) \vert^{\rho}$ is finite. The closed form equation \eqref{eq:closed:CIR} together with Minkowski inequality leads to 
\begin{align}\label{eq:step2}
\vert X(t) \vert^{2\rho} &\leq  2^{2\rho -1} \exp \{-2\rho \Lambda(t)\} \vert X(0) \vert^{2\rho} + 2^{2\rho -1} \exp \{-2\rho\Lambda(t)\} \left(\int_0^t\exp \{\Lambda(s)\}\sigma(s)X^{1/2}(s)dB(s)\right)^{2\rho}, \end{align}
and hence
\begin{align}\label{eq:step2:sup}
\underset{0 \leq t \leq 1}{\sup} \mathbb{E}\vert X(t) \vert^{2\rho} &\leq  2^{2\rho -1} \underset{0 \leq t \leq 1}{\sup}  \exp \{-2\rho \Lambda(t)\} \mathbb{E}\vert X(0) \vert^{2\rho} + 2^{2\rho -1} \underset{0 \leq t \leq 1}{\sup}  \exp \{-2\rho\Lambda(t)\} \mathbb{E} \left(\int_0^t\exp \{\Lambda(s)\}\sigma(s)X^{1/2}(s)dB(s)\right)^{2\rho}
\end{align}
Regarding the second summand on the right hand side of \eqref{eq:step2:sup} note that:
\begin{eqnarray}
\nonumber
\underset{0 \leq t \leq 1}{\sup} \mathbb{E}\left(\int_0^t\exp \{\Lambda(s)\}\sigma(s)X^{1/2}(s)dB(s)\right)^{2\rho} &\leq &  \mathbb{E}\underset{0 \leq t \leq 1}{\sup} \left(\int_0^t\exp \{\Lambda(s)\}\sigma(s)X^{1/2}(s)dB(s)\right)^{2\rho}\\ \nonumber
&\leq & c_1 \mathbb{E} \langle \int_0^{\cdot} \exp \{\Lambda(s)\}\sigma(s)X^{1/2}(s)dB(s) \rangle_1^{\rho}\\ \nonumber
&\leq &  c_1 \mathbb{E}
\left(\int_0^1\exp \{2\Lambda(s)\}\sigma^2(s)X(s)ds\right)^{\rho} \\ \label{ineq:jensen }
&\leq &  c_1 \mathbb{E}
\int_0^1\left(\exp \{2\Lambda(s)\}\sigma^2(s)X(s)\right)^{\rho}ds 
\\ \label{eq:rho}
& = & c_1 
\int_0^1 \exp \{2\rho\Lambda(s)\}\sigma^{2\rho}(s) \mathbb{E} \left(X(s)\right)^{\rho}ds.
\end{eqnarray}
Inequality \eqref{ineq:jensen } is via by Jensen's inequality. If we assume $\mathbb{E}\vert X(0) \vert^{2\rho}< \infty$, then clearly $\mathbb{E}\vert X(0) \vert^{\rho}< \infty$. This in turn, by the inductive assumption, implies that $\underset{0 \leq t \leq 1}{\sup} \mathbb{E}\vert X(t) \vert^{\rho}< \infty$ and hence finiteness of \eqref{eq:rho}. So, finiteness of the left hand side of  inequality \eqref{eq:step2:sup} is guaranteed by  $\mathbb{E}\vert X(0) \vert^{2\rho}< \infty$. The reverse conclusion is clear: $\underset{0 \leq t \leq 1}{\sup} \mathbb{E}\vert X(t) \vert^{2\rho}< \infty$ implies $\mathbb{E}\vert X(0) \vert^{2\rho}< \infty$.
\end{proof}

\begin{proof}[Proof of Proposition \ref{prop:system2:alpha = 0}]
Recall equation \eqref{eq:timedep:alpha = 0}  or its equivalent integral form
\begin{align}%\label{eq:int:timedep:alpha = 0}
    X(t) =& X(0) + \int_0^t \mu(u) du + \int_0^t \sigma(u) X^{\beta}(u)dB(u), \quad 0 \leq t \leq 1,\;\beta \in \{0,1/2,1 \}.
\end{align}
Taking expectations on both sides gives 
\begin{align} \nonumber
    m(t) =& m(0) + \int_0^t \mu(u) du , \quad 0 \leq t \leq 1,\;\beta \in \{0,1/2,1 \}.
\end{align}
This confirms the regularity property $m(\cdot) \in C^{d+1}([0,1], \mathbb{R})$ as well as the first equation appearing in \eqref{system2:alpha = 0}.
Using  \eqref{eq:timedep:alpha = 0}  we first conclude $G(\cdot,\cdot) \in C^{d+1}(\bigtriangleup, \mathbb{R})$. Indeed, for $s<t$ 
\begin{align}
\nonumber
G(s,t) = & \mathbb{E}(X(t) X(s)) = \mathbb{E}(X^2(s) ) +  \mathbb{E}(X(s)) \int_s^t \mu(u)du = D(s) + m(s) \int_s^t \mu(u)du.
\end{align}
 Smoothness of the drift function imply the desired smoothness property of $G(\cdot,\cdot)$. 
The rest of the  argument is similar to the proof of Proposition \ref{prop:system2}. For $s\geq0$, define the stopped process:
$$
Z(t) = X(t) \mathbb{I}\{t \leq s\} + X(s) \mathbb{I}\{s<t\},\;\;\;\; t\geq 0.
$$
and the corresponding coupled process
\begin{align}\label{eq:X&Z:alpha = 0}
\left(
\begin{array}{c}
   X(t) \\
   Z(t)
\end{array}
\right) &=& \left(
\begin{array}{c}
   X(0) \\
   X(0)
\end{array}
\right) + \int_0^t \left(
\begin{array}{c}
   \mu(u)  \\
   \mu(u) \mathbb{I}\{u \leq s\}
\end{array}
\right)du + \int_0^t \left(
\begin{array}{c}
   \sigma(u) X^{\beta}(u) \\
   \sigma(u)  X^{\beta}(u) \mathbb{I}\{u \leq s\}
\end{array}
\right)dB(u).
\end{align}
This defines a well-defined and unique (up to indistinguishablility) It{\^o} diffusion process. 
 It{\^o}'s integration by parts formula (see \citet{Oskendal03}) %to \eqref{eq:X&Z:CIR}, 
 for $\mathbb{E}\left( Z(t) X(t) \right)$ ($= \mathbb{E}\left( X(s) X(t) \right)$) with $s<t$ gives
\begin{eqnarray*}
\mathbb{E}\left( Z(t) X(t) \right) &=& \mathbb{E}\left( X(0) X(0) \right)  + \mathbb{E} \int_0^t X(u) dZ(u)  + \mathbb{E} \int_0^t Z(u) dX(u)  + \mathbb{E} \int_0^t dX(u) dZ(u).
\end{eqnarray*}
The definition of the stopped process $\{Z(t)\}$ implies:
\begin{eqnarray*}
\mathbb{E} \left( X(s) X(t) \right) &=& 
\mathbb{E}\left( X(0) X(0) \right)  + 2\mathbb{E} \int_0^s X(u) \mu \left(u\right) du  + \mathbb{E} \int_s^t X(s)  \mu\left(u \right) du  + \mathbb{E} \int_0^s \sigma^2 \left(u \right) X^{2\beta}(u) du.
\end{eqnarray*}
By the linear growth condition, the continuity of mean and mean squared functions, and the Fubini-Tonelli Theorem, we can interchange integral and expectation to obtain:
\begin{eqnarray*}
%\label{EX_sX_t}
\mathbb{E}\left( X(s) X(t) \right)&=& \mathbb{E}\left( X(0) X(0) \right)+ 2\int_0^s \mathbb{E} \left( X(u) \right) \mu \left(u\right)du  + \mathbb{E} \left( X(s)  \right) \int_s^t \mu(u) du+ \int_0^s  \sigma^2 \left(u \right) \mathbb{E}( X^{2\beta}(u)) du, 
\end{eqnarray*}
that is 
\begin{eqnarray*}
G\left(s,t \right)&=& G\left(0,0 \right)+ 2\int_0^s m \left( u \right) \mu \left(u\right)du  + m \left( s \right) \int_s^t  \mu(u) du+ \int_0^s  \sigma^2 \left(u \right)\mathbb{E}( X^{2\beta}(u)) du, 
\end{eqnarray*}
as claimed in \eqref{eq:G_expansion:alpha = 0}. 
Taking partial derivatives with respect to $s$  of equation above we obtain:
\begin{eqnarray*}
\partial_s G(s,t) &=&
  m(s) \mu(s)  + \partial m(s) \int_s^t  \mu(u) du   + \sigma^2(s) \mathbb{E}( X^{2\beta}(s)) ,
\end{eqnarray*}
showing \eqref{eq: G decomp:alpha = 0} and the second equality in \eqref{system2:alpha = 0}. 
\end{proof}

\bibliography{bib}
\end{document}